\documentclass[10pt,twoside,a4paper]{article}

\usepackage[utf8]{inputenc}

\usepackage[T1]{fontenc}

\usepackage{amsmath, bm}
\numberwithin{equation}{section}

\usepackage{amssymb}

\usepackage{microtype}
\usepackage{lmodern}

\usepackage{mathtools}
\usepackage[pagestyles]{titlesec}
\usepackage{placeins}
\usepackage{float}
\usepackage[backend=biber,maxnames=10,backref=true,hyperref=true,safeinputenc]{biblatex}

\DefineBibliographyStrings{english}{%
	backrefpage = {cited on page},
	backrefpages = {cited on pages},
}

\DeclareFieldFormat[report]{title}{``#1''}
\DeclareFieldFormat[book]{title}{``#1''}
\AtEveryBibitem{\clearfield{url}}
\AtEveryBibitem{\clearfield{note}}

\usepackage{graphicx}
\graphicspath{{images/}}

\usepackage{dsfont,bbm}
\usepackage[boxed]{algorithm2e}
\SetKwInput{KwInput}{Input}
\SetKwInput{KwOutput}{Output}

\usepackage{enumerate}
\usepackage{caption}
\usepackage{subcaption}

\usepackage{multirow}

\usepackage{graphicx,tikz,pgfplots}
\usepackage{tikz-3dplot}
\usetikzlibrary{arrows,shapes,fadings,intersections,decorations.pathreplacing,positioning,patterns}
\graphicspath{{images/}}
\pgfplotsset{compat=newest}

\usetikzlibrary{quotes}
\usetikzlibrary{angles}
\usetikzlibrary{babel}
\usetikzlibrary{decorations.pathreplacing}
\usetikzlibrary{decorations.markings}
\usepackage{mathabx}
\tikzset{
	>=stealth',
	punkt/.style={
		rectangle,
		rounded corners,
		draw=black, very thick,
		text width=6.5em,
		minimum height=2em,
		text centered},
	pil/.style={
		->,
		thick,
		shorten <=2pt,
		shorten >=2pt,}
}
\tikzstyle{block} = [rectangle, rounded corners, minimum width= 3cm, minimum height=1cm, text centered, draw=black, fill=blue!20,]
\tikzstyle{invisbleblock} = [minimum width= 3em, minimum height=1cm, text centered, ]
\tikzstyle{decision} = [diamond, minimum width=3cm, minimum height=1cm, text centered, draw=black, fill=orange!20]
\tikzstyle{arrow} = [thick,->,>=stealth]
\tikzstyle{line} = [thick,-]

\title{Diffraction Tomography for a Generalized Incident Field}
\author{Clemens Kirisits$^{1}$\\
		{\footnotesize\href{mailto:email}{clemens.kirisits@univie.ac.at}}
		\and Noemi Naujoks$^{1,3}$\\
		{\footnotesize\href{mailto:email}{noemi.naujoks@univie.ac.at}}
		\and Otmar Scherzer$^{1,2,3}$\\	
		{\footnotesize\href{mailto:email}{otmar.scherzer@univie.ac.at}}
}

\date{\today}

\usepackage[pdftex,colorlinks=true,linkcolor=blue,citecolor=green!50!black,urlcolor=blue,bookmarks=true,bookmarksnumbered=true,pdfusetitle]{hyperref}
\hypersetup
{
	pdfauthor={C. Kirisits, N. Naujoks, O. Scherzer},
	    pdftitle={Diffraction Tomography for a Generalized Incident Field},
}

\usepackage[hyperref,amsmath,thmmarks]{ntheorem}
\usepackage{aliascnt}

\newtheorem{lemma}{Lemma}[section]

\newaliascnt{proposition}{lemma}

\aliascntresetthe{proposition}

\newaliascnt{corollary}{lemma}

\aliascntresetthe{corollary}

\newaliascnt{theorem}{lemma}
\newtheorem{theorem}[theorem]{Theorem}
\aliascntresetthe{theorem}

\theorembodyfont{\normalfont}
\newaliascnt{definition}{lemma}
\newtheorem{definition}[definition]{Definition}
\aliascntresetthe{definition}

\newaliascnt{assumption}{lemma}

\aliascntresetthe{assumption}

\newaliascnt{remark}{lemma}
\newtheorem{remark}[remark]{Remark}
\aliascntresetthe{remark}

\newaliascnt{example}{lemma}
\newtheorem{example}[example]{Example}
\aliascntresetthe{example}

\theoremstyle{nonumberplain}
\theoremseparator{:}
\theoremheaderfont{\normalfont\itshape}

\theoremsymbol{\ensuremath{\square}}
\newtheorem{proof}{Proof}

\usepackage[a4paper,centering,bindingoffset=0cm,marginpar=2cm,margin=2.5cm]{geometry}
\usepackage[pagestyles]{titlesec}

\titleformat{\section}[block]{\large\sc\filcenter}{\thesection.}{0.5ex}{}[]
\titleformat{\subsection}[runin]{\bf}{\thesubsection.}{0.5ex}{}[.]

\newpagestyle{headers}

{

	\headrule

	\sethead[\footnotesize\thepage][\footnotesize\sc  C.~Kirisits, N.~Naujoks, O.~Scherzer][]{}{\footnotesize\sc Diffraction Tomography for a Generalized Incident Field}{\footnotesize\thepage}

	\setfoot{}{}{}

}

\usepackage[font=footnotesize,format=plain,labelfont=sc,textfont=sl,width=0.75\textwidth,labelsep=period]{caption}
\usepackage[labelfont = normalfont]{subcaption}

\usepackage{dsfont}
\usepackage{bbm}
\newcommand{\N}{\mathds{N}}
\newcommand{\Z}{\mathds{Z}}
\newcommand{\R}{\mathds{R}}
\newcommand{\C}{\mathds{C}}
\newcommand{\usc}{u^{\text{sca}}}
\newcommand{\ut}{u^{\text{tot}}}

\newcommand{\ui}{u^{\text{inc}}}

\newcommand{\Sp}{\mathbb{S}^1}
\newcommand{\Card}{\operatorname{Card}}

\let\RE\Re
\let\Re=\undefined
\DeclareMathOperator{\Re}{\RE e}

\let\IM\Im
\let\Im=\undefined
\DeclareMathOperator{\Im}{\IM m}

\DeclareMathOperator*{\argmin}{arg\,min}

\newcommand{\abs}[1]{\left|#1\right|}
\newcommand{\norm}[1]{\left\|#1\right\|}

\usetikzlibrary{decorations.markings}
\tikzset{
	halfarrow/.style={postaction={decorate},
		decoration={markings,mark=at position .4 with
			{\arrow{>}}}}}

\usepackage[hyperref,amsmath,thmmarks]{ntheorem}

\usepackage{aliascnt}
\usepackage{color}
\usepackage{enumerate}
\usepackage{chngcntr}
\counterwithout{equation}{section}

\begin{document}
	
	\maketitle
	\thispagestyle{empty}
	\begin{center}
		\hspace*{5em}
		\parbox[t]{12em}{\footnotesize
			\hspace*{-1ex}$^1$Faculty of Mathematics\\
			University of Vienna\\
			Oskar-Morgenstern-Platz 1\\
			A-1090 Vienna, Austria}
		\hfil
		\parbox[t]{17em}{\footnotesize
			\hspace*{-1ex}$^2$Johann Radon Institute for Computational\\
			\hspace*{1em}and Applied Mathematics (RICAM)\\
			Altenbergerstraße 69\\
			A-4040 Linz, Austria}\\
		\vspace*{0.5cm}
		\hspace*{5em}
		\parbox[t]{18em}{\footnotesize
			\hspace*{-1ex}$^3$Christian Doppler Laboratory for Mathematical\\
			\hspace*{1em}Modelling and Simulation of Next Generation\\
			\hspace*{1em}Medical Ultrasound Devices (MaMSi)\\
			Oskar-Morgenstern-Platz 1\\
			A-1090 Vienna, Austria}
	\end{center}
	\begin{abstract}
		Diffraction tomography is an inverse scattering technique used to reconstruct the spatial distribution of the material properties of a weakly scattering object. The object is exposed to radiation, typically light or ultrasound, and the scattered waves induced from different incident field angles are recorded. In conventional diffraction tomography, the incident wave is assumed to be a monochromatic plane wave, an unrealistic simplification in practical imaging scenarios. In this article, we extend conventional diffraction tomography by introducing the concept of customized illumination scenarios, with a pronounced emphasis on imaging with focused beams. We present a new forward model that incorporates a generalized incident field and extends the classical Fourier diffraction theorem to the use of this incident field. This yields a new two-step reconstruction process which we comprehensively evaluate through numerical experiments. 
	\end{abstract}
	\section{Introduction}\label{introduction}
The task of reconstructing the internal structure of an object from scattered radiation has traditionally been known as the inverse scattering problem \cite{ColKre92}. In this context, an object is exposed to a specific form of radiation, the resulting scattered waves are recorded, and the scattering potential is reconstructed from these recordings. Diffraction tomography (DT) provides a solution to the inverse scattering problem.

It is typically based on either the \textit{Born} or the \textit{Rytov approximation}, which simplifies the relationship between scattered wave and scattering potential by assuming that the object is only weakly scattering \cite{KakSla01,MueSchuGuc15_report, Nat15, Wol69}. This simplification allows DT to offer explicit solutions to the inverse scattering problem, enabling fast implementation and making it attractive for real-world applications. Examples include \emph{optical diffraction tomography} for visualizing living biological cells \cite{SunChoFanBadDas09}  via tomographic  phase microscopy, and medical ultrasound imaging, known as \emph{ultrasound computed tomography}, particularly used for breast imaging \cite{Nat15,SimHuaDur08, SimHuaDurLit09}. 

E. Wolf laid the mathematical foundation for DT in \cite{Wol69}. His central result is the \textit{Fourier diffraction theorem}. In a two-dimensional context, it states that the Fourier-transformed measurements of the scattered wave are equal to the 2D Fourier transform of the scattering potential on a semicircle. Incorporating a collection of measurements taken at different incident field angles generates more data in Fourier space, enabling the scattering potential to be extracted through a technique known as  \textit{filtered backpropagation}, see \cite{Dev82, KakSla01, KirQueRitSchSet21, Wol69}. 

The applicability of the Fourier diffraction theorem, and thus the reconstruction process in conventional DT, relies on the assumption that the incident wave is a monochromatic plane wave. However, this assumption is violated in various imaging applications. For instance, in ultrasound imaging the width of the transducer limits the emitted sound wave to a beam.  More significantly, new devices  provide high flexibility in controlling the shape of the wave sent into the body, which opens up the possibility to experiment with new beam designs. In this way, ultrasound devices can, for instance, emit focused beams to a region of interest in the human body to achieve a better resolution in the far field \cite{Kut91}.

Numerous efforts have been undertaken to extend the foundational concept of DT to alternative illumination scenarios. For instance, A. Devaney and G. Beylkin presented in \cite{Dev84b} how the scattered field in ultrasound imaging setups with arbitrarily configured source and receiver surfaces can be related to the conventional plane-wave forward problem. In their research, they present reconstruction algorithms for both plane wave and cylindrical wave illumination. Furthermore, Synthetic Aperture Diffraction Tomography (SADT), initially introduced by \cite{NahPanKak84}, provides reconstruction from scattered waves that are acquired by using two parallel single-element transducers (i.e. point sources) moving along a straight line. This methodology does not afford the capability to rotate the incident illumination around the object. Subsequently, SADT was further developed in \cite{SimHua09}, specifically within the context of breast ultrasound tomography. This approach includes a toroidal array transducer that encircles the breast and allows for simultaneously transmitting plane waves from different directions and receiving the resulting scattered waves.

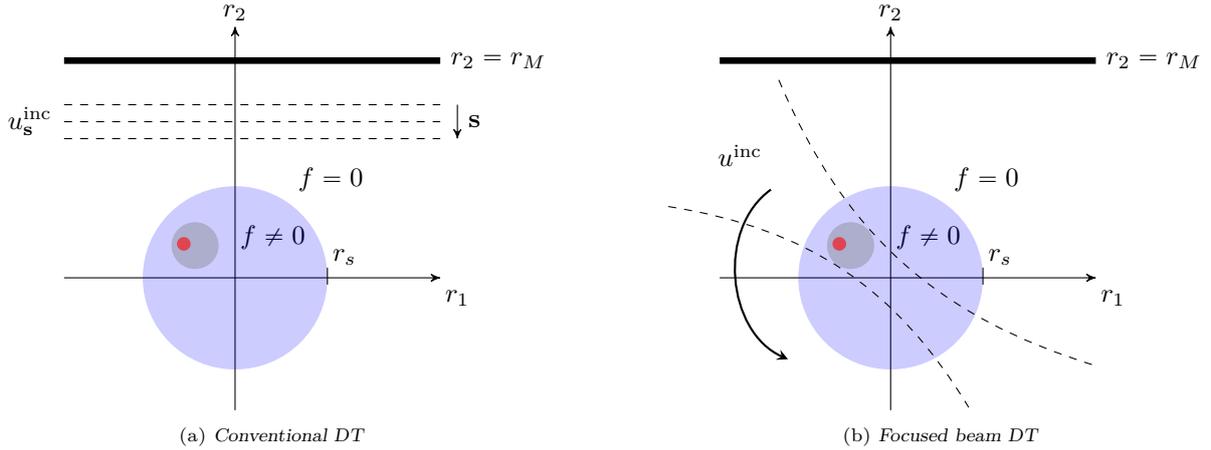
\begin{figure}
	\begin{subfigure}[t]{0.45\textwidth}
		\begin{tikzpicture}[scale=0.45,
			>=stealth',
			pos=.8,
			photon/.style={decorate,decoration={snake,post length=1mm}}
			]
			
			\draw[->] (-3,-0.6)--(8,-0.6)  ;
			\draw[->] (2,-4.5)--(2,6.8)  ;
			\node at (8.5,-1.2) {$r_1$};
			\node at (2,7.2) {$r_2$};
			\node at (4.8,2.3) {$f = 0$};
			\node at (3.1,0.6) {$f \neq 0$};
			
			\fill[blue,opacity=0.2] (2,-0.6) circle (2.7);
			\draw (4.7,-0.8) -- (4.7,-0.3);
			\node at (5.2,0) {$r_s$};
			\fill[gray,opacity=0.4] (0.83,0.35) circle (0.69);
			\fill [red, opacity=0.6]  (0.5,0.4) circle (0.2);
			
			\draw[dashed] (-3,3.5) -- (8,3.5);
			\draw[dashed] (-3,4) -- (8,4);
			\draw[dashed] (-3,4.5) -- (8,4.5);
			\draw[->] (8.5,4.5) -- (8.5,3.5);
			\node at (9,4) {$\mathbf s$};
			
			\node at (-4,4) {$\ui_{\mathbf{s}}$};

			\draw[line width = 2.5pt] (-3,5.8) -- (8,5.8);
			\node at (9.7,5.8) {$r_2 = r_M$};
		\end{tikzpicture}
		\caption{Conventional DT}\label{subfig:convDT}
	\end{subfigure}\hfill
	\begin{subfigure}[t]{0.45\textwidth}
		\begin{tikzpicture}[scale=0.45,
			>=stealth',
			pos=.8,
			photon/.style={decorate,decoration={snake,post length=1mm}}
			]
			
			\draw[->] (-3,-0.6)--(8,-0.6)  ;
			\draw[->] (2,-4.5)--(2,6.8)  ;
			\node at (8.5,-1.2) {$r_1$};
			\node at (2,7.2) {$r_2$};
			\node at (4.8,2.3) {$f = 0$};
			\node at (3.1,0.6) {$f \neq 0$};
			
			\fill[blue,opacity=0.2] (2,-0.6) circle (2.7);
			\draw (4.7,-0.8) -- (4.7,-0.3);
			\node at (5.2,0) {$r_s$};
			\fill[gray,opacity=0.4] (0.83,0.35) circle (0.69);
			\fill [red, opacity=0.6]  (0.5,0.4) circle (0.2);

			\coordinate (O) at (-0.5,6,2);
			\coordinate (A) at (10.2,-1,5.7);
			\coordinate (P) at (6.4,-2.3,5.5);
			\coordinate (Q) at (-4.5,1.5,0);
			
			\draw[color=black, dashed ] (O) to [bend right=25] (A);
			
			\draw[color=black, dashed ] (Q) to [bend left=25] (P);
			\node at (-2.4,3) {$\ui$};

			\draw [arrow, bend angle=60, bend right] (-1.5,2) to (-1,-3);
			
			\draw[line width = 2.5pt] (-3,5.8) -- (8,5.8);
			\node at (9.7,5.8) {$r_2 = r_M$};
		\end{tikzpicture}
		\caption{Focused beam DT} \label{subfig:focusedDT}
	\end{subfigure}
	\caption{ \textbf{Comparison of experimental setups.} In both cases the scattering potential $f$ is supported in $\mathcal{B}_{r_s}$ while measurements are taken at $r_2=  r_M>r_s$. (a) Conventional DT setup: Incident field $\ui_{\mathbf{s}}$ is a plane wave propagating in direction $\mathbf{s}$. Additional data are typically generated by varying $\mathbf{s}$ or rotating the object. (b) DT setup considered here: Generalized incident field $\ui$ of the form \eqref{eqn: Inc}, e.g.\ a focused beam, making a full $360^\circ$ rotation during data acquisition.}	\label{fig:concept}
\end{figure}
\subsection*{Contributions} Motivated by applications in ultrasound tomography, this article is devoted to extending DT to accommodate arbitrary incident field designs. We place particular emphasis on focused beams, which are commonly used in various tomographic contexts. A comparison of the experimental setup of conventional DT with the new beam DT approach is visualized in \autoref{fig:concept}. 

The fundamental concept is to model a generalized incident field via superposition of monochromatic plane waves, each characterized by different propagation direction and amplitude  \cite{BorWol99, Goo05}. We present a customized Fourier diffraction theorem, serving as the basis for reconstructing the scattering potential from a set of detected scattered waves. The measurements are obtained by rotating the beam around the object and capturing the resulting scattered waves along a fixed measurement line in front of the object. While the original Fourier diffraction theorem provides a \emph{direct} relation between the measurements and the scattering potential in Fourier space, this modified version involves an integral operator acting on the Fourier-transformed scattering potential. As a result, it establishes an \emph{indirect} relation, requiring the reconstruction of the scattering potential to be divided into two separated inversion steps: First, we invert the arising integral operator to extract the Fourier data of the scattering potential using truncated singular value decomposition. Subsequently, the scattering potential is reconstructed via Fourier inversion in the second step. 

In addition to the theoretical findings presented in this article, we perform numerical experiments to assess our proposed method. 
\subsection*{Outline} 
The article is organized as follows: We start with presenting the mathematical concept for the beam DT approach in \autoref{sec:forward}. It contains the formulation of a generalized incident field and the introduction of the underlying wave propagation model based on the Born approximation. Subsequently, we delve into the analysis of this forward model in the Fourier domain by deriving an extension of the Fourier diffraction theorem. With this fundament, we turn our attention to the inverse problem in \autoref{sec:invprob}, where we introduce a two-step technique for extracting the scattering potential from the measurements. Thereafter, we discuss in \autoref{sec:numerics} the numerical implementation of our developed method that is subsequently validated by means of numerical examples in \autoref{sec:experiments}. This article ends with a conclusion of our results in \autoref{sec:conclusion}.
	
\section{Forward Model}\label{sec:forward}
In this section, we provide an overview of the mathematical model that serves as the foundation for our tomographic reconstruction.  The concrete model under investigation follows the standard assumptions of DT by assuming that the acquired measurement data contain information about both the amplitude and phase of the scattered waves.
Furthermore, we operate under the assumption of scalar wave propagation. Hence, the wave propagation models are the same, regardless the motivation stems from acoustic or electromagnetic imaging modalities, making our results adaptable and applicable in both imaging techniques.
\subsection{Conceptual setup}
Although we are interested in realistic imaging applications, i.e. reconstructing three-dimensional objects, this article deals with an experimental setup in $\R^2$. This limitation is justified by the fact that the two-dimensional theory directly can be applied to three-dimensional objects that vary only slightly along one of the dimensions. Furthermore, images are often built using single slices of the object. 

The object we intend to quantify tomographically is assumed to be enclosed by the disk $\mathcal{B}_{r_s}(\mathbf{0})\coloneqq \{ \mathbf{r}\in\R^2:\norm{\mathbf{r}} < r_s \}$ centered at $\mathbf{0}$ with radius $r_s$. It is described by  the local refractive index $n : \R^2 \to \C$ with $n(\mathbf{r}) = 1$ on the complement of $\mathcal{B}_{r_s}(\mathbf{0})$. 

Further, let $\lambda$ denote the wavelength and $k_0 =\frac{ 2\pi}{\lambda}$ the wave number. The function
\begin{align}\label{eq: scattering_pot}
	f(\mathbf{r}) = k_0^2\left[n(\mathbf{r})^2-1\right],\qquad \mathbf r\in\R^2
\end{align}
defines the \textit{scattering potential} of the object \cite{Wol69} and is the quantity we aim to reconstruct. Note that $f$ differs only in $\mathcal{B}_{r_s}$ from $\mathbf{0}$ such that we observe
\begin{align*}
	\operatorname{supp}(f)\subseteq \mathcal{B}_{r_s}\subset(-r_s,r_s)^2
\end{align*}
by construction. An incident wave $\ui$ that propagates through the object induces a scattered wave $\usc$.
In our experiment, $\ui$ is rotated around the object, while the  scattered wave is recorded on a fixed measurement line $\{\mathbf{r}\in\R^2:r_2=r_M\}$, with $r_M>r_s$. See \autoref{subfig:focusedDT} for an illustration of the experimental setup.

\subsection{A generalized incident field}
Under the assumption that the incident field $\ui$ is represented by a time-harmonic solution of the wave equation, it satisfies the Helmholtz equation
\begin{align}\label{eqn: Helmholtz}
	\Delta\ui(\mathbf{r}) +k_0^2\ui(\mathbf{r}) = 0, \qquad \mathbf{r}\in\mathbb{R}^2
\end{align}
in free space, see \cite[section 3.3]{ColKre92}. In conventional diffraction tomography, the incident wave is assumed to be the simplest solution of \autoref{eqn: Helmholtz}, that is a monochromatic plane wave 
\begin{equation}\label{eq:plw}
	\ui_{\mathbf{s_0}}(\mathbf{r}) \coloneqq e^{ik_0\mathbf{r}\cdot\mathbf{s}_0}, 
\end{equation}
propagating in direction $\mathbf{s}_0\in\Sp\coloneqq\{\mathbf{s}\in\R^2:\norm{\mathbf{s}}= 1\}$. To extend DT to the use of arbitrary illumination scenarios satisfying \autoref{eqn: Helmholtz}, we suggest modeling the incident field as a superposition of monochromatic plane waves having different propagation directions
\begin{align}\label{eqn: Inc}
	\ui(\mathbf{r}) = \int_{\mathbb{S}^1} a(\mathbf{s}) e^{i k_0 \mathbf{r}\cdot\mathbf{s}}ds, 
\end{align}
where $ds = ds(\mathbf{s})$ is the standard line element and the function $a\in L^2(\mathbb{S}^1)$ specifies the amplitude of the individual plane waves arriving from each given direction $\mathbf{s}\in\mathbb S^1$. The Hilbert space $L^2(\mathbb S^1)$ is defined as the space of all square integrable functions $a: \mathbb S^1\to \mathbb C$. The technique of expanding a complex wave field into a superposition of plane waves having the same frequency but different propagating directions has its origin in the field of Fourier optics and is referred to as \textit{angular spectrum representation}, see \cite{BorWol99, ManWol95} or \cite[section 3.10]{Goo05}.

The function $a\in L^2(\Sp)$, which we refer to as \emph{beam profile} hereafter, can be chosen almost arbitrarily. This flexibility is what allows \autoref{eqn: Inc} to accommodate a wide range of incident fields.
In the theoretical part of this work, we refrain from constraining ourselves to a particular beam profile. Nevertheless, we will now introduce the concept of focused beams, as they hold relevance in a variety of tomographic applications.
\begin{example}\label{ex:gaussbeam}
Suppose the object is illuminated from the top, so that only downward propagating directions $\mathbf s= (s_1,s_2)\in\Sp$, $s_2<0$ are considered. Then, we define a \emph{Gaussian beam} to be the wave of the form in \autoref{eqn: Inc} having a Gaussian profile
\begin{align}\label{eq:gaussian_a}
	a(\mathbf{s}) =
	\begin{cases}
		e^{-As_1^2}, \quad &s_2 < 0,\\
		0,\quad &s_2>0,
	\end{cases} 
\end{align} 
where the parameter $A>0$ affects the beam waist, that is the diameter in the focus.
\end{example}
The term \textit{Gaussian beam} has its origin in optics and its construction according to \autoref{eqn: Inc} together with \autoref{eq:gaussian_a} can be found for instance in \cite{AgrPat79} or \cite[section 5.6.2]{ManWol95}. In \autoref{fig:gaussian_beams} an illustration for different amplitudes $A>0$ is given. Increasing $A$ narrows the Gaussian profile, resulting in a wider beam waist. 
\begin{figure}[t]
	\centering
	\subfloat[\tiny$A=10$]{\includegraphics[scale=0.31]{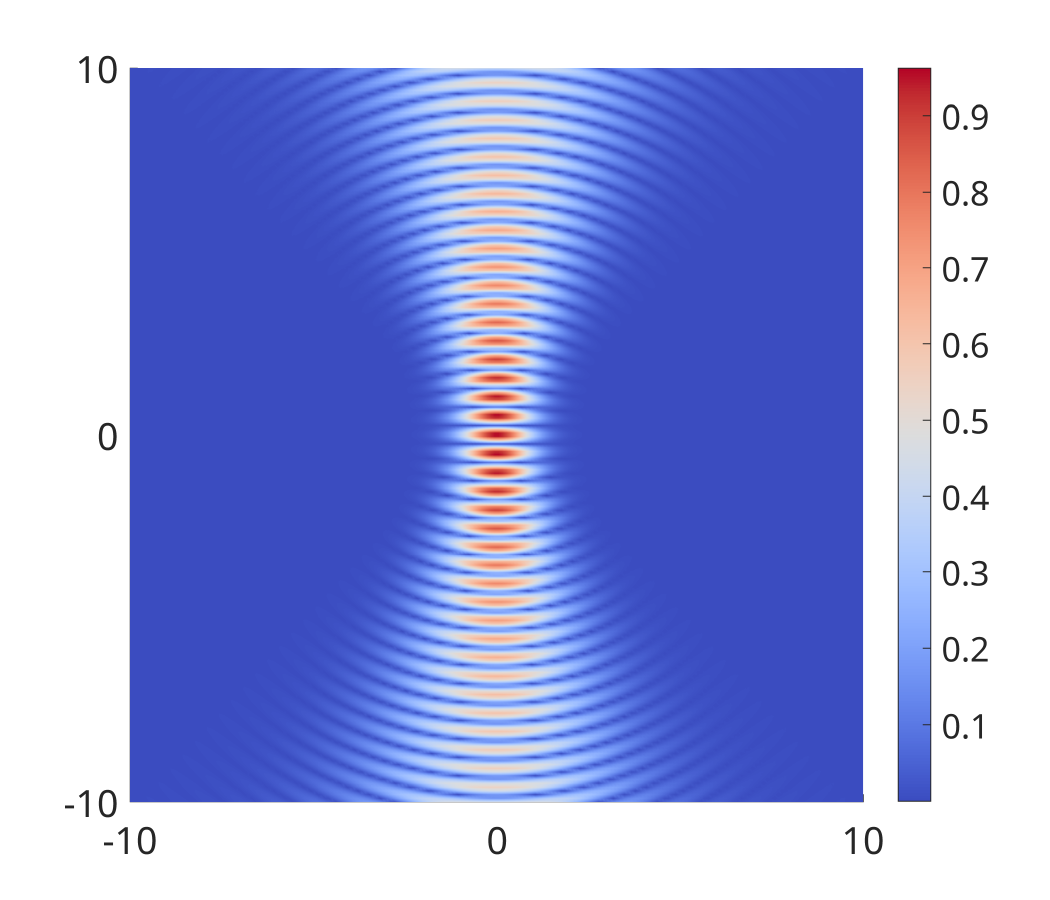}}
	\subfloat[\tiny$A=80$]{\includegraphics[scale=0.31]{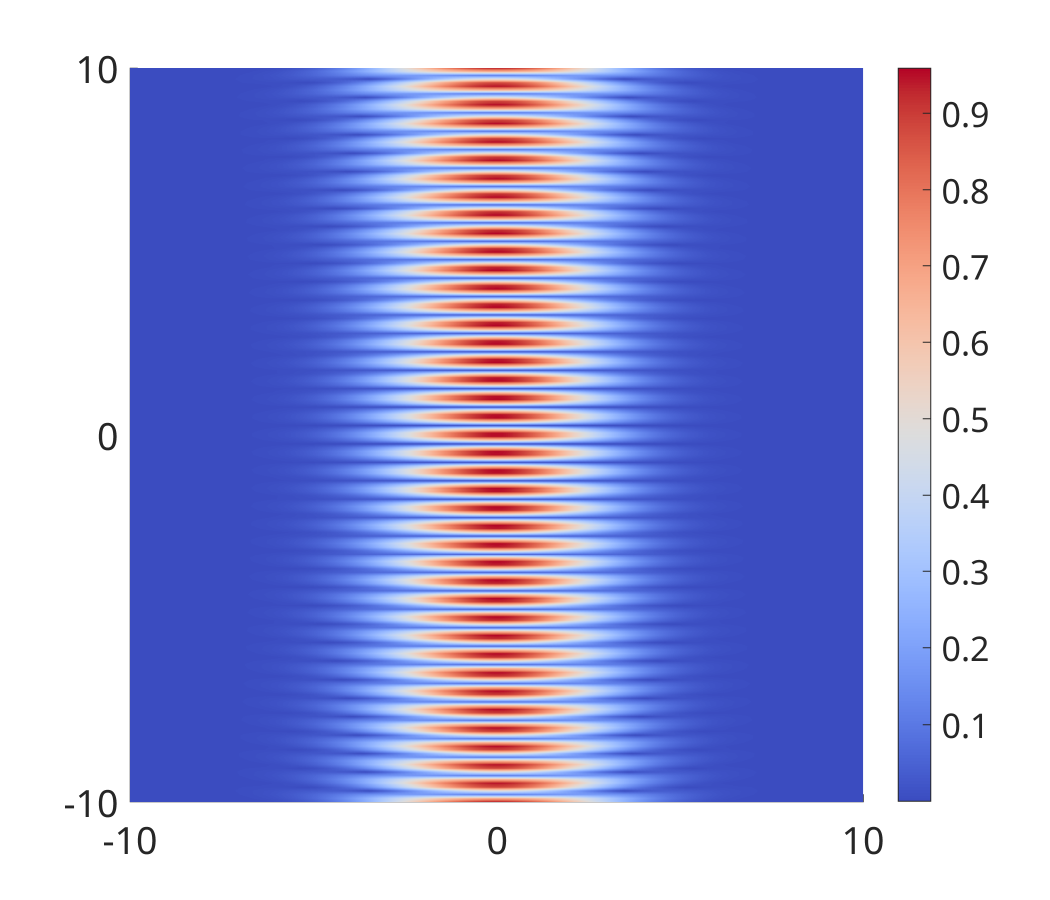}}
	\subfloat[Plane wave]{\includegraphics[scale=0.31]{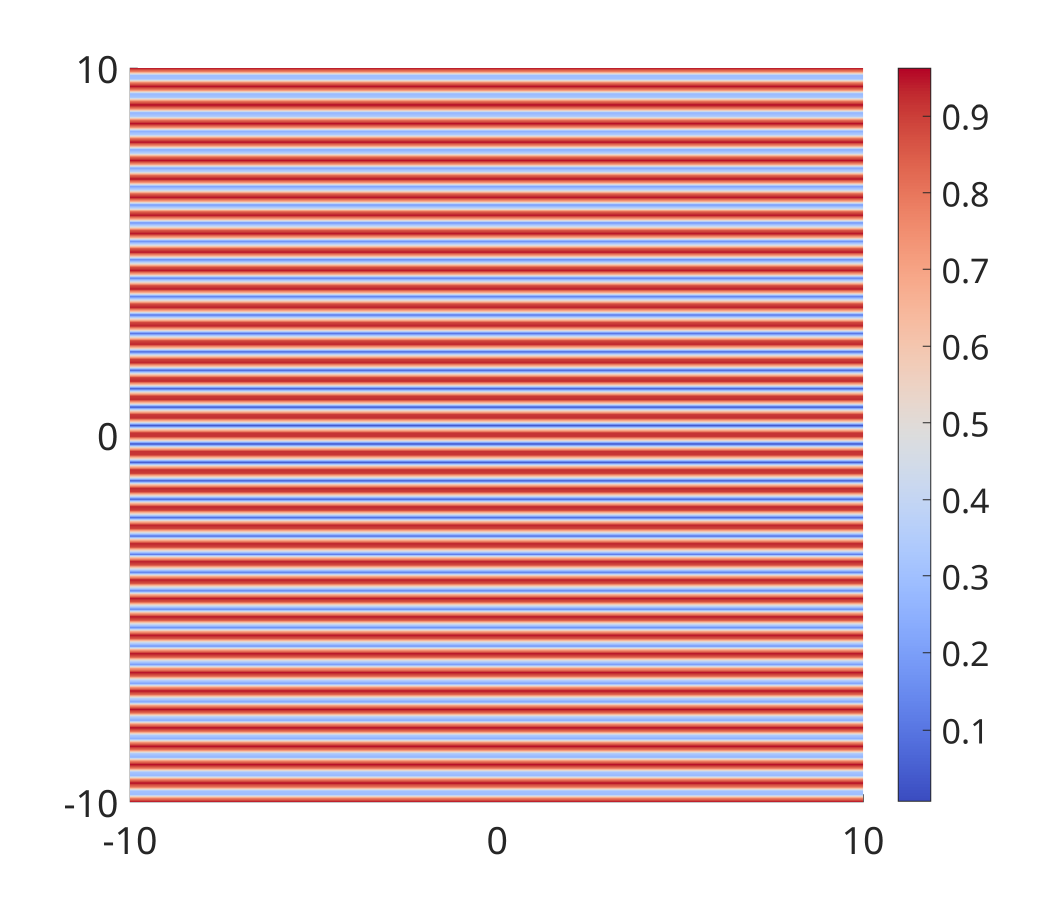}}
	\caption{\textbf{Incident illumination scenarios.}
		 (a) and (b) visualize waves $\abs{\operatorname{Re}(\ui)}$ of the form described in \autoref{eqn: Inc} with a Gaussian profile given via \autoref{eq:gaussian_a}, for different amplitudes $A$. (c) shows  a plane wave $	\ui_{\mathbf{s_0}}$ propagating in direction $\mathbf{s}_0 = (0,-1)^\intercal$.  The wavelength in all illumination settings is $\lambda = \frac{2\pi}{k_0} = 1$. }
	\label{fig:gaussian_beams}
\end{figure}

\subsection{The Born approximation} The total field $\ut$ is composed additively of the incident and the scattered wave 
\begin{equation}\label{eq:Huygen}
	\ut = \ui +\usc.
\end{equation}
Considering time harmonic waves, it satisfies the reduced wave equation \cite[Section 2.1]{ColKre92}
\begin{align}\label{eqn: Total}
	\Delta\ut(\mathbf{r})+k_0^2n^2(\mathbf{r})\ut(\mathbf{r}) = 0,\qquad\mathbf{r}\in\R^2.
\end{align} 
For the formulation of our forward model, we need to calculate the scattered wave. Inserting \autoref{eqn: Helmholtz} in \autoref{eqn: Total} leads to
\begin{align}\label{eqn: Scattered}
	\Delta\usc(\mathbf{r}) +k_0^2\usc(\mathbf{r}) = -f(\mathbf{r})(\ui(\mathbf{r}) +\usc(\mathbf{r})).
\end{align} 
To ensure uniqueness for the scattered wave, we consider in the following only outgoing waves satisfying the \textit{Sommerfeld radiation condition}, see \cite{Som12, ColKre92}.

The inverse problem of determining  $f$ from $\usc$  is inherently non-linear. 
However, if the refractive index distribution $n$ does not deviate much from the homogeneous background, the scattering effect is weak enough so that multiple scattering can be ignored. In that case one can assume $|\usc| \ll |\ui|$, which allows to replace \autoref{eqn: Scattered} by 
\begin{align}\label{eqn: Born}
	(\Delta  +k_0^2)u(\mathbf{r}) = -f(\mathbf{r})\ui(\mathbf{r}), 
\end{align}
where $u$ is the \textit{Born approximation} \cite{Wol69} to the scattered field.  
The unique outgoing solution of \autoref{eqn: Born} can be obtained by means of convolution
\begin{align}\label{eqn: u}
	u(\mathbf{r}) = \int_{\mathbb{R}^2} G(\mathbf{r}-\mathbf{r}')f(\mathbf{r}') \ui(\mathbf{r}')d\mathbf{r}', \qquad \mathbf r\in\R^2,
\end{align}
with the outgoing fundamental solution
\begin{equation*}
	G(\mathbf{r}) =\frac{i}{4}H_0^{(1)}(k_0\norm{\mathbf{r}}),
\end{equation*}
where $H_0^{(1)}$ defines the zeroth order Hankel function of the first kind \cite[Section 2.2]{ColKre92}.

\begin{remark}
	The validity of the Born approximation is limited to weakly scattering objects. In \cite{CheSta98,KakSla01, SlaKakLar84} a detailed classification of objects satisfying this assumption is provided.	There, they also discuss an alternative way to linearize \autoref{eqn: Scattered} via \textit{Rytov approximation} \cite{Dev81} and provide a comparison of the approximations.
\end{remark}

\subsection{Comparison of the forward models}\label{sec:comparison}
Now, in its mathematical formulation, the forward problem reduces to the calculation of $u\big|_{r_2 = r_M}$, which is induced by illuminating an object with $\ui$, via \autoref{eqn: u}. The original concept of DT assumes that $\ui$ is a plane wave,  whereas we have replaced this assumption with a weighted superposition of plane waves, see \autoref{eqn: Inc}. Next, we perform a numerical comparison of these forward models.

We discretize the convolution \autoref{eqn: u} using the discrete setting suggested in \autoref{sec:experiments}.
The illuminated test sample is specified by the characteristic function of the disk $\mathcal{B}_d$ with radius $d$, that is 
\begin{align}\label{eq:discsample}
	f_d(\mathbf{r}) \coloneqq \mathbf{1}_{\mathcal{B}_d}(\mathbf{r})\qquad \mathbf r\in\R^2.
\end{align}
 The results are shown in \autoref{fig:comp_forward_models}, where we compared the forward models using a downward propagating plane wave and a Gaussian beam with $A=10$ as imaging setup. The resulting scattered waves and hence the measured data significantly differ in their shapes and amplitudes. It is therefore requisite  to incorporate the correct illumination scenario in the reconstruction problem of DT.
	\begin{figure}[t]
		\centering
		\subfloat[Plane wave imaging]{\includegraphics[scale=0.5]{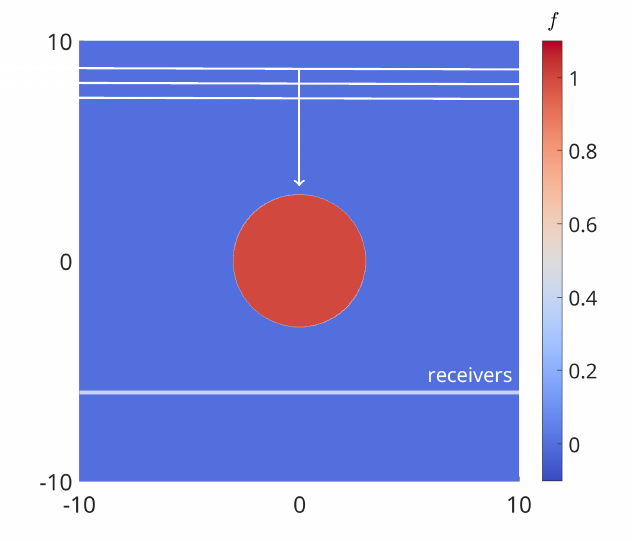}}
		\subfloat[$\operatorname{Re}(u)$]{\includegraphics[scale=0.3]{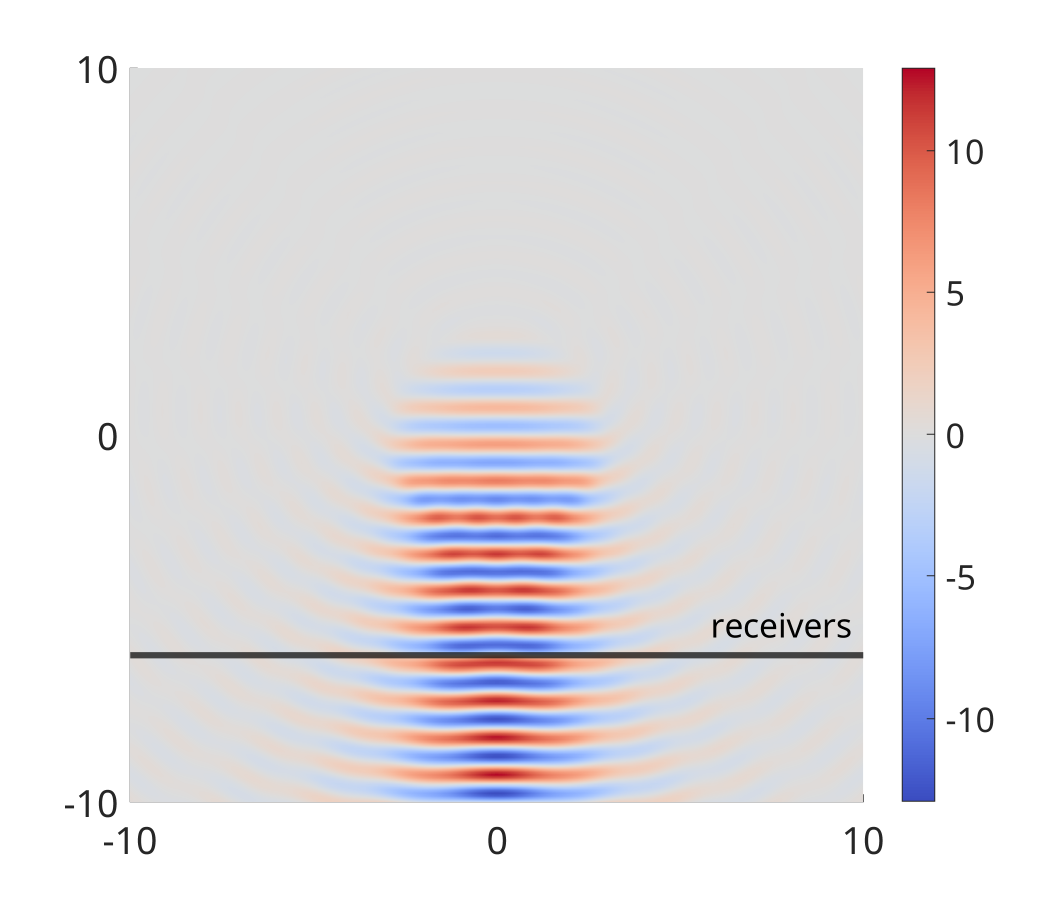}}
		\subfloat[$\operatorname{Re}\left(u\big|_{r_2 = -6}\right)$ ]{\includegraphics[scale=0.3]{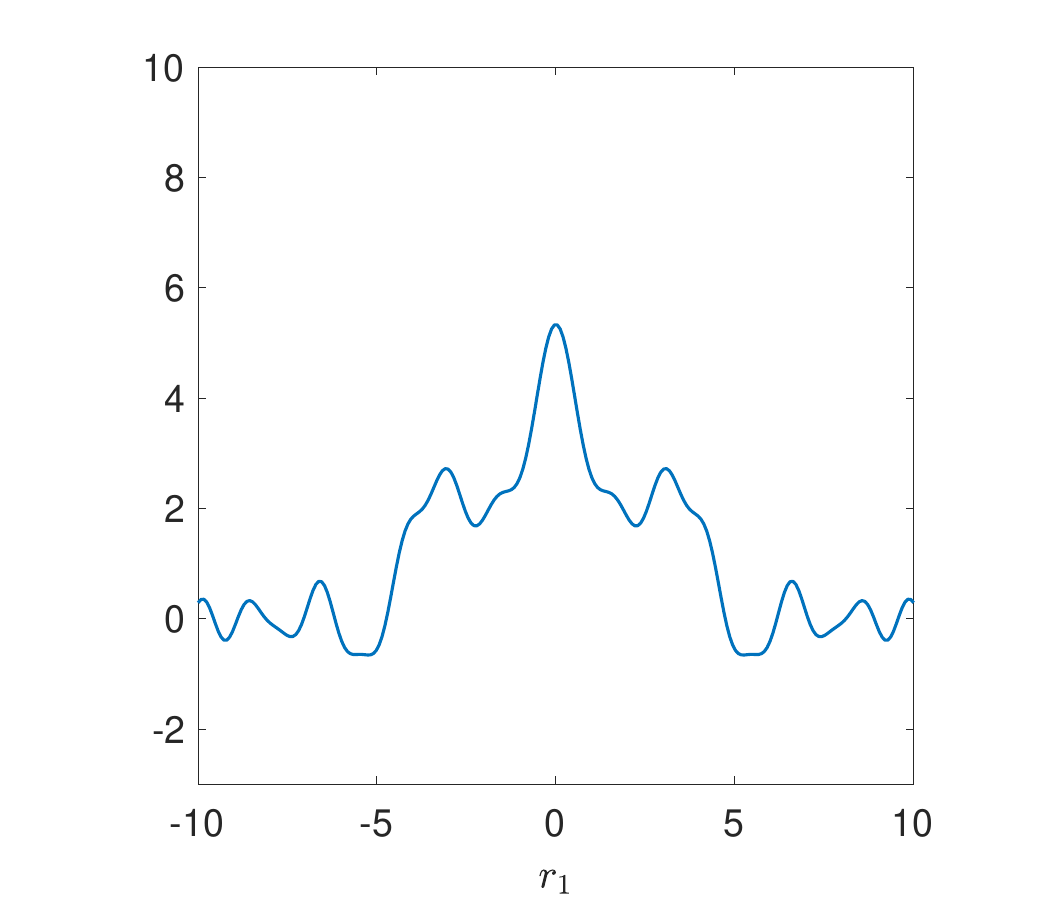}}\\
		\subfloat[Focused imaging]{\includegraphics[scale=0.5]{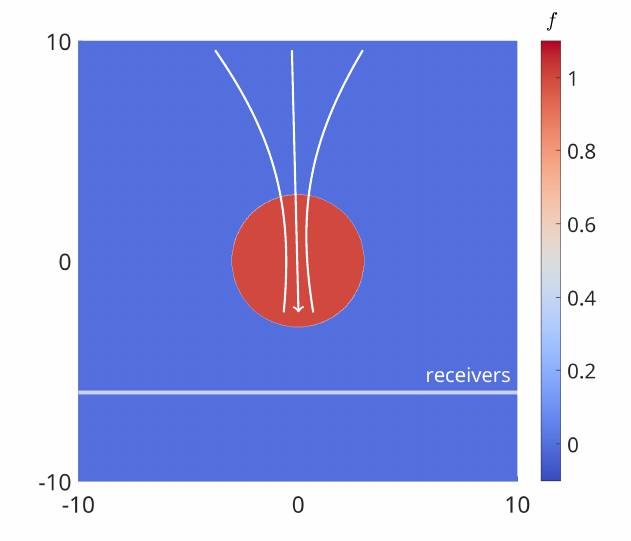}}
		\subfloat[$\operatorname{Re}(u)$]{\includegraphics[scale=0.3]{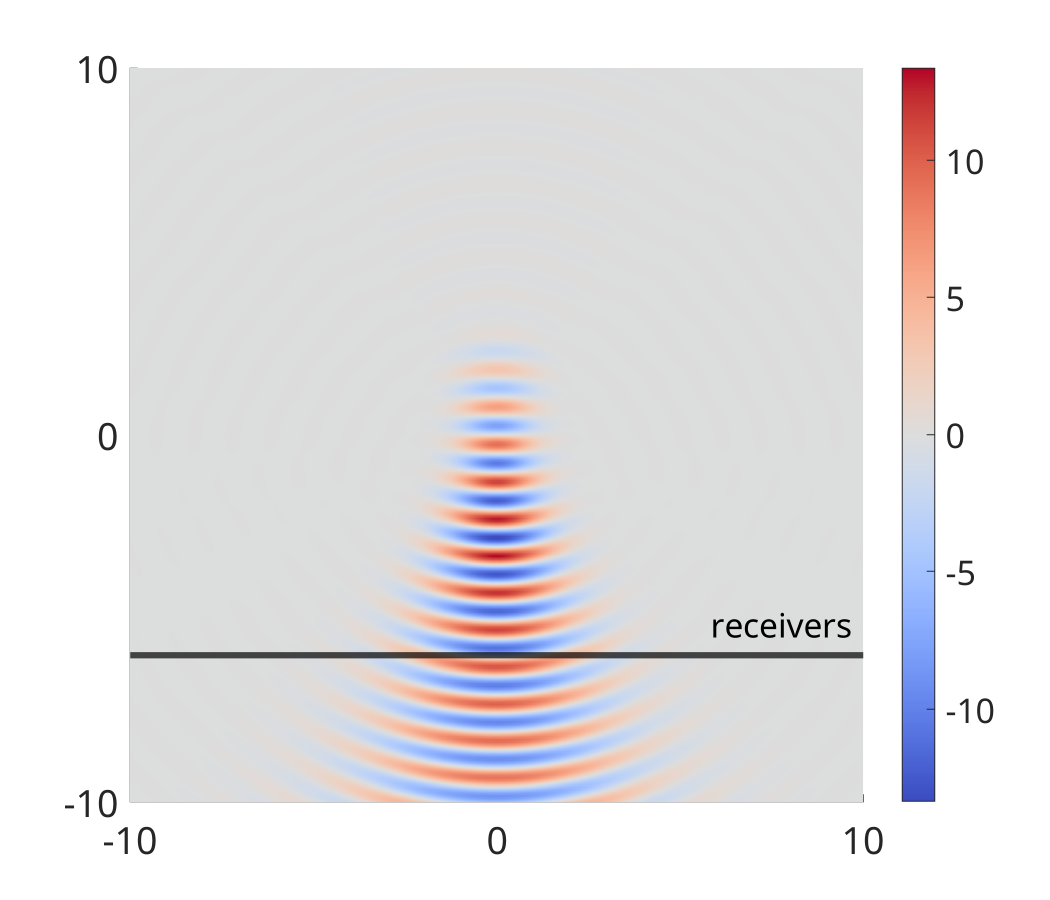}}
		\subfloat[$\operatorname{Re}\left(u\big|_{r_2 = -6}\right)$]{\includegraphics[scale=0.3]{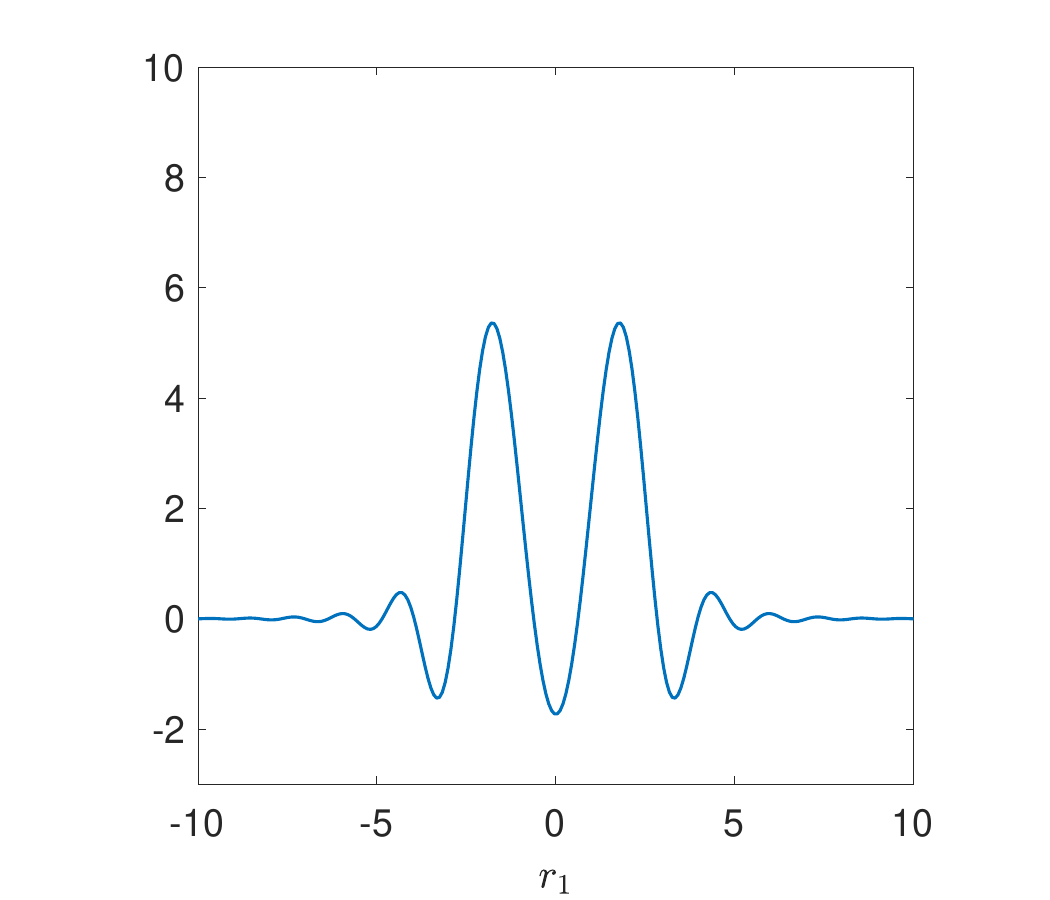}}
		\caption{\textbf{Comparison of the forward models.} An object is specified as the characteristic function $f_3\coloneqq\mathbf{1}_{\mathcal{B}_3}$. The scattered wave under Born approximation from \autoref{eqn: u} is calculated using plane wave illumination on the one hand and focused imaging on the other hand. For illumination we used a downward propagating plane wave $\ui_{\mathbf{s_0}}$, $\mathbf{s}_0 = (0,-1)^\intercal$ and a focused beam from \autoref{eq:gaussian_a} with parameter $A = 10$. In both cases the wave number is chosen as $\lambda = \frac{2\pi}{k_0}=1$ and the measurements are taken at $\{\mathbf{r}\in\R^2:r_2=-6\}$.}
		\label{fig:comp_forward_models}
	\end{figure}
\subsection{A Fourier diffraction relation}
Next, regarding the reconstruction of $f$ from a set of scattering experiments, it is convenient to examine the forward model in the spatial frequency domain, the \textit{$k$-space}. In doing so, we introduce some notation. 

Let $\mathcal{F}$ denote the $N$-dimensional Fourier transform 
\[
\mathcal F\phi(\mathbf{k}) = (2\pi)^{-\frac{N}{2}} \int_{\R^{N}} \phi(\mathbf{r}) e^{-i \mathbf{k}\cdot \mathbf{r}} d\mathbf{r}, \quad \mathbf{k} \in \R^{N}.
\]

Moreover, for a frequency $k\in(-k_0,k_0)$ we define
\begin{align}\label{def:kappa}
	\kappa(k) \coloneqq \sqrt{{k_0^2-k^2}},
\end{align}
and introduce the wave vector
\begin{align}\label{def:h}
	\mathbf{h}(k) \coloneqq (k,\kappa(k))^\intercal\in\R^2.
\end{align}
In conventional DT, the well-known Fourier diffraction theorem lays the theoretical foundation for reconstructing the scattering potential. This theorem establishes a relation between the detected scattered wave and the scattering potential in $k$-space.  Detailed explanations of this fundamental result can be found for instance in \cite{Wol69}, \cite[section 6.3]{KakSla01}, or \cite[section 3.3]{NatWue01}. However, this theorem requires that the scattered wave is induced by the illumination of the object with a plane wave. In our approach, the scattered wave results from illuminating the object with a superposition of plane waves. As a consequence, we obtain the following adapted Fourier diffraction relation.
\begin{theorem}\label{thm:FDT}
	Let $f\in L^2(\R^2)$ with $\operatorname{supp}(f)\subset\mathcal{B}_{r_s}$ and $a\in L^2(\mathbb S^1)$. Then, for $k\in(-k_0,k_0)$ and with the definitions introduced above, we have
	\begin{align}\label{eqn: FDT2}
		\mathcal{F} \big( u\big|_{r_2=r_M} \big) (k) = \sqrt{\frac{\pi}{2}}\frac{ie^{i\kappa(k) r_M}}{\kappa(k)} \int_{\mathbb{S}^1}a(\mathbf{s})\mathcal{F}f({\bf h}(k)-k_0\mathbf{s})ds.
	\end{align}
\end{theorem}

\begin{proof}
In order to simplify notation for this proof we set $v = u(\cdot,r_M)$ and $v_{\mathbf{s}} = u_{\mathbf{s}}(\cdot,r_M)$, where $u_{\mathbf{s}}$ is the scattered wave induced by the incident plane wave $\ui_\mathbf{s}(\mathbf{r}) = e^{ik_0\mathbf{r}\cdot \mathbf{s}}$.

According to the classical Fourier diffraction theorem, see \cite{KakSla01,NatWue01,Wol69}, we have
\begin{align*}
	\mathcal{F} v_{\mathbf{s}}(k) = \sqrt{\frac{\pi}{2}}\frac{ie^{i\kappa(k) r_M}}{\kappa(k)} \mathcal{F}f({\bf h}(k)-k_0\mathbf{s})
\end{align*}
for $(k,\mathbf{s}) \in (-k_0,k_0) \times \mathbb{S}^1$. Since $\mathcal{F}f$ is bounded we can multiply with $a$ and integrate over $\mathbb{S}^1$ to obtain
\begin{align}\label{eq:fdt-avg}
	\int_{\mathbb{S}^1} a(\mathbf{s}) \mathcal{F} v_{\mathbf{s}}(k) ds = \sqrt{\frac{\pi}{2}}\frac{ie^{i\kappa(k) r_M}}{\kappa(k)} \int_{\mathbb{S}^1} a(\mathbf{s}) \mathcal{F}f({\bf h}(k)-k_0\mathbf{s}) ds
\end{align}
for $k \in (-k_0,k_0).$ It remains to show that the left-hand side is equal to that of \autoref{eqn: FDT2}. This is essentially a consequence of the superposition principle,
\begin{align}\label{eq:superposition}
	\int_{\mathbb{S}^1} a(\mathbf{s}) u_{\mathbf{s}}(\mathbf{r})ds = u(\mathbf{r}),
\end{align}
for all $\mathbf{r} \in \R^2$, which can be verified using \autoref{eqn: u}. Note, however, that we cannot directly change the order of integration on the left-hand side of \autoref{eq:fdt-avg}, because $\mathcal{F} v_{\mathbf{s}}$ is in general a distributional Fourier transform (as is $\mathcal{F} \big( u\big|_{r_2=r_M} \big) = \mathcal{F}v$ on the left-hand side of \autoref{eqn: FDT2}). Recall that the Fourier transform of a tempered distribution $w \in \mathcal{S}'(\R)$ is the unique $\mathcal{F}w \in \mathcal{S}'(\R)$ satisfying $\mathcal{F}w[\phi] = w[\mathcal{F}\phi]$ for all Schwartz functions $\phi \in \mathcal{S}(\R)$, see for instance \cite{Gra08,Gru09}.

Note that $v$ and $v_{\mathbf{s}}$ are smooth by virtue of \autoref{eqn: u} and because $G$ is smooth on $\R \setminus \{0\}$. In addition they are bounded, since $H^{(1)}_0(\norm{\mathbf{x}}) = O(\norm{\mathbf{x}}^{-1/2})$ for $\norm{\mathbf{x}} \to \infty$, see \cite[\href{https://dlmf.nist.gov/10.2.E5}{(10.2.5)}]{DLMF}. It follows that $v,v_{\mathbf{s}}$ belong to the space of tempered distributions $\mathcal{S}'(\R)$.

Now consider a test function $\phi \in C_c^\infty (\R)\subset \mathcal{S}(\R)$ with support in $(-k_0,k_0)$. Then, since $1/\kappa$ is integrable on $(-k_0,k_0)$, it follows from \autoref{eq:fdt-avg} that
\begin{align*}
	\int_{\R} \phi(k) \int_{\mathbb{S}^1} a(\mathbf{s}) \mathcal{F} v_{\mathbf{s}}(k) ds \, dk = i \sqrt{\frac{\pi}{2}} \int_{\R} \phi(k) \frac{e^{i\kappa(k) r_M}}{\kappa(k)} \int_{\mathbb{S}^1} a(\mathbf{s}) \mathcal{F}f({\bf h}(k)-k_0\mathbf{s}) ds \, dk.
\end{align*}
Applying Fubini's theorem and the distributional definition of the Fourier transform we find
\begin{align*}
	\int_{\R} \phi(k) \int_{\mathbb{S}^1} a(\mathbf{s}) \mathcal{F} v_{\mathbf{s}}(k) ds \, dk
		&= \int_{\mathbb{S}^1} a(\mathbf{s}) \int_{\R} \phi(k) \mathcal{F} v_{\mathbf{s}}(k) dk \, ds
		 = \int_{\mathbb{S}^1} a(\mathbf{s})\mathcal{F} v_{\mathbf{s}} [\phi] ds
		 = \int_{\mathbb{S}^1} a(\mathbf{s}) v_{\mathbf{s}}[\mathcal{F} \phi] ds \\
		&= \int_{\mathbb{S}^1} a(\mathbf{s}) \int_\R v_{\mathbf{s}}(x) \mathcal{F} \phi(x) dx \, ds.
	\intertext{Another application of Fubini's theorem combined with \autoref{eq:superposition} yields}
		&= \int_\R \mathcal{F} \phi(x) \int_{\mathbb{S}^1} a(\mathbf{s}) v_{\mathbf{s}}(x) ds \, dx = \int_\R \mathcal{F} \phi(x) v(x) \, dx
		 = v[\mathcal{F}\phi] = \mathcal{F}v[\phi].
\end{align*}
Summarizing, we have shown that
\begin{align*}
	\mathcal{F}v[\phi] = i \sqrt{\frac{\pi}{2}} \int_{\R} \phi(k) \frac{e^{i\kappa(k) r_M}}{\kappa(k)} \int_{\mathbb{S}^1} a(\mathbf{s}) \mathcal{F}f({\bf h}(k)-k_0\mathbf{s}) ds \, dk
\end{align*}
for every test function $\phi$ supported in $(-k_0,k_0)$. Therefore, on the interval $(-k_0,k_0)$ we can identify $\mathcal{F}v = \mathcal{F} \big( u\big|_{r_2=r_M} \big)$ with the function given by the right-hand side of \autoref{eqn: FDT2}.
\end{proof}

While \autoref{thm:FDT} is an adaption of the classical Fourier diffraction theorem, there is a decisive difference: Instead of having \textit{direct} access to $\mathcal{F}f$, now the measurements only provide information about averages of $\mathcal{F}f$ taken over circles and weighted by the beam profile.
\subsection{Rotating the incident field}\label{subsec:data_ac}
 The data presented in \autoref{eqn: FDT2} correspond to a scenario in which the object is illuminated once and the scattered field is measured at a line below the object. However, this configuration does not provide sufficient data for a complete reconstruction of the scattering potential $f$, which is a two-dimensional complex-valued function. In response to the demands of imaging applications, particularly those involving the movement of the emitting device around the object, we address this limitation by rotating the incident field, see also \autoref{subfig:focusedDT}. We introduce the rotated incident field
\begin{equation*}
	\ui_\theta(\mathbf{r})\coloneqq\ui(R_\theta\mathbf{r}) =\int_{\mathbb S^1} a(R^{\intercal}_{\theta}\mathbf{s})e^{ik_0\mathbf{r}\cdot\mathbf{s}}ds,
\end{equation*}
where the matrix
\begin{align*}
	R_{\theta}=
	\begin{pmatrix} 
		\cos\theta & -\sin\theta  \\ 
		\sin\theta & \cos\theta  \\ 
	\end{pmatrix}
\end{align*}
specifies the counter clockwise rotation through the angle $\theta\in[-\pi,\pi]$. The corresponding scattered wave  $u_\theta(\mathbf{r})$ satisfies
\begin{align*}
	(\Delta  +k_0^2)u_\theta(\mathbf{r}) = -f(\mathbf{r}) \ui_\theta(\mathbf{r}).
\end{align*}
 For each rotation of the incident field a measurement is conducted. Hereby, we assume that the detector is kept in the same position for every experiment.
 Now,  we define the full set of data in $k$-space by 
 \begin{equation}\label{eq:m}
 	m(k,\theta)\coloneqq-\sqrt{\frac{2}{\pi}}i\kappa(k) e^{-i\kappa(k) r_M}\mathcal{F} u_\theta(k,r_M)
 \end{equation}
which we assume to be given for all $k\in(-k_0,k_0)$ and all $ \theta\in[-\pi,\pi]$. 
According to \autoref{thm:FDT} they are related to the scattering potential $f$ via 
\begin{equation}\label{eq:setup}
	m(k,\theta) = \int_{\mathbb S^1}  a(R^\intercal_\theta\mathbf{s})  \mathcal Ff(\mathbf{h} (k)- k_0  \mathbf{s}) ds.
\end{equation} 
Using a parametrization  of $\mathbf{s}\in\mathbb{S}^1$ in terms of polar coordinates  
\begin{equation}\label{eq:param}
	\mathbf{s}(\varphi)=(\cos\varphi,\sin\varphi)^\intercal, \quad  \varphi\in[-\pi,\pi)
\end{equation}
turns \autoref{eq:setup} into
\begin{equation}\label{eq:setup1}
	m(k,\theta) = \int_{-\pi}^\pi a(\mathbf{s}(\varphi-\theta)) \mathcal Ff(\mathbf{h} (k)- k_0  \mathbf{s}(\varphi))d\varphi, \qquad k\in(-k_0,k_0), \ \theta\in[-\pi,\pi]
\end{equation}
which lays the groundwork for the process of determining $f$.
	\section{Inverse problem}\label{sec:invprob}
Based on our results derived up to this point we now address the \emph{inverse problem}. It consists in extracting the scattering potential $f$ from $m$ via \autoref{eq:setup1}. This relation is not direct. However, given our measurement configuration, which involves rotating the incident field, we can solve this problem through a two-step inversion process:
\begin{enumerate}
	\item[1.)] In the first step, the Fourier transform $\mathcal Ff(\mathbf{h} (k)- k_0  \mathbf{s}(\varphi))$ is to be identified for all $k\in(-k_0,k_0)$ and $\varphi\in[-\pi,\pi)$ from \autoref{eq:setup1}.
	\item[2.)] Subsequently, in the second step, the scattering potential $f$ can be extracted via Fourier inversion.
\end{enumerate}
In the following subsections, we discuss both inversion steps separately and present an explicit reconstruction formula afterward.

\subsection{Identification of the data in $k$-space} \label{sec:svd}
The aim of this section is to infer the $k$-space data $\mathcal{F}f(\mathbf{h} (k)- k_0  \mathbf{s}(\varphi))$ for all $k\in(-k_0,k_0)$ and $\varphi\in[-\pi,\pi)$ from \autoref{eq:setup1}. 
For this purpose, we use the dependency of the measurements on the rotation angles $\theta \in [-\pi,\pi]$ and consider the frequency $k\in(-k_0,k_0)$ to be fixed. Then,
\[g(k,\cdot)\coloneqq {}\mathcal{F}f(\mathbf{h} (k)- k_0  \mathbf{s}(\cdot))\in L^2\left( [-\pi,\pi]\right),\] because $\mathcal{F}f$ is smooth.

Moreover, with the Cauchy-Schwarz inequality, we obtain that also
\[m(k,\cdot)\in L^2\left( [-\pi,\pi]\right).\]

For this reason, we  introduce the following integral operator
\begin{align}\label{eq:Ka}
	\mathcal A: L^2\left( [-\pi,\pi]\right) \to L^2\left( [-\pi,\pi]\right) , \qquad \mathcal Ag(\theta) =\int_{-\pi}^{\pi}a(\varphi-\theta)g(\varphi)d\varphi,
\end{align}
whose integral kernel is defined as a translation of the beam profile $a\in L^2([-\pi,\pi])$. As in \autoref{eq:Ka} we identify $a$ with a $2\pi$-periodic function on $\R$ from now on and refrain from explicitly writing the parametrization. This operator is in general not self-adjoint since we do not impose $a(\varphi)=\overline{a(-\varphi)}$. From the fact that $a\in L^2([-\pi,\pi])$ it follows directly that $\mathcal A$ is compact, see \cite[theorem 3.4]{MueSil12}.

Accordingly, the problem presented in \autoref{eq:setup1} can be reformulated as a system of (operator) equations 
\begin{equation}\label{eq:sysop}
	\mathcal{A} g(k,\cdot) = m(k,\cdot), \qquad k\in(-k_0,k_0).
\end{equation}

Each equation leads due to the compactness of $\mathcal A$ to an \emph{ill-posed} problem.  For a comprehensive treatment of such problems, we refer to  \cite{EngHanNeu96,Kre89,MueSil12}. For simplicity of notation, we refrain from writing the dependency of the functions $g$ and $m$ on the frequency $k\in(-k_0,k_0)$ in this subsection.
 
 A foundational tool for the analysis of a non-selfadjoint compact linear operator is the theory of \emph{singular systems}. We recall from \cite[section 2.2]{EngHanNeu96} the following definition:
\begin{definition}\label{def:sing_sys}
	Let $K:\mathcal X\to \mathcal Y$ be a compact operator between two Hilbert spaces $\mathcal{X},\mathcal{Y}$ and let $K^*:\mathcal{Y}\to\mathcal{X}$ denote the adjoint of $K$. The set $\left\lbrace \left( \sigma_n;v_n,u_n\right): n\in\N_0\right\rbrace$ defines a \emph{singular system} of $K$ if 
	$\sigma_n > 0$, $\{(\sigma_n^2,v_n):\ n\in\N_0\}$ forms an eigensystem to $K^*K$ and $u_n=\frac{Kv_n}{\norm{ Kv_n}}.$ Here, the singular values can be written down in decreasing order with multiplicity, i.e. $\sigma_1	\geq \sigma_2\geq\dots>0$.
\end{definition}
To determine a singular system for $\mathcal A$ as defined in \autoref{eq:Ka}, we introduce for $g,h\in L^2([-\pi,\pi])$ the inner product
\begin{equation*}
	\langle g,h\rangle\coloneqq\frac{1}{2\pi}\int_{-\pi}^\pi g(\varphi)\overline{h(\varphi)}d\varphi,
\end{equation*}
that induces a norm $$ \norm{g}^2 = \langle g,g\rangle.$$

The adjoint operator $\mathcal A^\ast$ of $\mathcal A$ is given by 
\[
\mathcal A^*: L^2\left( [-\pi,\pi]\right) \to L^2\left( [-\pi,\pi]\right) , \qquad (\mathcal A^*m)(\varphi) = \int_{-\pi}^{\pi}\overline{a(\varphi-\theta)}m(\theta)d\theta.
\]
Note that the functions
\begin{equation*}
	e_n:[-\pi,\pi] \to \C,\qquad e_n(\varphi)\coloneqq e^{-in\varphi}, \quad n\in\Z
\end{equation*}
form a complete orthonormal system in  $L^2\left( [-\pi,\pi]\right)$ equipped with the above defined inner product. Herewith, a singular system can be characterized as follows:
\begin{lemma}\label{lem:sing_sys}
	Let $a\in L^2([-\pi,\pi])$ and
	\begin{equation}\label{eq:coeffa}
		a_n = \langle a,\overline{e_n}\rangle, \quad n\in\Z
	\end{equation}
	be its non-zero Fourier coefficients.
	Then, the set
	\begin{align}\label{eq:singularsys}
		\left\lbrace \left( 2\pi\abs{a_n};e_n,\frac{a_n}{|a_n|}e_n\right): n\in\Z\right\rbrace
	\end{align}
	forms a singular system of $\mathcal A$. 
\end{lemma}
\begin{proof}
First, we expand the integral kernel as a Fourier series 
\begin{equation}\label{eq:a_fourierseries}
	a(\varphi-\theta)=\sum_{n\in\Z} a_n \overline{e_n(\varphi-\theta)} =
	\sum_{n\in\mathcal \Z} a_n \overline{e_n(\varphi)}e_n(\theta)
\end{equation} 
for almost every $\varphi\in[-\pi,\pi)$. Then, by exploiting \autoref{eq:a_fourierseries}, we calculate
\begin{align*}
	\left(\mathcal A e_n\right)(\theta) = \int_{-\pi}^\pi a(\varphi-\theta)e_n(\varphi)d\varphi
	= \sum_{l\in\Z}a_l\int_{-\pi}^\pi e_n(\varphi)\overline{e_l(\varphi)}d\varphi e_l(\theta) 
	= 2\pi\sum_{l\in\Z}a_le_l(\theta)\delta_{nl}= 2\pi a_n e_n(\theta),
\end{align*}
and, therefore, together with \autoref{eq:a_fourierseries}, we get
\begin{align*}
	\left(\mathcal A^*\mathcal Ae_n\right)(\varphi) 
	&= 2\pi a_n \int_{-\pi}^\pi \overline{a(\varphi-\theta)}e_n(\theta)d\theta
	= 2\pi a_n\sum_{l\in\Z} \overline{a_l} \int_{-\pi}^\pi e_n(\theta)\overline{e_l(\theta)}d\theta e_l(\varphi)
	\\& = 	4\pi^2 a_n\sum_{l\in\Z} \overline{a_l} e_l(\varphi)\delta_{nl}= 4\pi^2\abs{a_n}^2 e_n(\varphi)
\end{align*}
for all $n\in\Z$. Hence, $4\pi^2 \abs{a_n}^2>0, \ n\in\Z$ are eigenvalues of $\mathcal A^*\mathcal A$. Further, we have
\[
 \norm{\mathcal Ae_n} = 2\pi \abs{a_n} \norm{e_n}= 2\pi \abs{a_n}
\]
so that in total \autoref{eq:singularsys} forms a singular system of $\mathcal A$ according to \autoref{def:sing_sys}.
\end{proof}
Using this singular system, the operator $\mathcal A$ can be diagonalized as 
\begin{equation}\label{eq:diag}
	\mathcal Ag = 2\pi\sum_{n\in\Z}a_n\langle g ,e_n\rangle e_n,\qquad g\in L^2([-\pi,\pi]).
\end{equation}
According to \cite[theorem 2.8]{EngHanNeu96}, the singular system provides a series representation for the \emph{best-approximate solution} of $\mathcal{A}g = m$, which is defined as the unique least-squares solution of minimal norm, that is, 
\[ \mathcal{A}^\dagger m  \coloneqq \argmin \left\{ \norm{g} \colon g\in L^2([-\pi,\pi]), \norm{\mathcal{A} g - m} = \inf_{h} \norm{\mathcal{A}h  - m} \right\}. \]
 \begin{theorem}\label{thm:invert_int}
 	Suppose that $m \in L^2([-\pi,\pi])$ satisfies the Picard condition
  	\begin{equation}\label{eq:picard}
	\sum_{n\in\Z}\left( \frac{\abs{\langle m,e_n\rangle}}{ \abs{a_n}}\right)^2<\infty.
	\end{equation}
	Then, 
 	\begin{equation}\label{eq:fourier_direct}
 		\mathcal{A}^\dagger m = \frac{1}{2\pi}\sum_{n\in\Z}\frac{\langle m,e_n\rangle}{a_n}e_n.
 	\end{equation}
 \end{theorem}
This result demonstrates the ill-posed nature of the problem which is characterized by the fact that the singular values $|a_n|$ tend to zero: The condition \autoref{eq:picard}, which ensures the existence of a least-squares solution, is satisfied if the coefficients $\langle m,e_n\rangle$ decay fast enough relative to the values $a_n$.  Moreover, in the presence of perturbations in the input data $m$, we can deduce from \autoref{eq:fourier_direct} that the instability of the inverse problem is also characterized by the decay of the coefficients $a_n$. This is why the choice of the integral kernel $a\in L^2([-\pi,\pi])$  can influence the \emph{degree of ill-posedness}. For further details see \cite[sect.\ 2.2]{EngHanNeu96}, \cite[sect.\ 15.4]{Kre89} or \cite[sect.\ 3.3]{MueSil12}.

However, it is straightforward to apply a regularization scheme that filters out the data associated with small singular values. To achieve this, we consider the \emph{truncated singular value decomposition}  (TSVD) \cite[section 4.1]{EngHanNeu96}
	\begin{equation}\label{eq:tuncsvd}
	\mathcal{A}^\dagger_Nm\coloneqq \frac{1}{2\pi}\sum_{n=-N}^{N}\frac{\langle m,e_n\rangle}{a_n}e_n, \qquad N\in\N
	\end{equation}
for some $N\in\N$.  Here, the truncation level $N$ acts as a regularization parameter, that should be chosen such a way that all noise-dominated singular value coefficients are neglected. In the numerical part of this article, see \autoref{subsec:testsvd}, we specify this parameter by inspecting the discrete Picard criterion as it is proposed and studied in \cite{Han10}. As an alternative regularization scheme it is also straightforward to apply Tikhonov-type regularization which corresponds to a multiplication of the summands in \autoref{eq:fourier_direct} with some damping coefficients. Anyhow, we restrict ourselves to the regularization via \autoref{eq:tuncsvd} and refer to \cite{EngHanNeu96, Kre98} for a detailed discussion of the regularization strategies.

\subsection{Frequency coverage and Fourier inversion}\label{subsec:bp}
According to the analysis of the integral operator $\mathcal{A}$, the $k$-space data $\mathcal{F}f(\mathbf{h}(k)-k_0\mathbf{s}(\varphi))$ are now approximated by ($\mathcal{A}^\dagger_Nm)(k,\varphi)$ for all $k\in(-k_0,k_0)$ and all $\varphi\in[-\pi,\pi).$
Based on this, we intend to finally quantify the target function $f$ via Fourier inversion. In doing so, we introduce the set
\begin{equation*}
	\mathcal{U} \coloneqq \left\lbrace (k,\varphi)\in\R^2: k\in(-k_0,k_0), \ \varphi\in[-\pi,\pi)\right\rbrace
\end{equation*}
and define the map
\begin{equation}\label{eq:T}
	T:\mathcal{U} \to \R^2,\qquad T(k,\varphi)\coloneqq \mathbf{h}(k)-k_0\mathbf{s}(\varphi) = \begin{pmatrix}
		k-k_0\cos\varphi\\
		\kappa(k)-k_0\sin\varphi
	\end{pmatrix}
\end{equation}
that traces out the domain in $k$-space that is covered by the measurements. Under the assumption that $\mathcal{F}f(T(k,\varphi))$ can be determined exactly from $m$ for all $(k,\varphi)\in\mathcal{U}$, we have access to $\mathcal{F}f$ on the domain
\begin{equation}\label{eq:frequency_coverage}
	T(\mathcal{U})=\left\lbrace
	\begin{pmatrix}
		k-k_0\cos\varphi\\
		\kappa(k)-k_0\sin\varphi
	\end{pmatrix}\in\R^2: \  k\in(-k_0,k_0), \ \varphi\in[-\pi,\pi)\right\rbrace\subset \R^2
\end{equation} 
in $k$-space which is depicted in \autoref{fig:frequency_coverage}. Every plane wave angle $\varphi\in[-\pi,\pi)$ provides data about $\mathcal{F}f$ along a semicircle with radius $k_0$. For $\varphi\in[-\pi,0]$ these semicircles fill the half disk of radius $2k_0$ in the upper half plane except for the disks of radius $k_0$ around $(\pm k_0,0)^\intercal$ that are captured by $\varphi\in(0,\pi)$. In total, a frequency coverage of the form depicted in \autoref{subfig:fullcov} is obtained. Note that this coverage is bounded by 
\[
 \norm{T(k,\varphi)} \leq 2k_0
\]
such that the determination of $f$ via Fourier inversion using the available frequencies on $T(\mathcal{U})\subset \R^2$ in $k$-space provides us a \emph{low-pass filtered} approximation of the scattering potential
\begin{equation}\label{eq:f_bp}
	f(\mathbf{r})\approx \frac{1}{2\pi}\int_{T(\mathcal{U})}\mathcal{F}f(\mathbf{y})e^{i\mathbf{y}\cdot\mathbf{r}}d\mathbf{y}.
\end{equation}	
To calculate this formula explicitly we first need to perform the change of variables $\mathbf{y} \mapsto (k,\varphi)$ in the integral in order to subsequently insert the $k$-space data  approximated by $\mathcal{A}^\dagger_Nm$. The map $T:\mathcal{U}\to\R^2$, however, is not injective, as is shown in \autoref{lem:Card} below.  We need to take into account the number of times a point $\mathbf y\in T(\mathcal{U})$ in the frequency coverage is captured by $T$. Therefore, we introduce the \emph{Banach indicatrix} \cite{Haj93} $\Card(T^{-1}(\mathbf{y}))$ of $T$, where $\Card$ is the counting measure.
\begin{figure}[t]
	\centering
	\begin{subfigure}[b]{0.3\textwidth}
		\centering
		\begin{tikzpicture}[scale=1]
			\fill[blue,opacity=0.2] (2,0) arc (0:180:2cm);
			\fill[white] (2,0) arc (0:180:1cm);
			\fill[white] (0,0) arc (0:180:1cm);
			
			\draw[blue] (0,0) arc (0:180:1cm);
			\draw[blue] (2,0) arc (0:180:1cm);
			\draw[blue] (1,1) arc (0:180:1cm);
			\draw[blue,dashed] (1,0) arc (0:180:1cm);
			
			\draw[->] (0,-1.2) -- (0,3) node[anchor = west]{$y_2$};
			\draw[->] (-2.5,0) -- (2.5,0) node[anchor = north]{$y_1$};
			
			\draw (1,-0.08) -- (1,0.08);
			\node[anchor = south] at (1,0){$k_0$};
			\draw (-1,-0.08) -- (-1,0.08);
			\node[anchor = south] at (-1,0){$-k_0$};
			\draw (-0.08,2) -- (0.08,2);
			\node[anchor = west] at (0,2.15){$2k_0$};
		\end{tikzpicture}
	\caption{$\mathcal{I}=[-\pi,0]$}
	\label{subfig:upperfq}
	\end{subfigure}\hfill
\begin{subfigure}[b]{0.3\textwidth}
	\centering
	\begin{tikzpicture}[scale=1]
		\fill[blue,opacity=0.2] (2,0) arc (0:360:1cm);
		\fill[blue,opacity=0.2] (0,0) arc (0:360:1cm);
		
		\draw[blue] (0,0) arc (0:180:1cm);
		\draw[blue] (2,0) arc (0:180:1cm);
		\draw[blue] (1,-1) arc (0:180:1cm);
		\draw[blue,dashed] (-1,0) arc (180:360:1cm);
		
		\draw[->] (0,-1.2) -- (0,3) node[anchor = west]{$y_2$};
		\draw[->] (-2.5,0) -- (2.5,0) node[anchor = north]{$y_1$};
		
		\draw (1,-0.08) -- (1,0.08);
		\node[anchor = south] at (1,0){$k_0$};
		\draw (-1,-0.08) -- (-1,0.08);
		\node[anchor = south] at (-1,0){$-k_0$};
		\draw (-0.08,2) -- (0.08,2);
		\node[anchor = west] at (0,2.15){$2k_0$};
	\end{tikzpicture}
\caption{$\mathcal{I}=[0,\pi]$}
\label{subfig:twocircles}
\end{subfigure}\hfill
\begin{subfigure}[b]{0.3\textwidth}
	\centering
	\begin{tikzpicture}[scale=1]
		\fill[blue,opacity=0.2] (2,0) arc (0:180:2cm);
		\fill[blue,opacity=0.2] (2,0) arc (0:-180:1cm);
		\fill[blue,opacity=0.2] (0,0) arc (0:-180:1cm);
		
		\draw[blue] (0,0) arc (0:180:1cm);
		\draw[blue] (2,0) arc (0:180:1cm);
		\draw[blue] (1,-1) arc (0:180:1cm);
		\draw[blue] (1,1) arc (0:180:1cm);
		\draw[blue,dashed] (-1,0) arc (-180:180:1cm);
		
		\draw[->] (0,-1.2) -- (0,3) node[anchor = west]{$y_2$};
		\draw[->] (-2.5,0) -- (2.5,0) node[anchor = north]{$y_1$};
		
		\draw (1,-0.08) -- (1,0.08);
		\node[anchor = south] at (1,0){$k_0$};
		\draw (-1,-0.08) -- (-1,0.08);
		\node[anchor = south] at (-1,0){$-k_0$};
		\draw (-0.08,2) -- (0.08,2);
		\node[anchor = west] at (0,2.15){$2k_0$};
	\end{tikzpicture}
\caption{$\mathcal{I}=[-\pi,\pi]$}
\label{subfig:fullcov}
\end{subfigure}
	\caption{\textbf{Construction of the 2D frequency coverage.} The coverage $T((-k_0,k_0)\times \mathcal{I})$ for different sub-intervals $\mathcal{I}\subseteq[-\pi,\pi]$ is illustrated. The resulting shadowed coverages are built from a union of infinitely many semicircles of radius $k_0$ whose centres lie on the dashed line $-k_0\mathbf{s}(\varphi)$, $\varphi\in\mathcal{I}$. Some of these semicircles are depicted in blue. (c) shows the full coverage $T(\mathcal{U})$ as the union of (a) and (b).}
	\label{fig:frequency_coverage}
\end{figure}
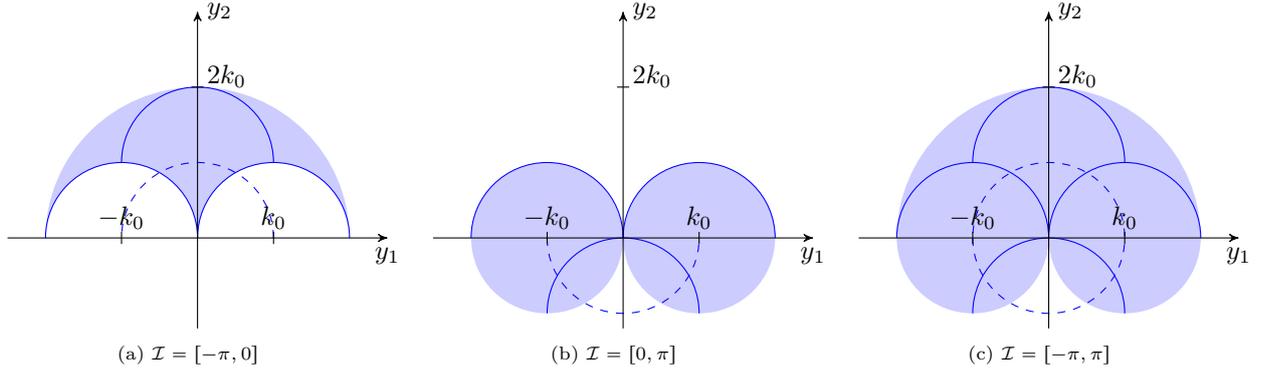

\begin{theorem}\label{thm:backprop}
	 Let $T$ be defined as in \autoref{eq:T}. Then, the integral in \autoref{eq:f_bp} can be determined via
	\begin{equation}\label{eq:fbp}
		\begin{split}
			\int_{T(\mathcal{U})}\mathcal{F}f(\mathbf{y})e^{i\mathbf{y}\cdot\mathbf{r}}d\mathbf{y} = \int_{\mathcal{U}} \mathcal Ff(T(k,\varphi)) e^{i T(k,\varphi)\cdot\mathbf{r}}\frac{\abs{\det(\nabla T(k,\varphi))} }{\operatorname{Card}(T^{-1}(T(k,\varphi)))}d(k,\varphi), \quad \mathbf{r}\in\R^2,
		\end{split}
	\end{equation}
	where 
	\begin{equation}\label{eq:jacobian}
		\det(\nabla T(k,\varphi)) = k_0\left( \frac{k}{\kappa(k)}\sin\varphi-\cos\varphi\right), \quad (k,\varphi)\in\mathcal{U}
	\end{equation} 
	is the Jacobian determinant of $T$.
\end{theorem}
\begin{proof}
	The map $T$ is differentiable and it is straightforward to calculate the Jacobian matrix
	\begin{align*}
		\nabla T(k,\varphi) = \begin{pmatrix}
			1& k_0\sin\varphi\\
			-\frac{k}{\kappa(k)} & -k_0\cos\varphi
		\end{pmatrix}
	\end{align*}
	and hence the determinant \eqref{eq:jacobian}.
	As it is shown below in \autoref{lem:Card} the Banach indicatrix $\Card(T^{-1}(\cdot))$ is finite almost everywhere and vanishes nowhere in $T(\mathcal{U})$. For this reason,
 	we can write
	\[
	\int_{T(\mathcal{U})}\mathcal{F}f(\mathbf{y})e^{i\mathbf{y}\cdot\mathbf{r}}d\mathbf{y}=\int_{T(\mathcal{U})}\mathcal{F}f(\mathbf{y})e^{i\mathbf{y}\cdot\mathbf{r}}\frac{\operatorname{Card}(T^{-1}(\mathbf{y}))}{\operatorname{Card}(T^{-1}(\mathbf{y}))}d\mathbf{y}.
	\]
	Then, applying a change of variables according to \cite[theorem 5.8.30]{Bog07} for the integrand
	\[\mathbf{y}\mapsto \mathcal{F}f(\mathbf{y})e^{i\mathbf{y}\cdot\mathbf{r}}/\operatorname{Card}(T^{-1}(\mathbf{y}))\]
	gives 
	\begin{equation*}
			\int_{T(\mathcal{U})}\mathcal{F}f(\mathbf{y})e^{i\mathbf{y}\cdot\mathbf{r}}d\mathbf{y} = \int_{\mathcal{U}}\mathcal{F}f(T(k,\varphi))e^{iT(k,\varphi)\cdot\mathbf{r}}\frac{\big\vert \det(\nabla T(k,\varphi))\big\vert }{\operatorname{Card}(T^{-1}(T(k,\varphi)))}d(k,\varphi).
	\end{equation*}
\end{proof}
Finally, using this integral transform in \autoref{eq:f_bp} together with the data in $k$-space we obtain for $\mathbf{r}\in\R^2$ the \textit{backpropagated} solution
	\begin{equation}\label{eq:approxfbp}
			\boxed{f(\mathbf{r}) \approx f^{\text{bp}}_N(\mathbf{r})\coloneqq \frac{1}{2\pi}\int_{\mathcal{U}}(\mathcal{A}^\dagger_Nm)(k,\varphi)e^{i T(k,\varphi)\cdot\mathbf{r}}\times\frac{\abs{\det(\nabla T(k,\varphi))} }{\operatorname{Card}(T^{-1}(T(k,\varphi)))}d(k,\varphi)}
	\end{equation}
	 of our reconstruction problem, where the Jacobian determinant is given via \autoref{eq:jacobian}. Hence, in the context of \autoref{eq:approxfbp}, the term \textit{backpropagation} refers to the computation of the low-pass filtered Fourier inversion using the available Fourier data $\mathcal{A}^\dagger_Nm$, which must first be extracted from the original data $m$ using TSVD. To actually calculate \autoref{eq:approxfbp} it remains to determine the Banach indicatrix $\Card(T^{-1}(\mathbf{y}))$ for $\mathbf{y}\in T(\mathcal{U})$.
	 
\begin{lemma}\label{lem:Card} 
For almost every $(k,\varphi) \in \mathcal{U}$
\begin{align*}
	\Card \left( T^{-1}(T(k,\varphi)) \right) =
	\begin{cases}
		2,	& -\pi \leq \varphi < 0,\\
		1,	& 0 \le \varphi < \pi.
	\end{cases}
\end{align*}
\end{lemma}
\begin{proof}
Let $\mathbf{y}\in T(\mathcal U)$, i.e. $\mathbf{y}=\mathbf{h}(k)-k_0\mathbf{s}(\varphi)$ for a $(k,\varphi) \in \mathcal{U}$. 
Note that $ k_0\mathbf{s}(\varphi)\in k_0\mathbb{S}^1$, while $\mathbf{h}(k) = (k,\kappa(k))^\intercal \in k_0 \mathbb{S}_+^1=\{\mathbf{z}\in k_0\mathbb{S}^1 : z_2>0\}$. Therefore, the problem of determining $\Card(T^{-1}(\mathbf{y}))$ consists in finding all $(\mathbf{a},\mathbf{b}) \in k_0 \mathbb{S}_+^1 \times k_0\mathbb{S}^1$ such that $\mathbf{y}=\mathbf{a}-\mathbf{b}$.

We note first that $\Card(T^{-1}(0)) = +\infty$ and that $T^{-1}(0) \subset \mathcal{U}$ is a null set. Therefore, assume now that $\mathbf{y} \neq 0$. Then, with $\mathbf{y}^\perp=(y_2,-y_1)^\intercal$, we can write
	\begin{align*}
		\mathbf{a}= \alpha_0\frac{\mathbf{y}}{\norm{\mathbf{y}}}+\alpha_1\frac{\mathbf{y}^\perp}{\norm{\mathbf{y}}},\qquad
		\mathbf{b}= \beta_0\frac{\mathbf{y}}{\norm{\mathbf{y}}}+\beta_1\frac{\mathbf{y}^\perp}{\norm{\mathbf{y}}},
	\end{align*}
	where $a_2>0$ and the following conditions 
	\begin{align*}
		\alpha_0-\beta_0 =\norm{\mathbf{y}},\\
		\alpha_1-\beta_1= 0,\\
		 \norm{\mathbf{a} }^2=\alpha_0^2+\alpha_1^2=k_0^2,\\
		 \norm{\mathbf{b}} ^2 =\beta_0^2+\beta_1^2= k_0^2
	\end{align*}
	have to hold. Solving this system of equations gives
	\begin{align*}
		\alpha_0 = -\beta_0 =\frac{\norm{\mathbf y}}{2}, \qquad \alpha_1 = \beta_1 = \pm\sqrt{k_0^2-\frac{\norm{\mathbf{y}}^2}{4}}.
	\end{align*}
	Consequently, we have the two potential solutions $(\mathbf a^+, \mathbf b^+)$ and $(\mathbf a^-,\mathbf b^-)$ defined by
	\begin{align}\label{eq:ab}
		\mathbf a^\pm = \frac{\mathbf y}{2}\pm\sqrt{k_0^2-\frac{\norm{\mathbf{y}}^2}{4}}\frac{\mathbf y^\perp}{\norm{\mathbf y}}, \qquad
		\mathbf b^\pm = -\frac{\mathbf y}{2}\pm\sqrt{k_0^2-\frac{\norm{\mathbf{y}}^2}{4}}\frac{\mathbf y^\perp}{\norm{\mathbf y}}
	\end{align}
	so that $\Card \left( T^{-1}(\mathbf y) \right) \le 2$ for $\mathbf y \neq 0.$ First, we note that $\mathbf a^+ = \mathbf a^-$ and $\mathbf b^+ = \mathbf b^-$ if $ \norm{\mathbf y}  = 2k_0,$ which means $\Card \left( T^{-1}(\mathbf y) \right) = 1$ in this case.
	
	Next, we have to check the condition $a^\pm_2>0$, i.e.\ $\mathbf{a}^\pm \in k_0 \mathbb{S}^1_+$. It follows from \autoref{eq:ab} that
	\begin{align*}
		a^\pm_2 > 0 \quad \Leftrightarrow \quad \frac{y_2}{2} \mp \sqrt{k_0^2-\frac{ \norm{\mathbf{y}}^2}{4}}\frac{y_1}{ \norm{\mathbf{y}}}>0,
	\end{align*}
	and after some rearrangements
	\begin{align}
		a^+_2 > 0 \quad  &\Leftrightarrow \quad \mathbf y \in A_1 \coloneqq \left\lbrace \mathbf{y}\in T(\mathcal U):\ y_2 > \sqrt{k_0^2-(y_1-k_0)^2},\text{ if } y_1 > 0 	\right\rbrace, \label{eq:A1}  \\
		a^-_2 > 0 \quad  &\Leftrightarrow \quad \mathbf y \in A_2 \coloneqq \left\lbrace \mathbf{y}\in T(\mathcal U):\ y_2 > \sqrt{k_0^2-(y_1+k_0)^2},\text{ if } y_1 < 0	\right\rbrace.\label{eq:A2}
	\end{align}
	See \autoref{fig:Card} for an illustration of these two sets. Summarizing, for arbitrary $\mathbf y \in T(\mathcal{U})$ we have
	\begin{align*}
		\Card \left( T^{-1}(\mathbf y) \right) =
		\begin{cases}
			+\infty,	& \text{if } \mathbf{y} = 0, \\
			2,		& \text{if } 0 < \norm{\mathbf{y}}  < 2k_0 \text{ and } \mathbf{y} \in A_1 \cap A_2, \\
			1,		& \text{otherwise.}
		\end{cases}
	\end{align*}
	Observing that $A_1 \cap A_2$ agrees with $T((-k_0,k_0) \times (-\pi,0) )$ up to a null set, see also \autoref{fig:frequency_coverage}, finishes the proof.
\end{proof}

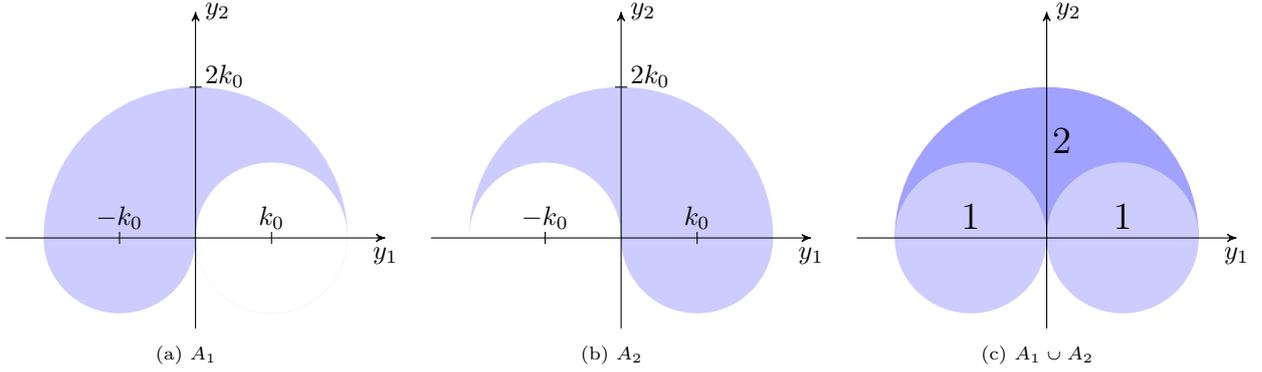
\begin{figure}
	\centering
	\begin{subfigure}[b]{0.3\textwidth}
		\begin{tikzpicture}[scale=1]
			\fill[blue, opacity = 0.2] (2,0) arc (0:180:2cm);
			\fill[blue, opacity = 0.2] (2,0) arc (0:-180:1cm);
			\fill[blue, opacity = 0.2] (0,0) arc (0:-180:1cm);
			\fill[white] (2,0) arc (0:360:1cm);

			\draw[->] (0,-1.2) -- (0,3) node[anchor = west]{$y_2$};
			\draw[->] (-2.5,0) -- (2.5,0) node[anchor = north]{$y_1$};
			
			\draw (1,-0.08) -- (1,0.08);
			\node[anchor = south] at (1,0){$k_0$};
			\draw (-1,-0.08) -- (-1,0.08);
			\node[anchor = south] at (-1,0){$-k_0$};
			\draw (-0.08,2) -- (0.08,2);
			\node[anchor = west] at (0,2.15){$2k_0$};
		\end{tikzpicture}
		\caption{$A_1$}
	\end{subfigure}\hfill
	\begin{subfigure}[b]{0.3\textwidth}
		\begin{tikzpicture}[scale=1]
			\fill[blue, opacity = 0.2] (2,0) arc (0:180:2cm);
			\fill[blue, opacity = 0.2] (2,0) arc (0:-180:1cm);
			\fill[white] (0,0) arc (0:360:1cm);
			
			\draw[->] (0,-1.2) -- (0,3) node[anchor = west]{$y_2$};
			\draw[->] (-2.5,0) -- (2.5,0) node[anchor = north]{$y_1$};
			
			\draw (1,-0.08) -- (1,0.08);
			\node[anchor = south] at (1,0){$k_0$};
			\draw (-1,-0.08) -- (-1,0.08);
			\node[anchor = south] at (-1,0){$-k_0$};
			\draw (-0.08,2) -- (0.08,2);
			\node[anchor = west] at (0,2.15){$2k_0$};
		\end{tikzpicture}
		\caption{$A_2$}
	\end{subfigure}\hfill
	\begin{subfigure}[b]{0.3\textwidth}
		\begin{tikzpicture}[scale=1]
			\fill[blue, opacity = 0.1] (2,0) arc (0:180:2cm);
			\fill[blue, opacity = 0.3] (2,0) arc (0:180:2cm);
			\fill[white] (2,0) arc (0:360:1cm);
			\fill[white] (0,0) arc (0:360:1cm);
			
			\fill[blue, opacity = 0.2] (2,0) arc (0:360:1cm);
			\fill[blue, opacity = 0.2] (0,0) arc (0:360:1cm);
			
			\draw[->] (0,-1.2) -- (0,3) node[anchor = west]{$y_2$};
			\draw[->] (-2.5,0) -- (2.5,0) node[anchor = north]{$y_1$};
			
			
			\node[anchor = south] at (1,0){\Large$1$};
			\node[anchor = south] at (-1,0){\Large$1$};
			\node[anchor = south] at (.2,1){\Large$2$};
		\end{tikzpicture}
		\caption{$A_1\cup A_2$}
	\end{subfigure}
	\caption{\textbf{Level sets of Banach indicatrix.} (a) and (b) illustrate the the sets $A_1,A_2\subset T(\mathcal{U})$ as defined in \autoref{eq:A1} and \autoref{eq:A2}. (c) shows the union of the sets and the level sets of the Banach indicatrix $\Card(T^{-1} (\mathbf y) )$.}
	\label{fig:Card}
\end{figure}
	\section{Numerical implementation}\label{sec:numerics}
We have now intensively studied DT for a generalized incident field and presented the components needed for a reconstruction of the scattering potential from a set of scattering experiments. The aim of this section is now to transfer this theoretical knowledge into practice by providing the numerical framework that will subsequently be used for numerical experiments. 

As $f$ is assumed to be compactly supported within $\mathcal{B}_{r_s}$, we discretize the object domain for $M\in2\N$ in terms of 
\[
\mathcal R_M \coloneqq \frac{2r_s}{M}\mathcal{I}_{M}^2\subset [-r_s,r_s]^2, 
\]
where
\[
\mathcal{I}_M\coloneqq\left\lbrace -\frac{M}{2}+j: j = 0,\dots,M-1\right\rbrace
\] 
are equidistant grid points. Furthermore, we introduce for $D\in2\N$ the set of plane wave angles
\[S_D\coloneqq \frac{2\pi}{D}\mathcal{I}_D\subset[-\pi,\pi)\]
that is used for modeling the incident wave and herewith the grid
\[
\mathcal{U}_{M,D}\coloneqq\left\lbrace (k,\varphi): \ k\in\frac{2k_0}{M}\mathcal{I}_M,\  k<k_0,  \ \varphi\in S_D\right\rbrace.
\]
In the remainder of this paper, for simplicity, we also use the set $S_D$ to specify the rotation angles of the incident field. The sampling set in $k$-space is specified by
\[
\mathcal{Y}_{M,D} \coloneqq T(\mathcal{U}_{M,D}).
\] 

\subsection{Simulation of measurements}
For numerical experiments, we generate synthetic data by applying the forward model to a known scattering potential $f$. Applying a quadrature on the uniform grid $S_D$ to \autoref{eq:setup1} yields
\begin{equation}\label{eq:discr_meas1}
	m(k,\theta) \approx \frac{2\pi}{D} \sum_{\varphi\in S_D} a(\varphi-\theta) \mathcal{F}f(T(k,\varphi)),\qquad (k,\theta)\in\mathcal{U}_{M,D}.
\end{equation}
Here, the Fourier transform $\mathcal{F}f$ of our phantom $f$ is in general analytically not known. We therefore calculate an approximation via the two-dimensional \emph{nonuniform discrete Fourier transform} (NDFT) \cite[section 7.1]{PloPotSteTas18}
\begin{equation}\label{eq:ndft}
	\mathcal{F}f(\textbf{y}) \approx \frac{1}{2\pi} \left( \frac{2r_s}{M}\right)^2\sum_{\textbf{r}\in\mathcal R_M} f(\textbf{r})e^{-i\textbf{r}\cdot\textbf{y}}= \frac{1}{2\pi} \left( \frac{2r_s}{M}\right)^2\sum_{\textbf{j}\in\mathcal{I}_M^2}f\left( \frac{2r_s}{M}\textbf{j}\right) e^{-i\frac{2r_s}{M}\textbf{j}\cdot\textbf{y}},\quad \textbf{y}\in\mathcal{Y}_{M,D}.
\end{equation} 
Then, by inserting \autoref{eq:discr_meas1} in \autoref{eq:ndft} we obtain for known $f$ with unknown Fourier transform the simulated measurement data 
\begin{align}\label{eq:discr_meas}
	\mathbf{m}(k,\theta) \coloneq \frac{2\pi}{D} \left( \frac{2r_s}{M}\right)^2\sum_{\varphi\in S_D} a(\varphi-\theta) \sum_{\textbf{j}\in\mathcal I_M^2}f\left( \frac{2r_s}{M}\textbf{j}\right) e^{-i\frac{2r_s}{M}\textbf{j}\cdot T(k,\varphi)}, \quad (k,\theta)\in\mathcal{U}_{M,D}
\end{align}
from which we aim to reconstruct $f$. To avoid committing an inverse crime \cite[section 1.2]{KaiSom05}, we ensure that the discretization in the numerical forward simulation differs from the one used in the inversion. Additionally, we disturb the simulated data with white Gaussian noise 
\begin{equation*}
	\mathbf{m}^\delta =\mathbf{m} + \delta\mathcal{N}(0,1),
\end{equation*}
where $\delta>0$ and $\mathcal{N}(0,1)$ denotes the standard Gaussian distribution. 

\subsection{Discrete backpropagation}\label{subsec:discrsvd}
The objective is now to reconstruct the scattering potential $f$ on the grid $\mathcal{R}_M$ using the simulated measurement data from \autoref{eq:discr_meas}. Following \autoref{sec:invprob}, this requires a two-step inversion process: First, we have to perform the  TSVD to extract $\mathcal{F}f$ on $\mathcal Y_{M,D}$. Based on these $k$-space data, the scattering potential can be reconstructed using Fourier inversion. To calculate the coefficients needed for the first inversion step, we apply a discretization of the inner $L^2$-product. For $N\in\N$ and $e_n$, $n=-N,\dots,N$ being the complex exponentials (as defined in \autoref{eq:coeffa}), we obtain
\begin{equation}\label{eq:disca}
	a_n=\frac{2\pi}{D}\sum_{\varphi\in S_D} a(\varphi)\overline{e_n(\varphi)} ,\qquad \ n = -N,\dots,N,
\end{equation}
and
\begin{equation}\label{eq:discm}
	\mathbf{m}_{n}(k) \coloneq  \frac{2\pi}{D}\sum_{\theta\in S_D}\mathbf{m}^\delta(k,\theta)e_n(\theta),\qquad\  n = -N,\dots,N,\ k\in\frac{2k_0}{M}\mathcal{I}_M, \ k<k_0.
\end{equation} 
Then, according to  \autoref{eq:tuncsvd} we approximate via
\begin{equation}\label{eq:discr_pseudo}
	\mathcal{F}f(T(k,\varphi)) \approx \mathbf{g}_N(k,\varphi)\coloneqq\frac{1}{2\pi}\sum_{n=-N}^N\frac{\mathbf{m}_n(k)}{a_n}e_n(\varphi), \qquad (k,\varphi)\in\mathcal{U}_{M,D}
\end{equation} 
the data in $k$-space. Herewith the scattering potential can be determined using \autoref{eq:approxfbp} which turns in the discrete setting into
\begin{equation}\label{eq:discrbp}
	\mathbf{f}_N^{\text{bp}}(\mathbf r)\coloneqq\frac{2k_0}{MD}\sum_{(k,\varphi)\in \mathcal{U}_{M,D}} \mathbf{g}_N(k,\varphi)e^{i T(k,\varphi)\cdot\mathbf r} \times \frac{|\det(\nabla T(k,\varphi))|}{\Card(T^{-1}(T(k,\varphi)))}, \quad \mathbf{r}\in\mathcal{R}_M,
\end{equation} 
where the Jacobian determinant is specified in \autoref{eq:jacobian} and the cardinality can be taken from \autoref{lem:Card}. Finally, the algorithm for reconstructing the scattering potential from the measurements can be summarized as follows:\vspace{0.5cm}\\
\begin{algorithm}[H]
	\caption{\vspace*{0.5cm}Discrete backpropagation of $f$.}\label{alg:full}
	\vspace*{0.3cm}
	\KwInput{Beam profile $a$, measurements $\mathbf m$ on $\mathcal{U}_{M,D}$, wavenumber $k_0>0$, truncation level $N$}
	\KwOutput{Backpropagated solution $\mathbf{f}_N^{\text{bp}}$ on $\mathcal{R}_M$}\vspace*{0.1cm}
	\For{$k\in\frac{2k_0}{M}\mathcal{I}_M$}{
		\For{$n=-N,\dots,N$}{
			Compute $a_n$ via \autoref{eq:disca} and $\mathbf{m}_n(k)$ via \autoref{eq:discm};\\		
		}
		\emph{Step 1:} Calculate $\mathbf{g}_N$ via \autoref{eq:discr_pseudo} to approximate $\mathcal{F}f$ on $\mathcal{Y}_{M,D}$;\\
		\emph{Step 2:} Find $\mathbf{f}^{\text{bp}}_N$ on $\mathcal{R}_M$ by evaluating \autoref{eq:discrbp};
	\vspace*{0.3cm}
	}
\end{algorithm}
\section{Numerical experiments}\label{sec:experiments}
The aim of this section is to substantiate the reconstruction process presented in \autoref{alg:full} by conducting computer simulations. We, therefore, specify a phantom $f$ as illustrated in \autoref{subfig:phantom}, which consists of a circular medium containing two smaller inclusions. In all experiments, we illuminate this test sample using a Gaussian beam profile from \autoref{eq:gaussian_a}. To compare the impact of focusing, we choose both amplitudes, $A=10$ (focused) and $A=80$ (unfocused), as shown in \autoref{fig:gaussian_beams}. We assume that the incident field oscillates at the frequency $k_0=\frac{2\pi}{\lambda}$ with $\lambda=1$ and performs a full rotation around the object. Further, we collect the data at the receiver line $\{\mathbf{r}\in\R^2: \  r_2 = 5\}$. For reconstruction, we use $D=200$ angles of incidence and rotations. Further, the resolution of the object and Fourier domain is chosen as $M=400$. The resulting simulated data, depending on the beam waist $A$, which will be used for the reconstruction in the following, are visualized in \autoref{subfig:meas10} and \autoref{subfig:meas80}.

\begin{figure}[t]
	\centering 
	\captionsetup[subfigure]{width=0.9\linewidth}
	\begin{subfigure}[t]{0.3\textwidth}
		\centering
		\includegraphics[scale=0.3]{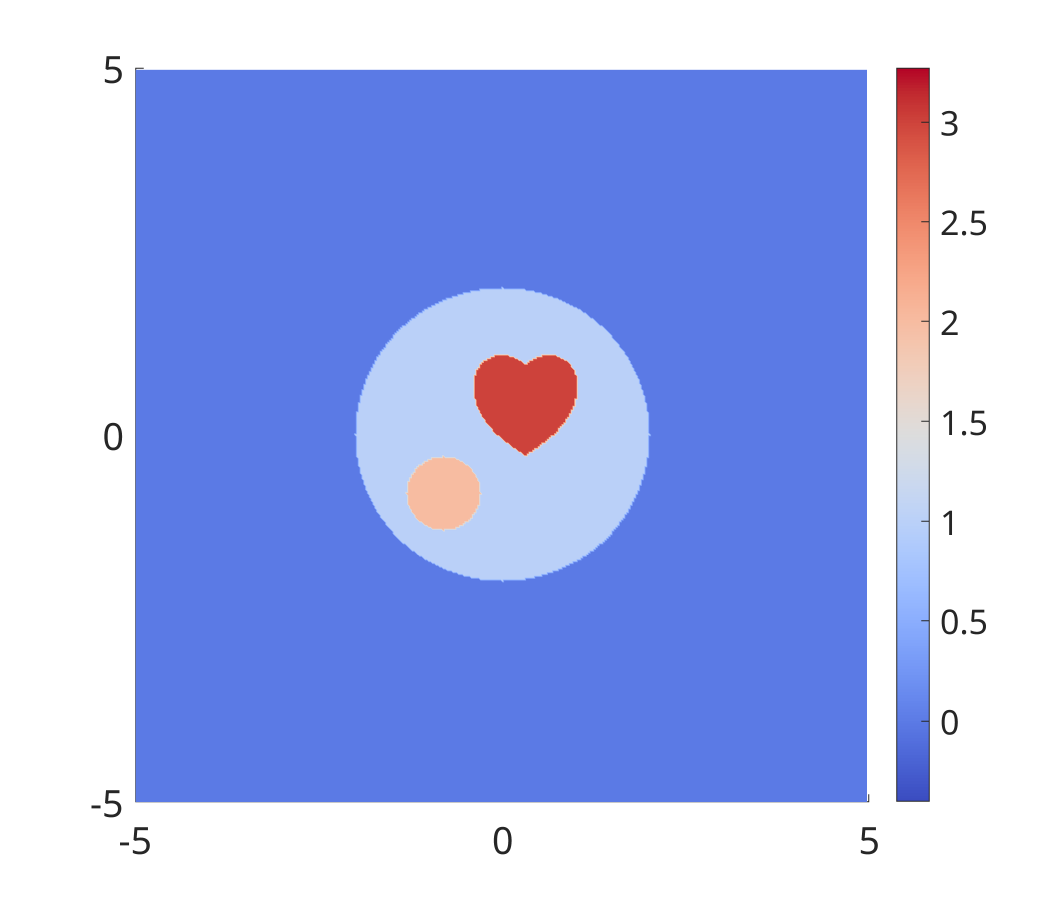}
		\caption{Test sample $f$}
		\label{subfig:phantom}
	\end{subfigure}\hfil
	\begin{subfigure}[t]{0.3\textwidth}
		\centering
		\includegraphics[scale=0.3]{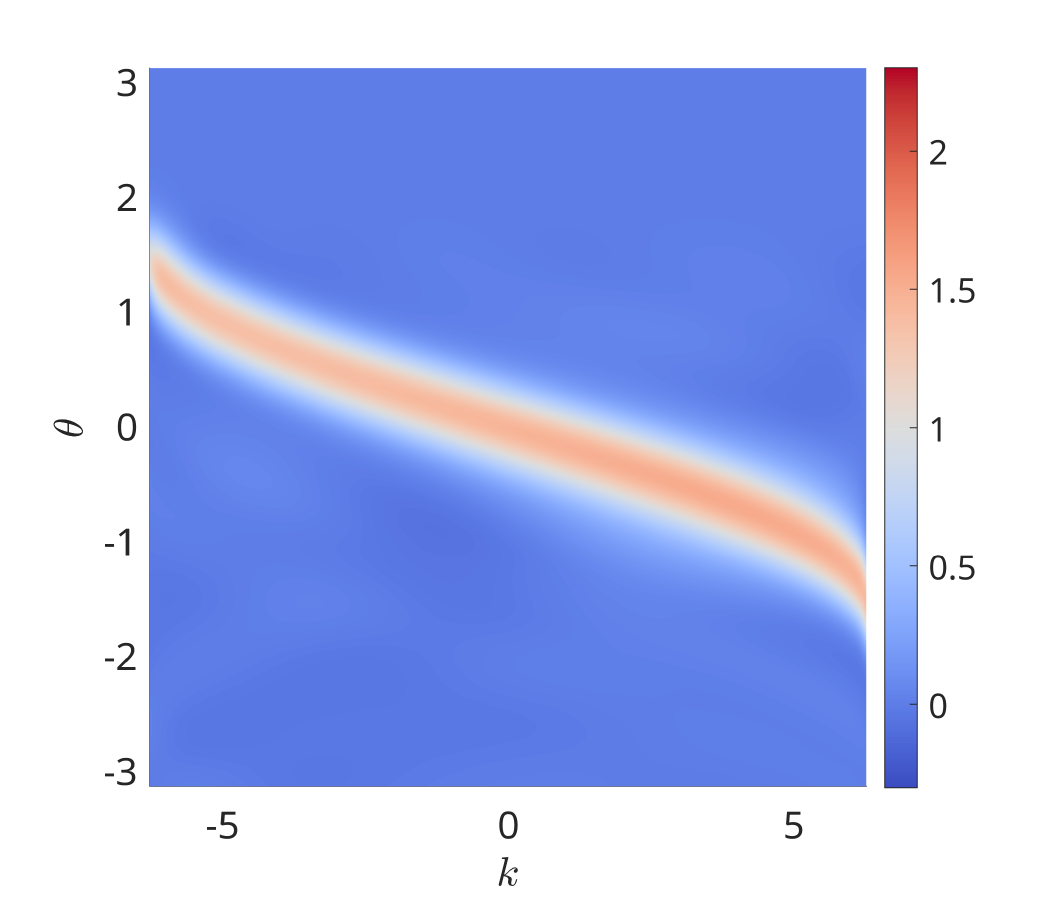}
		\caption{Data $\text{Re}(\mathbf{m})$ for $A=10$}
		\label{subfig:meas10}
	\end{subfigure}\hfil
	\begin{subfigure}[t]{0.3\textwidth}
		\centering
		\includegraphics[scale=0.3]{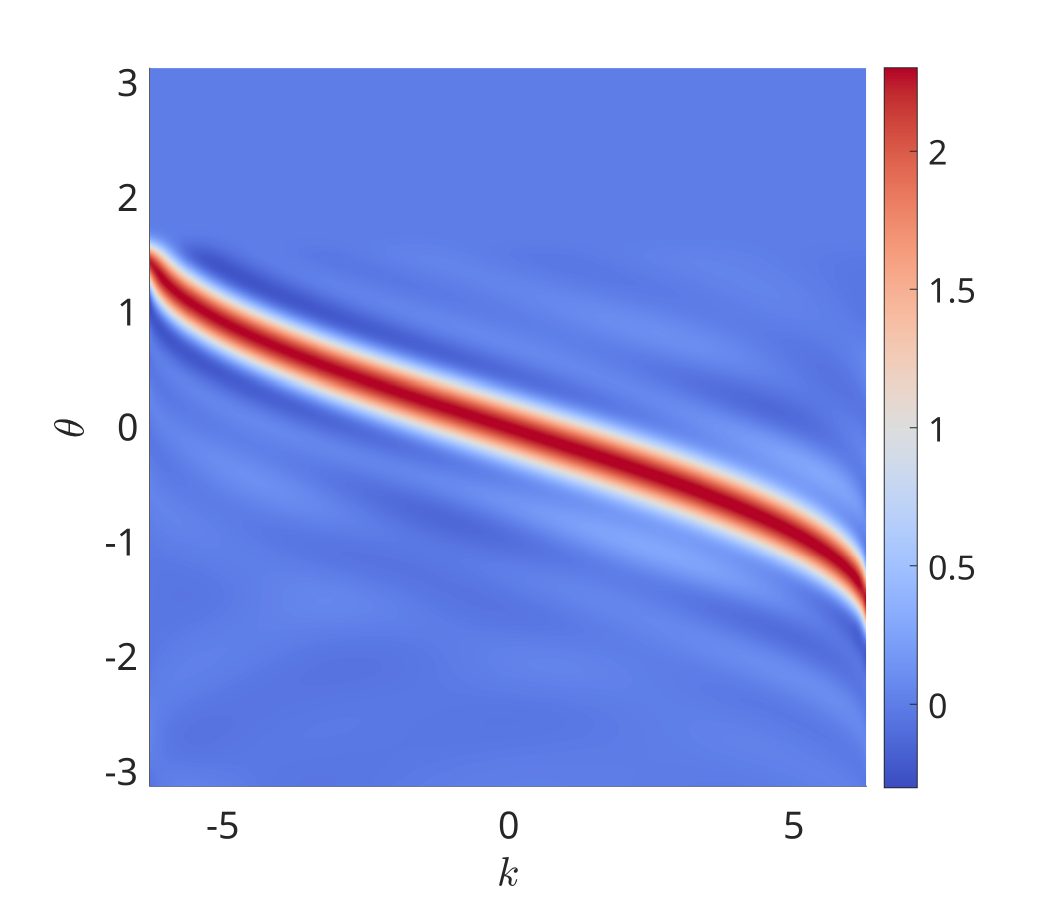}
		\caption{Data $\text{Re}(\mathbf{m})$ for $A=80$}
		\label{subfig:meas80}
	\end{subfigure}
\caption{Illustration of the phantom and the corresponding data according to \autoref{eq:discr_meas} using a Gaussian beam with amplitudes $A\in\{10,80\}$.}
\label{fig:phantom}
\end{figure}
For quality assessment, we introduce the \emph{peak-signal-to-noise-ratio} (PSNR) that compares a reconstruction $\mathbf{v}$ with a ground truth $\mathbf{u}$ by 
\[
\operatorname{PSNR}(\mathbf{u},\mathbf{v})\coloneqq 10\log_{10}\frac{\max_{\textbf{r}\in\mathcal{R}_M}\abs{\mathbf{u}(\mathbf{r})}^2}{M^{-2}\sum_{\textbf{r}\in\mathcal{R}_M} \abs{\mathbf{u}(\mathbf{r})-\mathbf{v}(\mathbf{r})}^2}.
\]
\subsection{Evaluation of the TSVD and choice of the truncation level}\label{subsec:testsvd}
As already pointed out in \autoref{sec:svd}, the parameter $N$ in \autoref{alg:full} influences the accuracy and stability of the first inversion step and, consequently, the backpropagation. In order to select an appropriate $N\in\N$ we start with numerically evaluating the first inversion step separately. To simplify the analysis we fix the frequency $k=0$ and empirically verify that the chosen parameter $N$ is suitable for other $k\in(-k_0,k_0)$. Therefore, without loss of generality, we approximate the function $g \coloneqq \mathcal{F}f(T(0,\cdot))$ from $\mathbf{m}(0,\cdot)$ via
\begin{equation}\label{eq:tsvd_discr}
 \mathbf{g}_N(\varphi)= \sum_{n=-N}^{N}\frac{\mathbf{m}_n(0)}{a_n}e_n(\varphi), \quad \varphi\in S_D.
\end{equation}
To determine a suitable parameter $N\in\N$, we analyze the behavior of the SVD coefficients $\mathbf{m}_n$ and $\mathbf{m}_n/a_n$ by examining the Picard plots for various noise levels, following the approach suggested in \cite[section 3.3]{Han10}.

In \autoref{fig:svd10} a Gaussian beam with $A=10$ was used for illumination. Different noise levels are applied to the data, where the expression ``$X$\% noise'' denotes
\[ 
\frac{\norm{\mathbf{m}^\delta - \mathbf{m}}}{ \norm{\mathbf{m}} } \approx \frac{X}{100}. 
\]
The Picard plots show that the SVD coefficients $\mathbf{m}_n$ initially decay,  leveling off at a plateau for larger values of the index $\abs{n}>12$. As a result, the solution coefficients $\mathbf{m}_n/a_n$ initially decrease, but for $\abs{n}>12$, they begin to increase again, indicating the necessity of truncating the sum in \autoref{eq:tsvd_discr} at level $N=12$. The target function $g$ as well as the regularized solution $\mathbf{g}_{12}$ for different noise levels are depicted below the corresponding Picard plots. For perturbed data, we observe oscillations that occur due to the fact that the singular values $\abs{a_n}$ are already relatively small for $\abs{n}<12$. Nonetheless, choosing a smaller parameter $N$ would also lead to an erroneous result since in this case coefficients contributing to the solution would be neglected. This trade-off between accuracy and robustness against perturbations is a characteristic feature of solutions to ill-posed problems.

Interestingly, the observations change in \autoref{fig:svd_80}, where we used the same setting except for the parameter $A=80$ meaning that the illumination was without focusing. The Picard plots indicate that singular values decay more slowly leading to a noise-robust reconstruction. Therefore, in terms of stability, it appears that focusing is disadvantageous for the full reconstruction of $f$. However, in the case of noisy measurements, the solution coefficients $\mathbf{m}_n/a_n$ also start increasing for $\abs{n}>12$. Therefore, we choose $N=12$ for the beam DT reconstruction in both scenarios: focused and unfocused imaging.
\begin{figure}[t]
	\centering 
	\captionsetup[subfigure]{width=0.9\linewidth}
	\begin{subfigure}[t]{0.3\textwidth}
		\centering
		\includegraphics[scale=0.28]{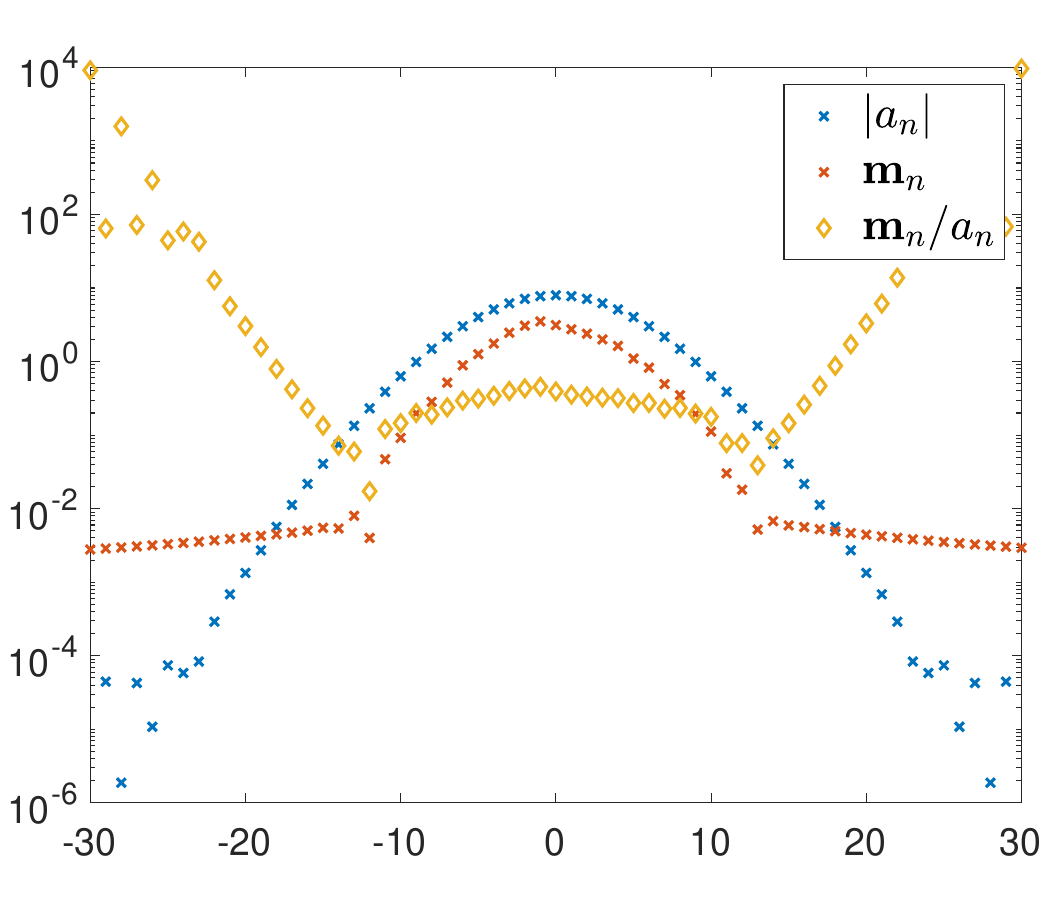}
		\caption{Picard plot without noise}
	\end{subfigure}\hfil
	\begin{subfigure}[t]{0.3\textwidth}
	\centering
	\includegraphics[scale=0.28]{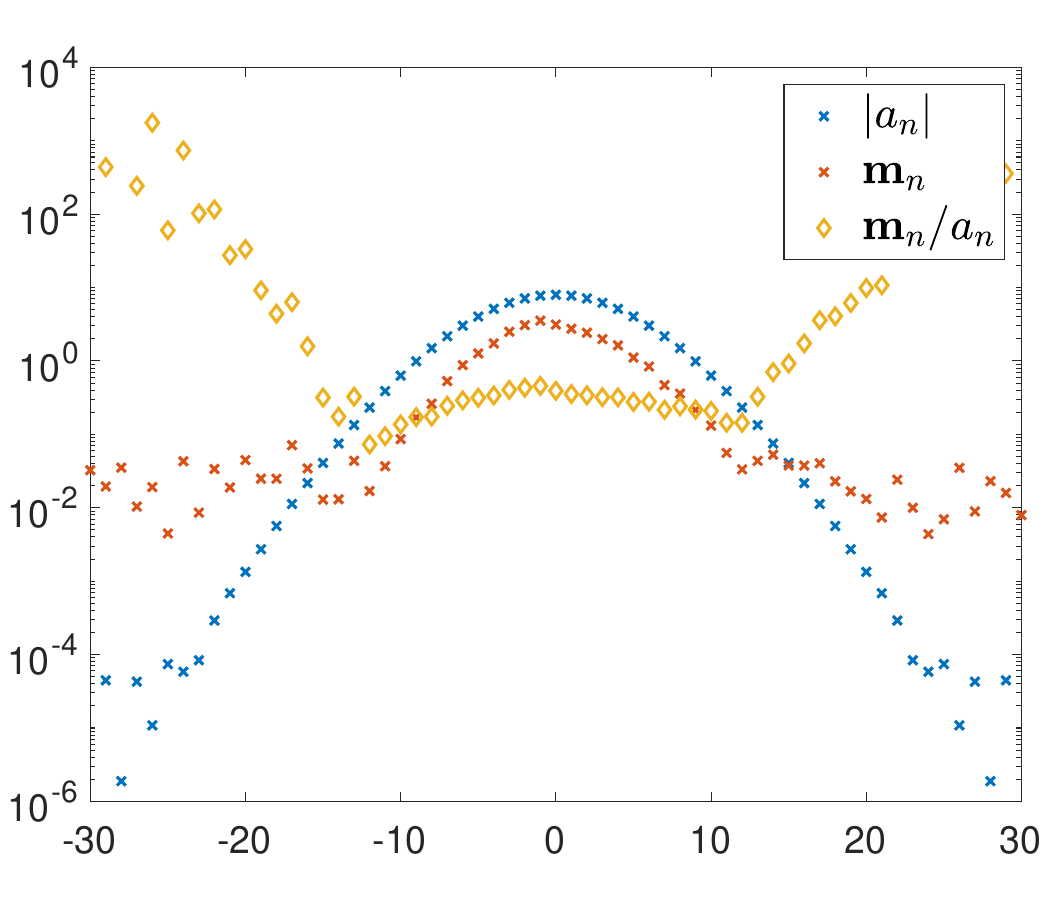}
	\caption{Picard plot using 5\% noise}
\end{subfigure}\hfil
	\begin{subfigure}[t]{0.3\textwidth}
	\centering
	\includegraphics[scale=0.28]{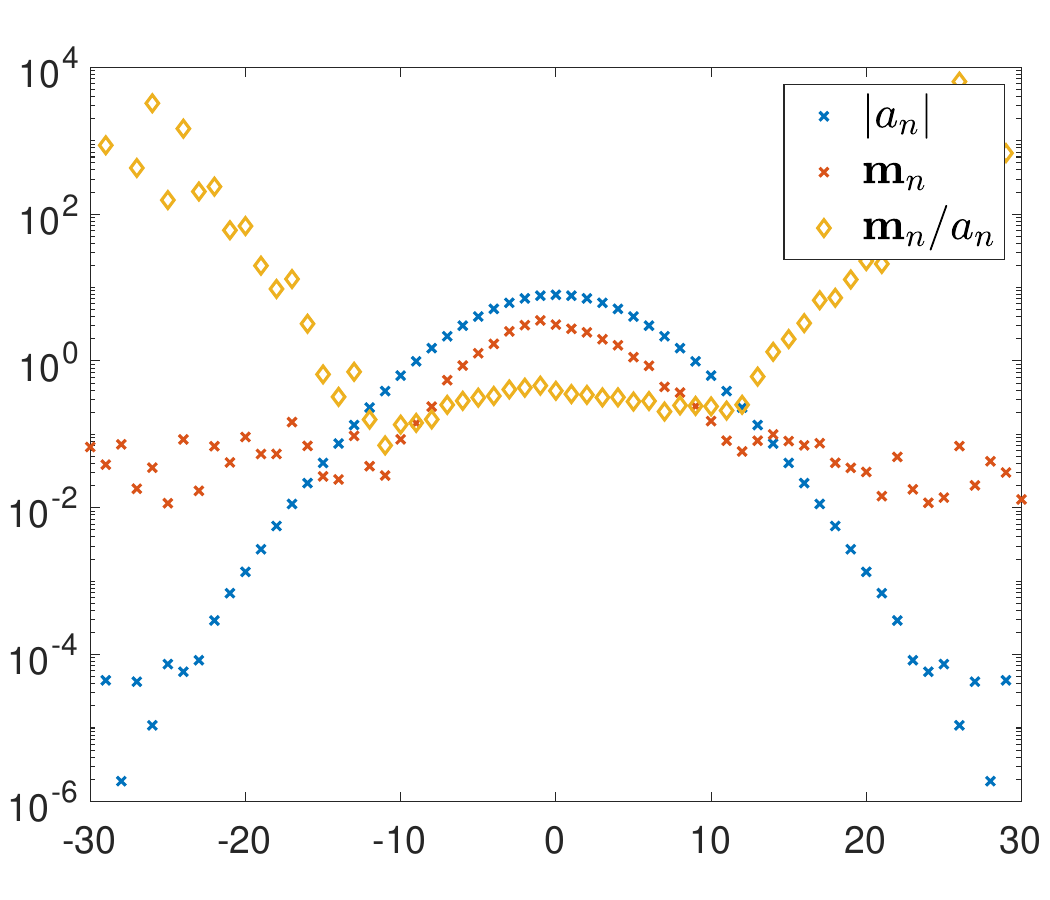}
	\caption{Picard plot using 10\% noise}
\end{subfigure}\\
	\begin{subfigure}[t]{0.3\textwidth}
		\centering
		\includegraphics[scale=0.28]{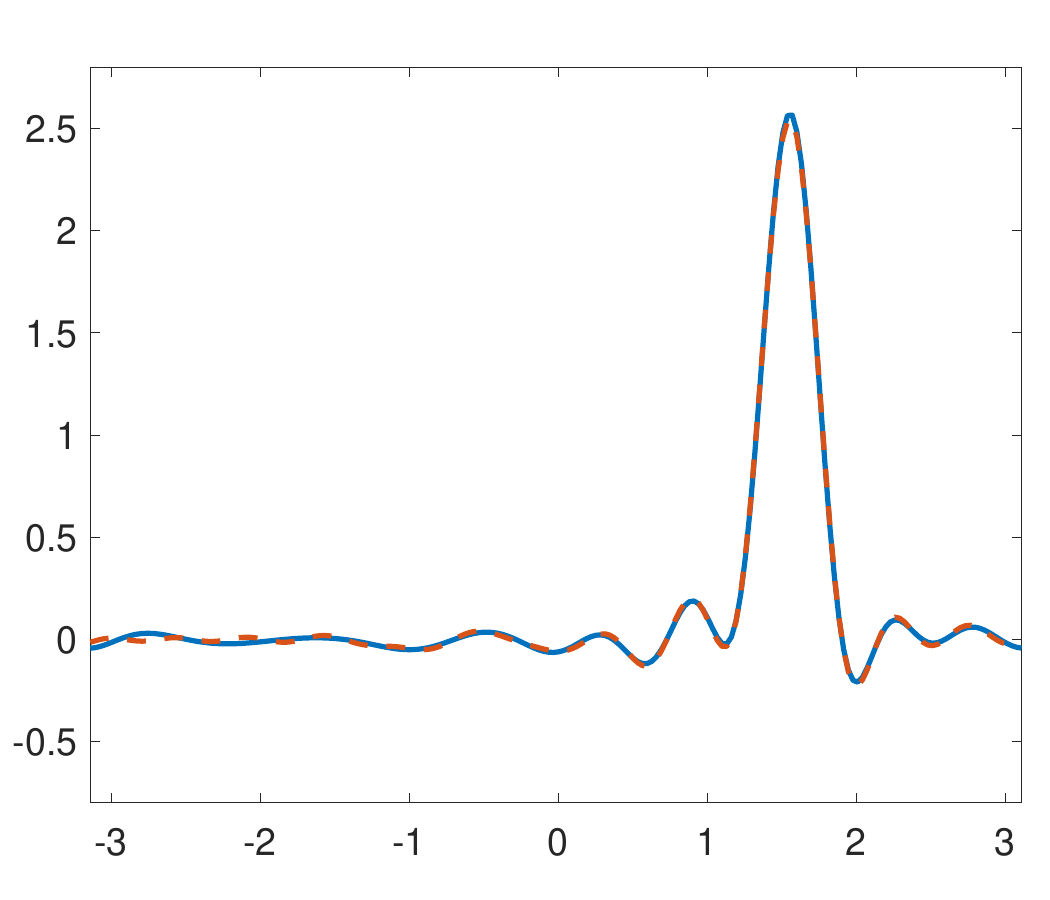}
	\caption{Reconstruction without noise}
	\end{subfigure}\hfil
	\begin{subfigure}[t]{0.3\textwidth}
	\centering
	\includegraphics[scale=0.28]{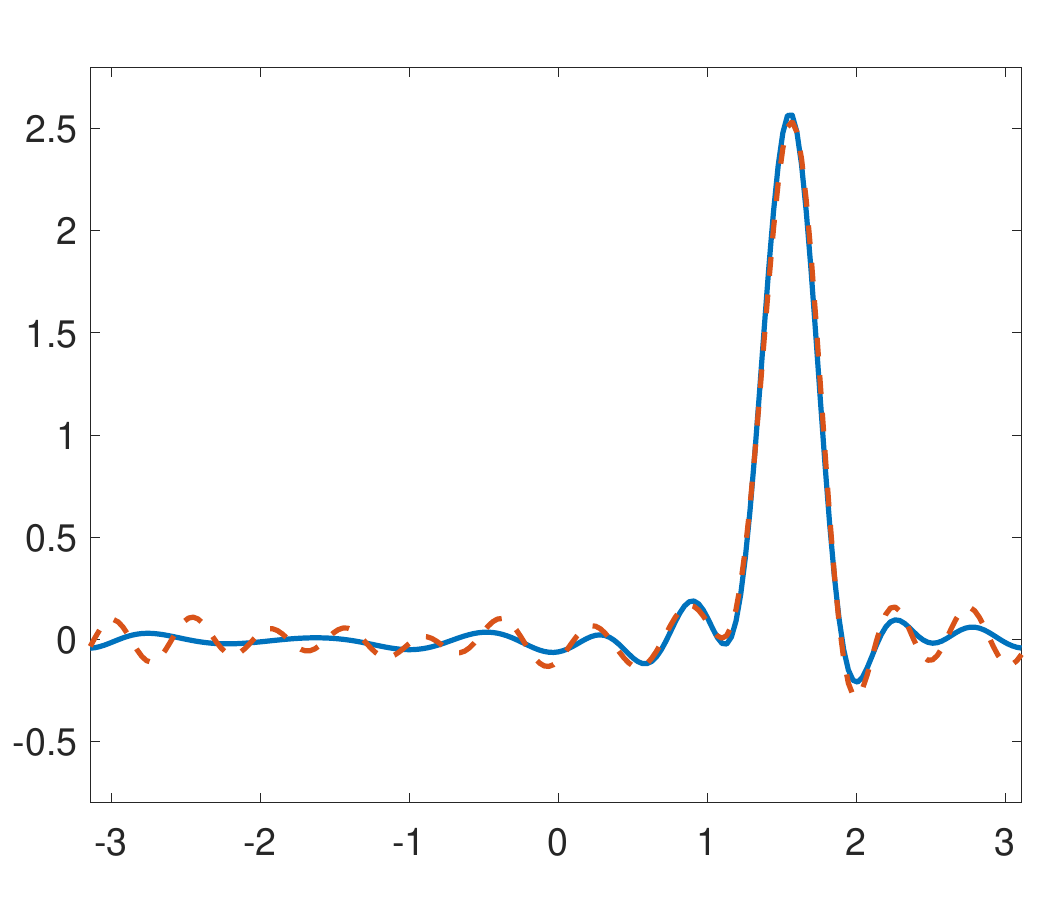}
	\caption{Reconstruction using 5\% noise}
	\end{subfigure}\hfil
	\begin{subfigure}[t]{0.3\textwidth}
	\centering
	\includegraphics[scale=0.28]{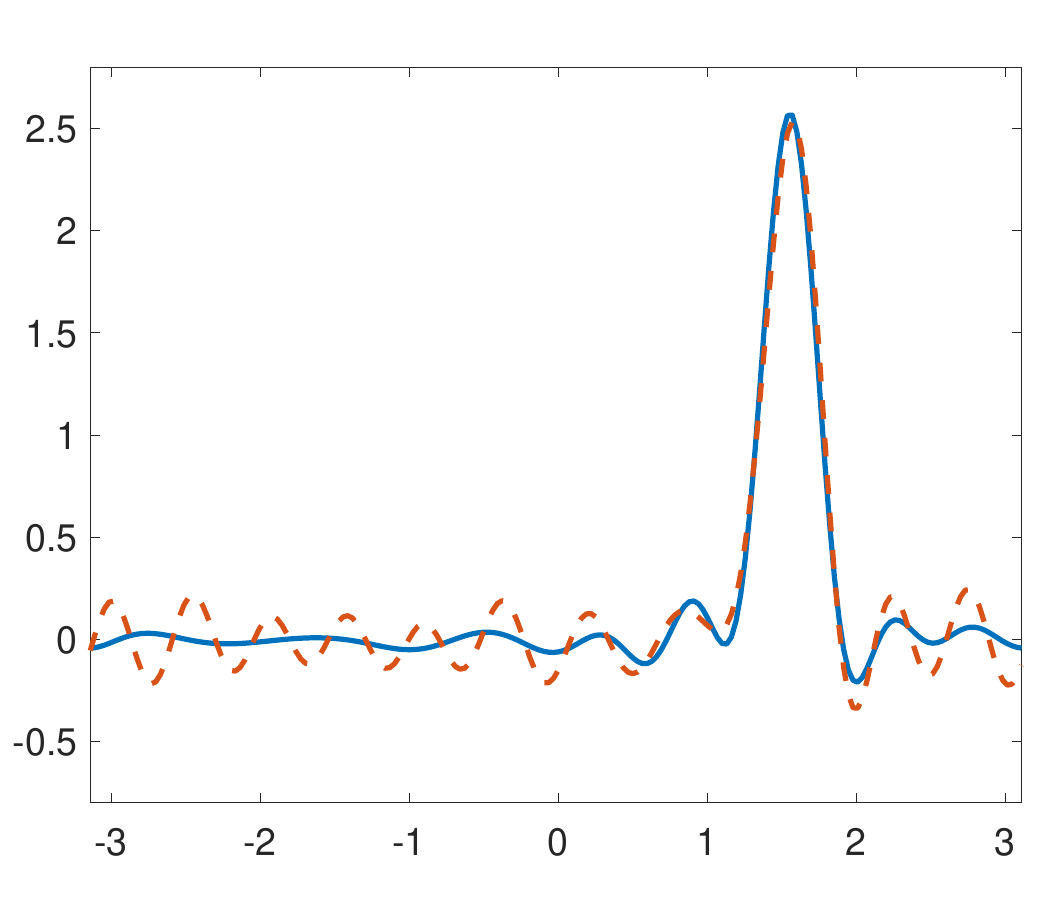}
	\caption{Reconstruction using 10\% noise}
\end{subfigure}
\caption{\textbf{Choice of the truncation level.} Inversion via TSVD, specified in \autoref{eq:tsvd_discr}, using data having different noise levels and the Gaussian integral kernel from \autoref{eq:gaussian_a} with $A=10$. Top: Illustration of the Picard condition indicating that $N=12$ is a suitable regularization parameter for all noise levels. Bottom: The original target function $g$ is plotted in a solid blue line while the reconstruction $\mathbf{g}_{12}$ is shown as a dashed red line. The reconstruction starts oscillating in case of  $5\%$ noise.}
\label{fig:svd10}
\end{figure}

\begin{figure}[t]
	\centering 
	\captionsetup[subfigure]{width=0.9\linewidth}
	\begin{subfigure}[t]{0.3\textwidth}
		\centering
		\includegraphics[scale=0.28]{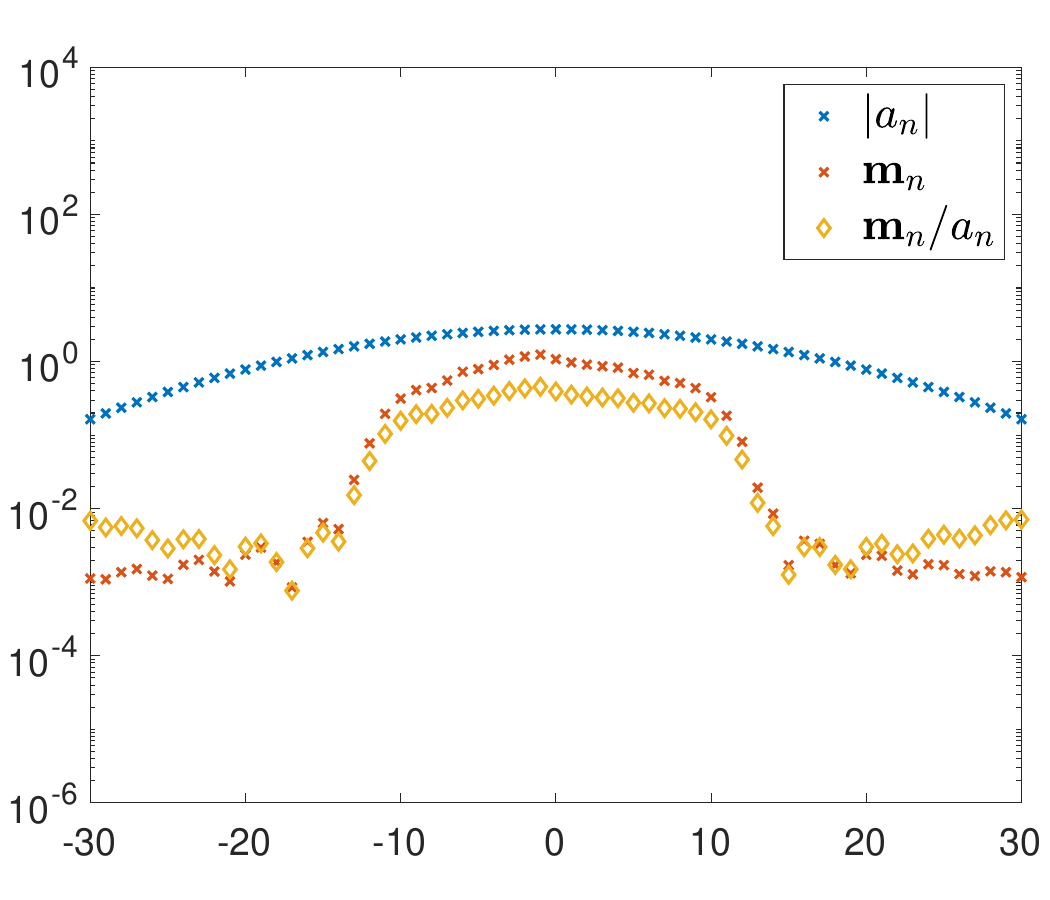}
		\caption{Picard plot without noise}
	\end{subfigure}\hfil
	\begin{subfigure}[t]{0.3\textwidth}
	\centering
	\includegraphics[scale=0.28]{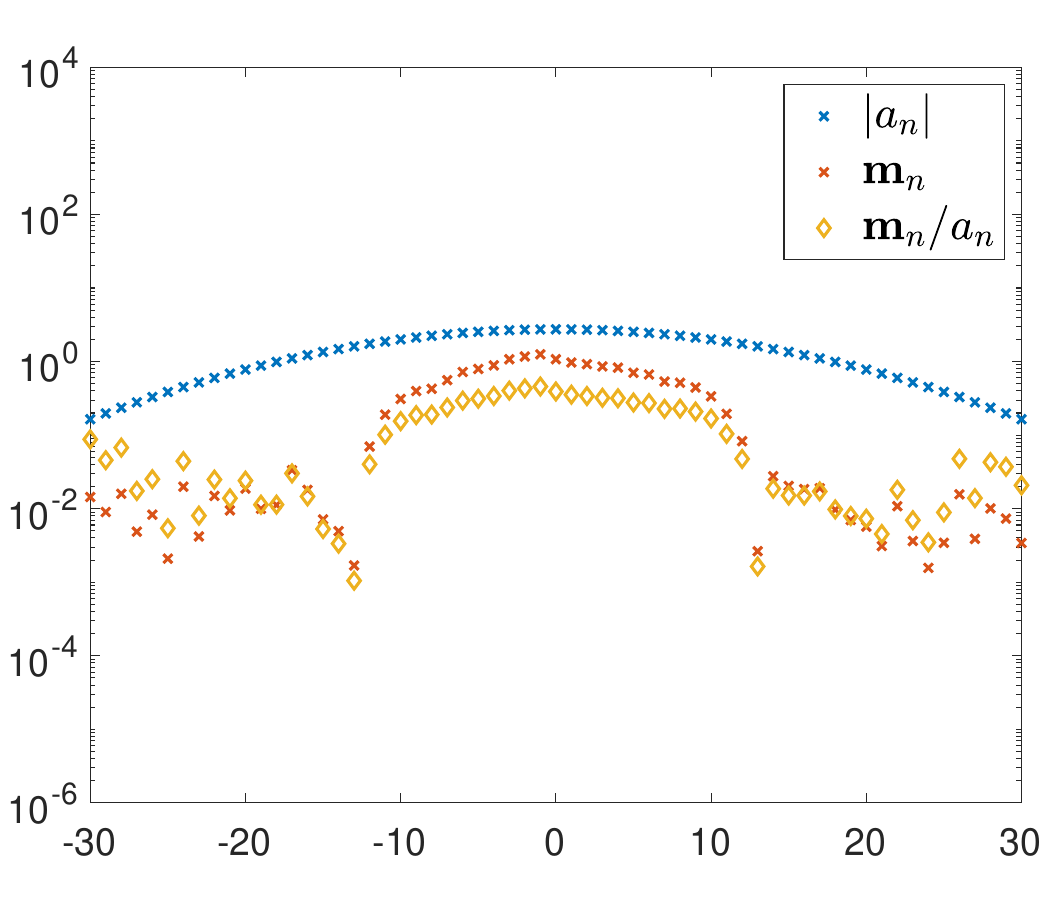}
	\caption{Picard plot using 5\% noise}
\end{subfigure}\hfil
	\begin{subfigure}[t]{0.3\textwidth}
	\centering
	\includegraphics[scale=0.28]{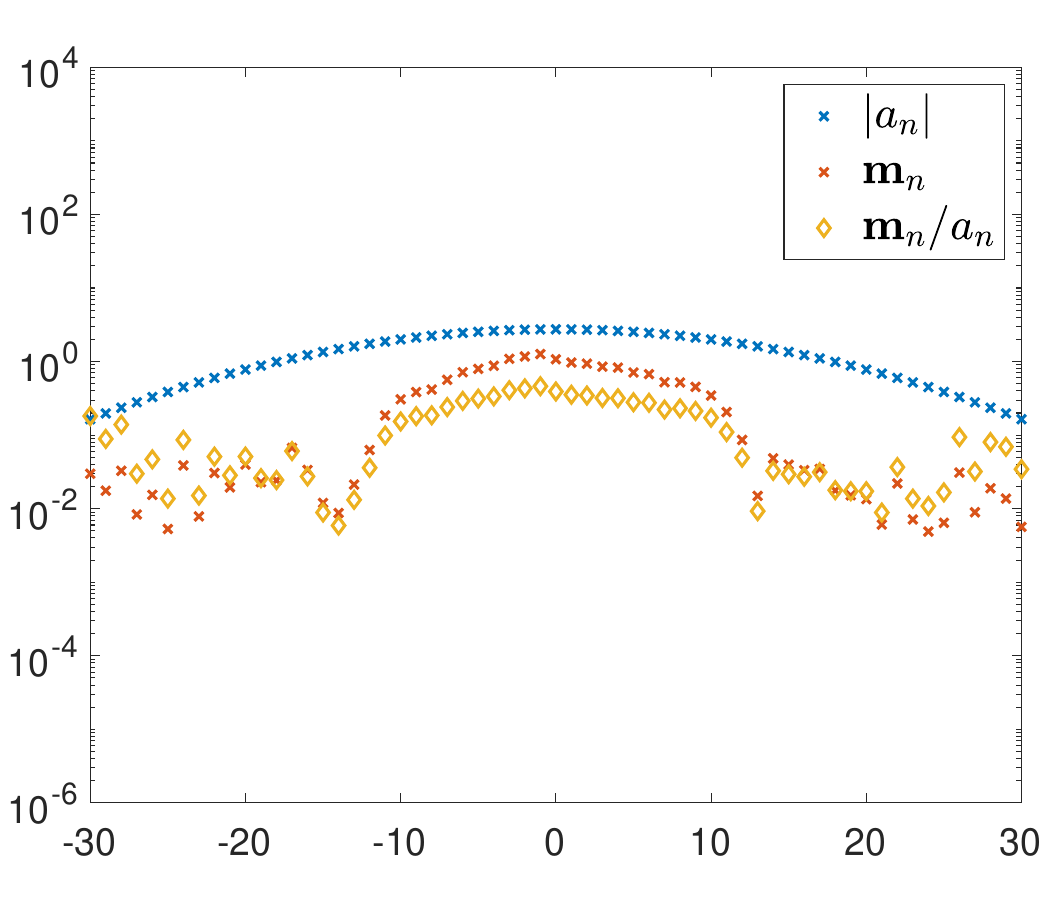}
	\caption{Picard plot using 10\% noise}
\end{subfigure}\\
	\begin{subfigure}[t]{0.3\textwidth}
		\centering
		\includegraphics[scale=0.28]{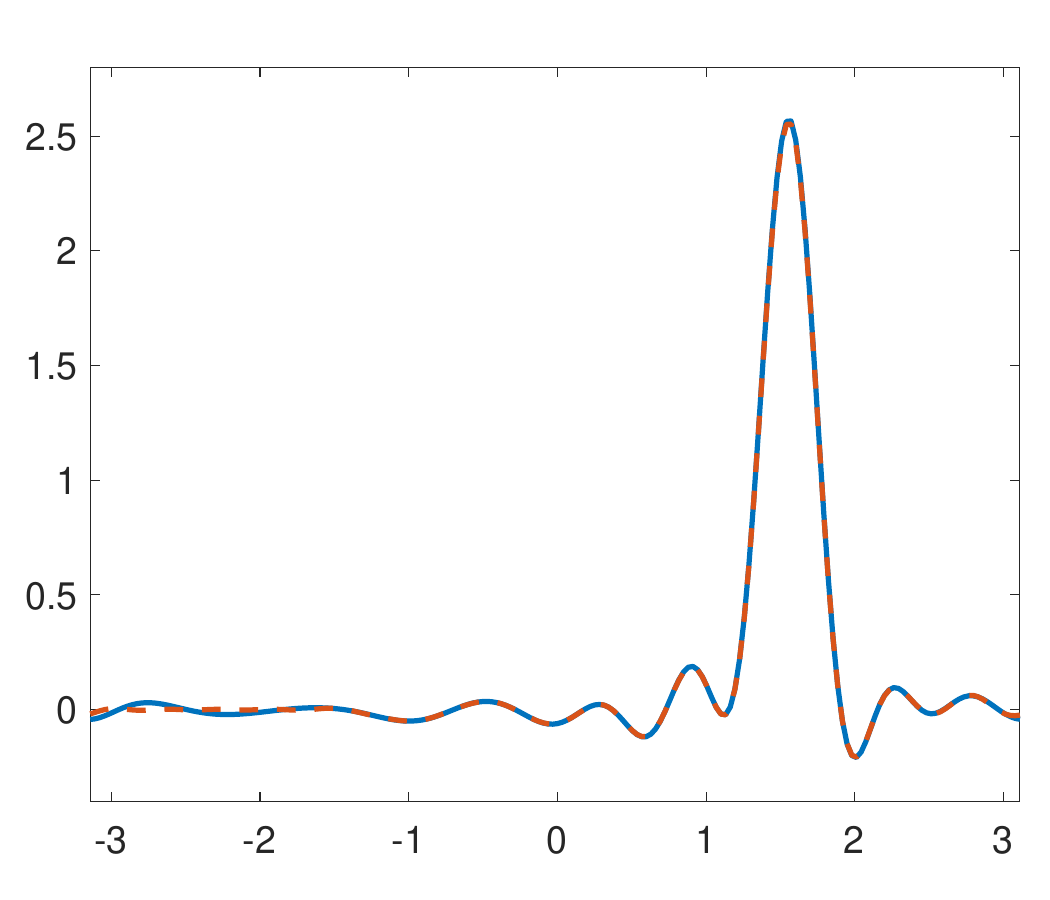}
		\caption{Reconstruction  without noise}
	\end{subfigure}\hfil
\begin{subfigure}[t]{0.3\textwidth}
\centering
\includegraphics[scale=0.28]{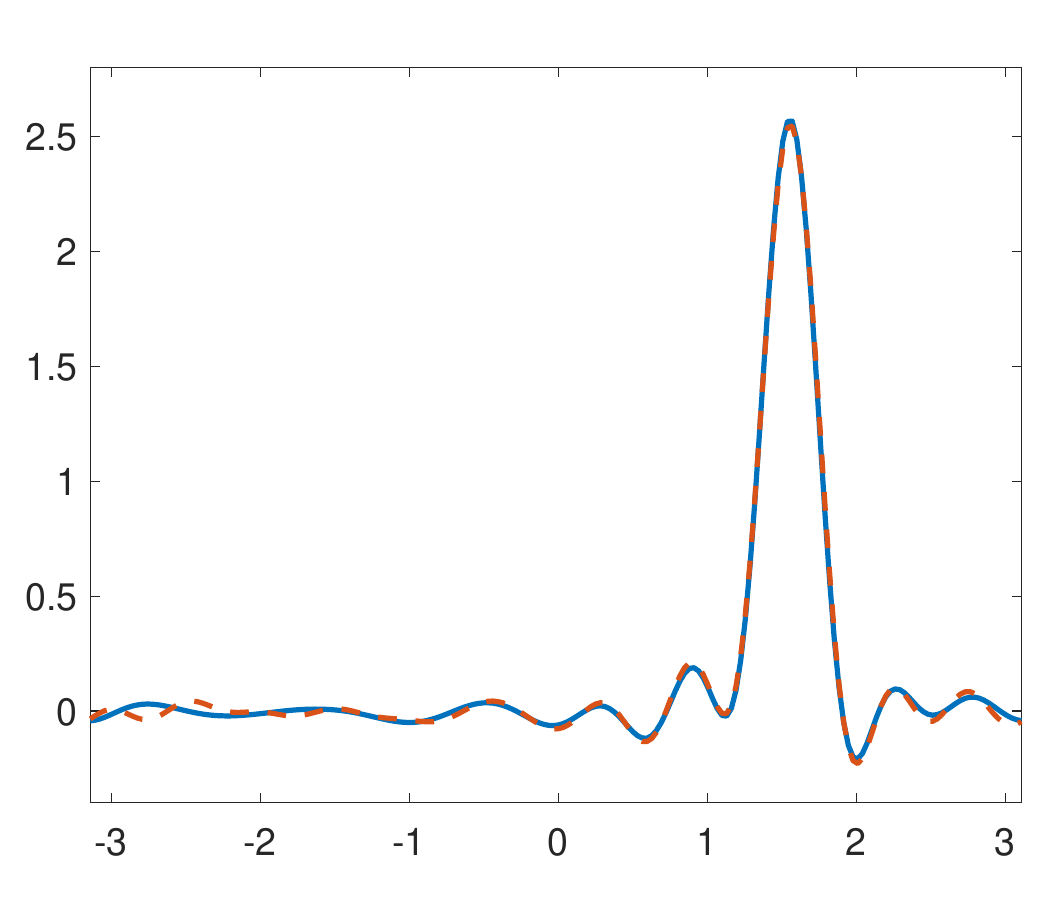}
\caption{Reconstruction  using 5\% noise}
\end{subfigure}\hfil
\begin{subfigure}[t]{0.3\textwidth}
\centering
\includegraphics[scale=0.28]{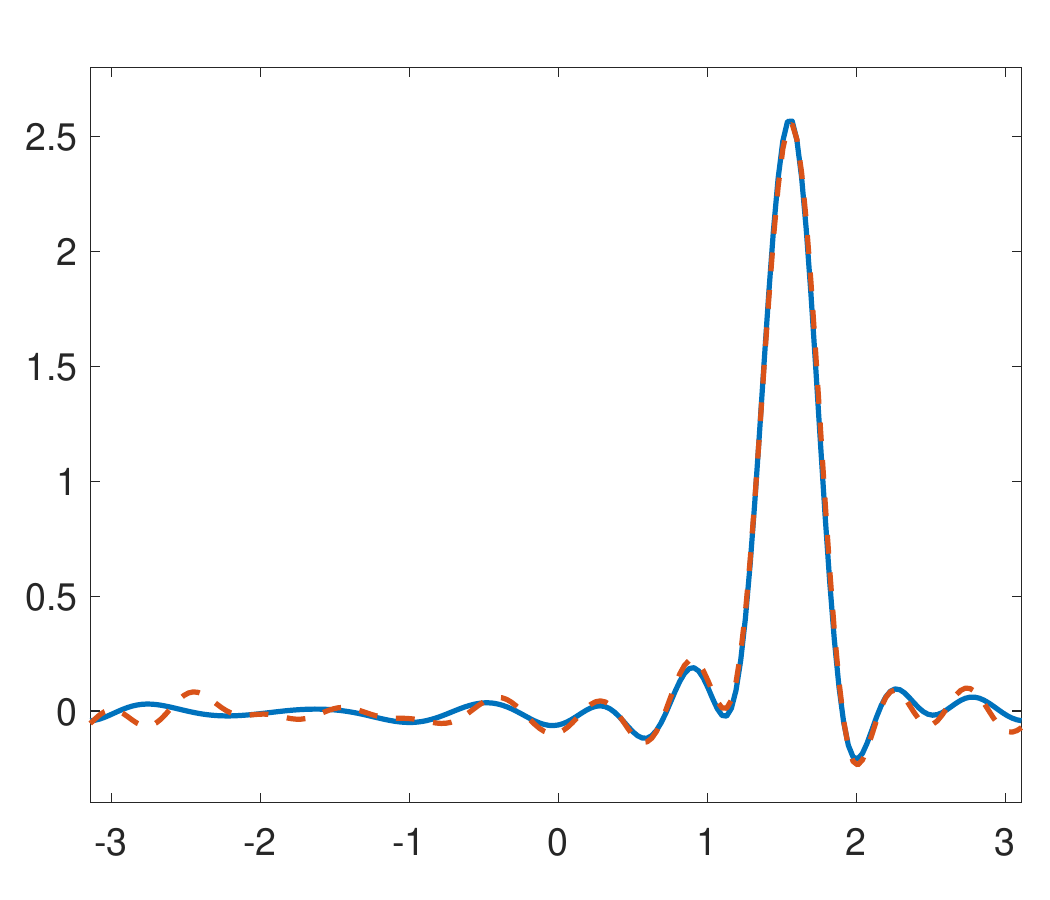}
\caption{Reconstruction  using 10\% noise}
\end{subfigure}
	\caption{\textbf{Choice of the truncation level.} Inversion via TSVD, specified in \autoref{eq:tsvd_discr}, using data having different noise levels and the Gaussian integral kernel from \autoref{eq:gaussian_a} with $A=80$. Top: Illustration of the Picard condition indicating that $N=12$ is a suitable regularization parameter for all noise levels. Bottom: The original target function $g$ is plotted in a solid blue line while the reconstruction $\mathbf{g}_{12}$  is shown as a dashed red line. No oscillations appear in the reconstruction for all noise levels.}
	\label{fig:svd_80}
\end{figure}

\subsection{Evaluation of the backpropagation}
In this section, we illustrate the effectiveness of the reconstruction via backpropagation presented in  \autoref{alg:full}.
We test our method for different noise levels and we use Gaussian incident illumination with both amplitudes $A=10$ and $A=80$. The results for the former are shown in \autoref{fig:fbp_10} while those for the latter are depicted in \autoref{fig:fbp_80}. In both cases, the backpropagated solutions exhibit oscillations due to the low-pass filtered Fourier inversion. This is also known as the \emph{Gibbs phenomenon}, see \cite[section 1.4.3]{PloPotSteTas18}, and consistent with the observations for conventional DT in \cite{FauKirQueSchSet22a,KakSla01}.  We further observe, that the use of noisy measurement data significantly affects the reconstruction quality with focused illumination, see \autoref{fig:fbp_10}. The effect of noisy measurements in the reconstruction is less pronounced when unfocused illumination is used as depicted in \autoref{fig:fbp_80}.
We attribute this behavior to the backpropagation process, which employs the TSVD as a regularization technique. As examined in \autoref{subsec:testsvd}, this approach leads to a noise-robust intermediate solution, particularly in the case of unfocused illumination.
\begin{figure}[t]
	\centering 
	\captionsetup[subfigure]{width=0.9\linewidth}
	\begin{subfigure}[t]{0.33\textwidth}
		\centering
		\includegraphics[scale=0.32]{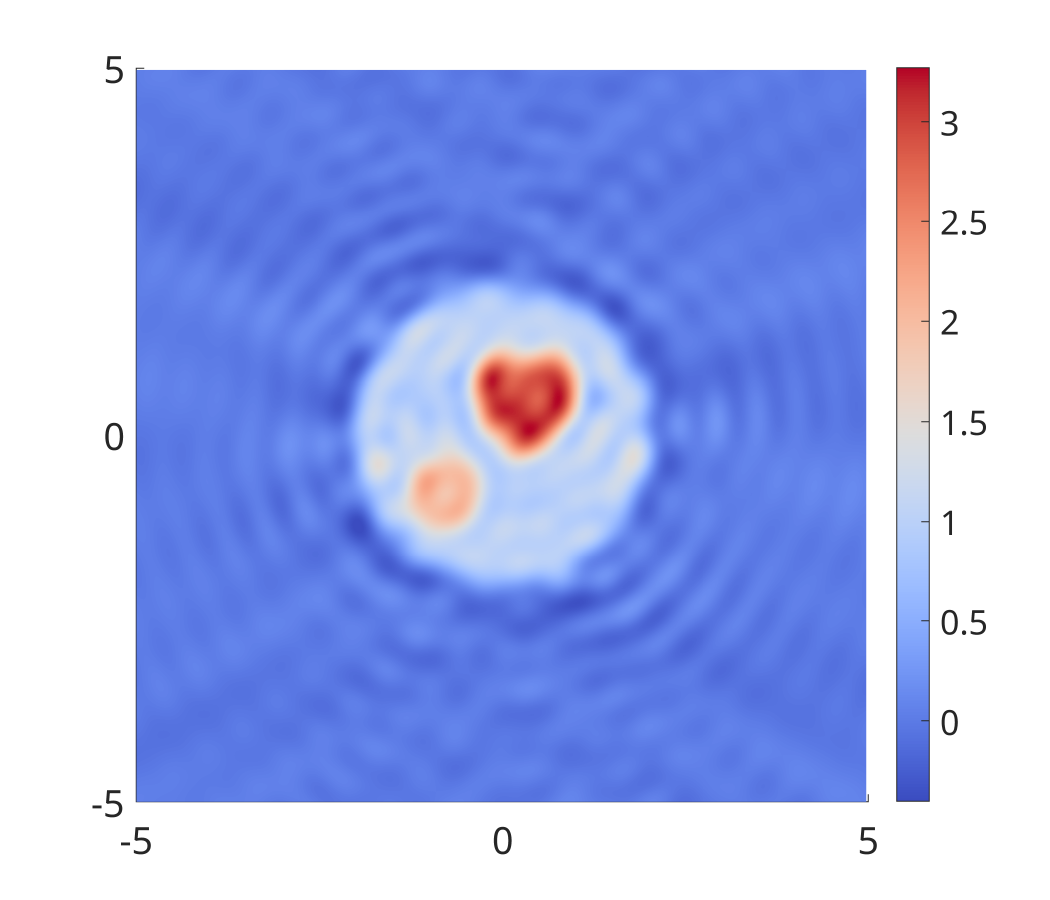}
		\caption{$\mathbf{f}^{\text{bp}}_{12}$ from noiseless data, \\PSNR 27.81, RMSE 0.24, SSIM 0.19}
	\end{subfigure}\hfill 
	\begin{subfigure}[t]{0.33\textwidth}
		\centering
		\includegraphics[scale=0.32]{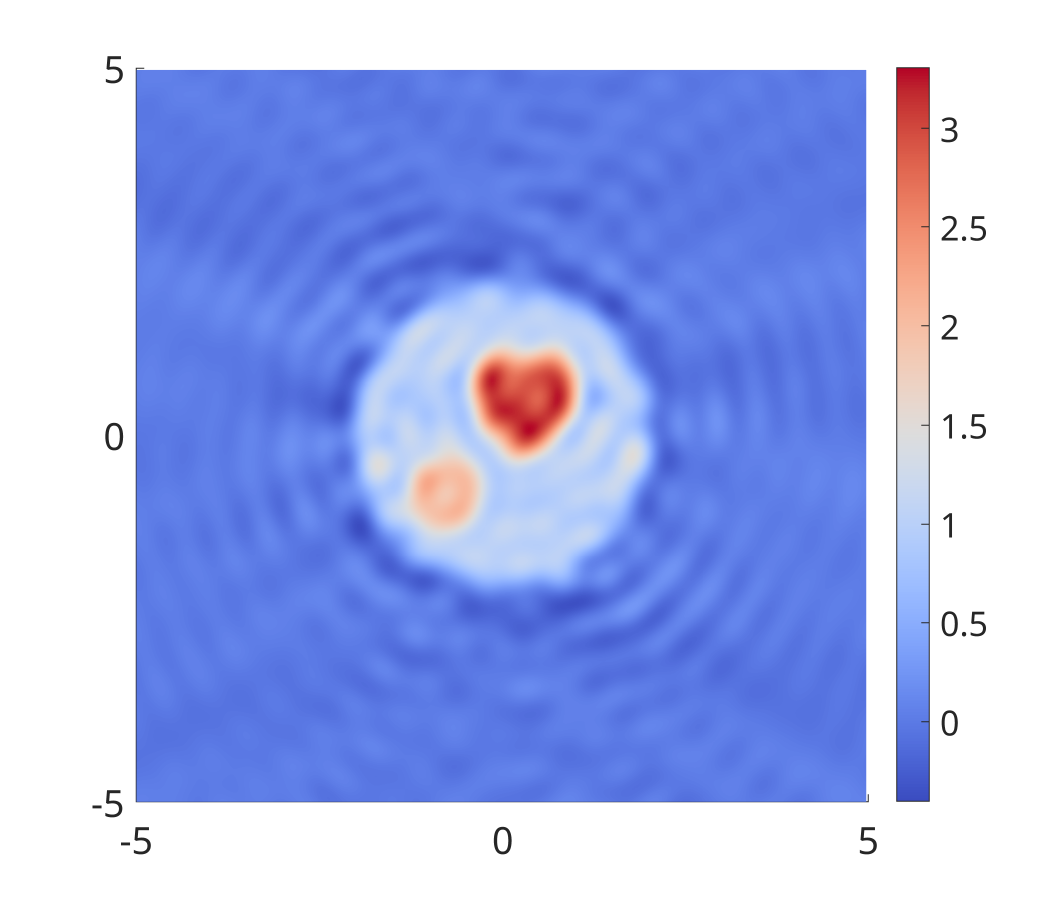}
		\caption{$\mathbf{f}^{\text{bp}}_{12}$ from data with $1\%$ Gaussian noise,\\ PSNR 27.80, RMSE 0.24, SSIM 0.19}
	\end{subfigure}\hfill
	\begin{subfigure}[t]{0.33\textwidth}
	\centering
	\includegraphics[scale=0.32]{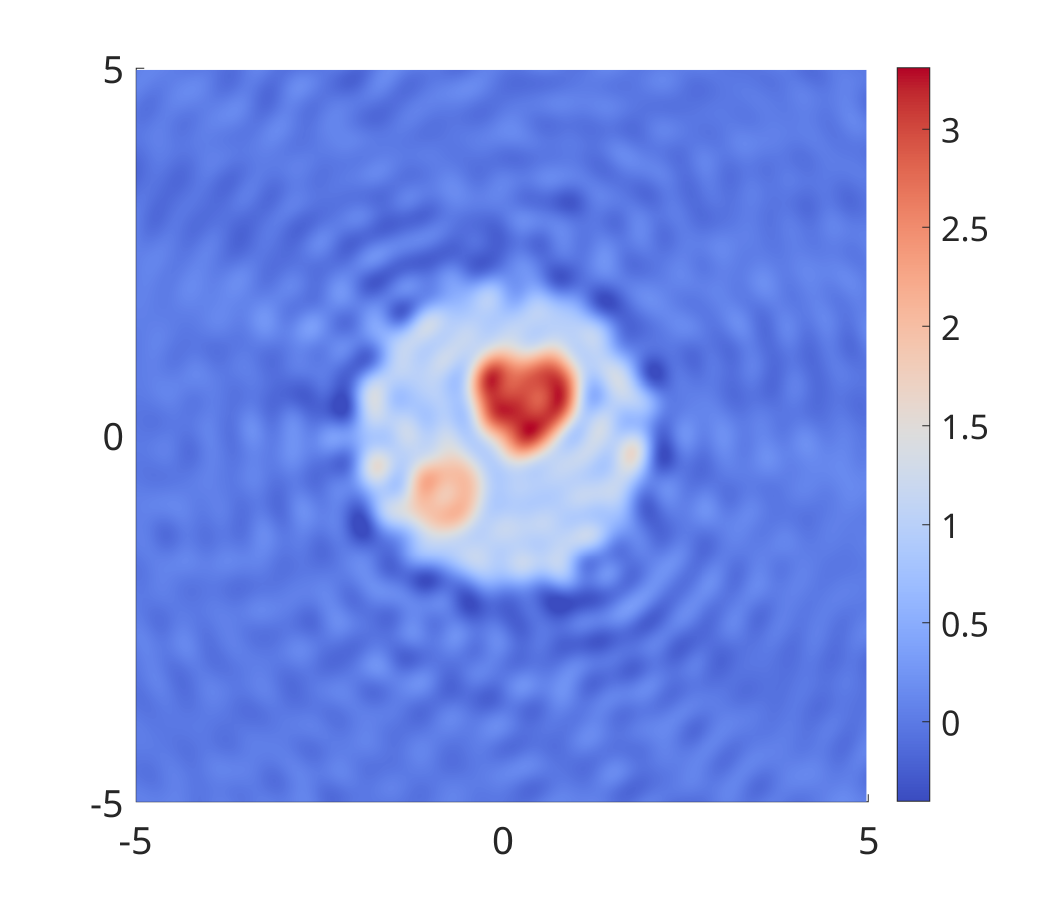}
	\caption{$\mathbf{f}^{\text{bp}}_{12}$ from data with $5\%$ Gaussian noise \\PSNR 26.87, RMSE 0.26, SSIM 0.17}
	\end{subfigure}
	\caption{Reconstruction of the test sample $f$ shown in \autoref{fig:phantom}. For illumination the Gaussian profile from \autoref{eq:gaussian_a} with $A=10$ was used. The reconstruction is carried out using perturbed data with varying noise levels.}
	\label{fig:fbp_10}
\end{figure}
\begin{figure}[t]
	\centering 
	\captionsetup[subfigure]{width=0.9\linewidth}
	\begin{subfigure}[t]{0.33\textwidth}
		\centering
		\includegraphics[scale=0.32]{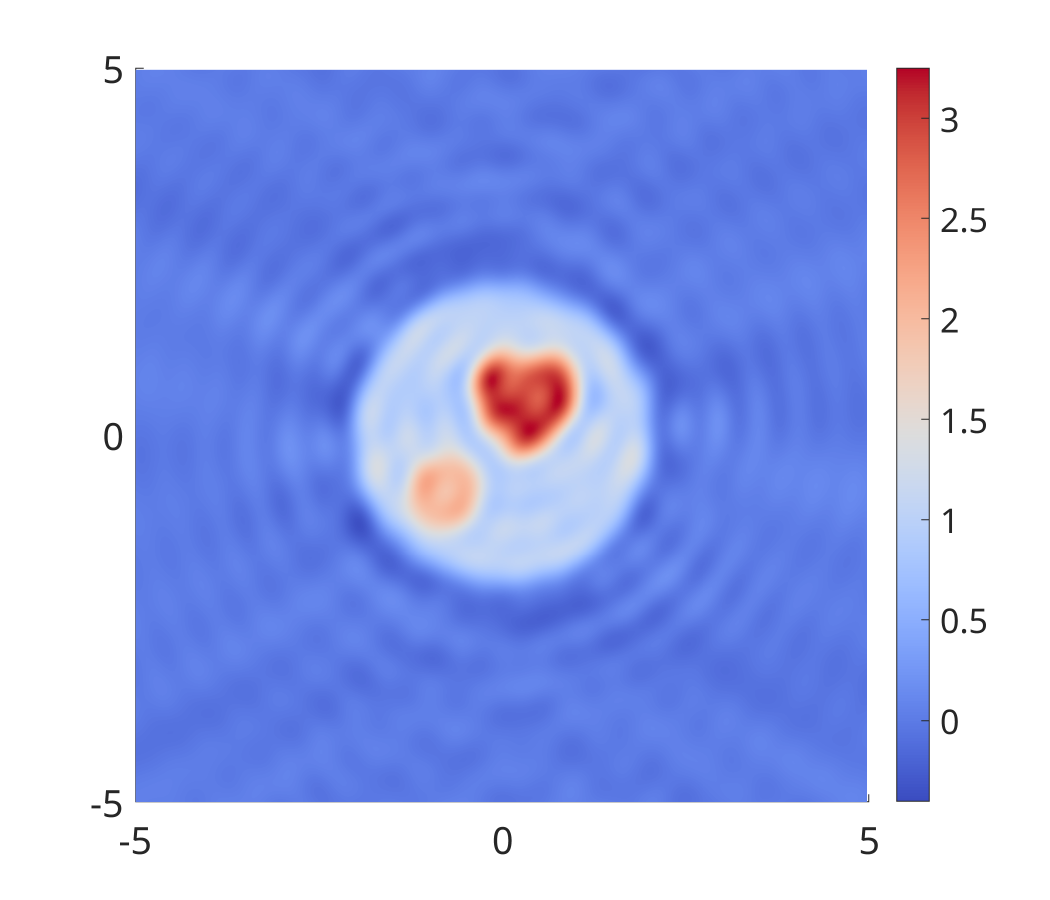}
		\caption{$\mathbf{f}^{\text{bp}}_{12}$ from noiseless data, \\ PSNR 28.24, RMSE 0.22, SSIM 0.22}
	\end{subfigure}\hfill
	\begin{subfigure}[t]{0.33\textwidth}
		\centering
		\includegraphics[scale=0.32]{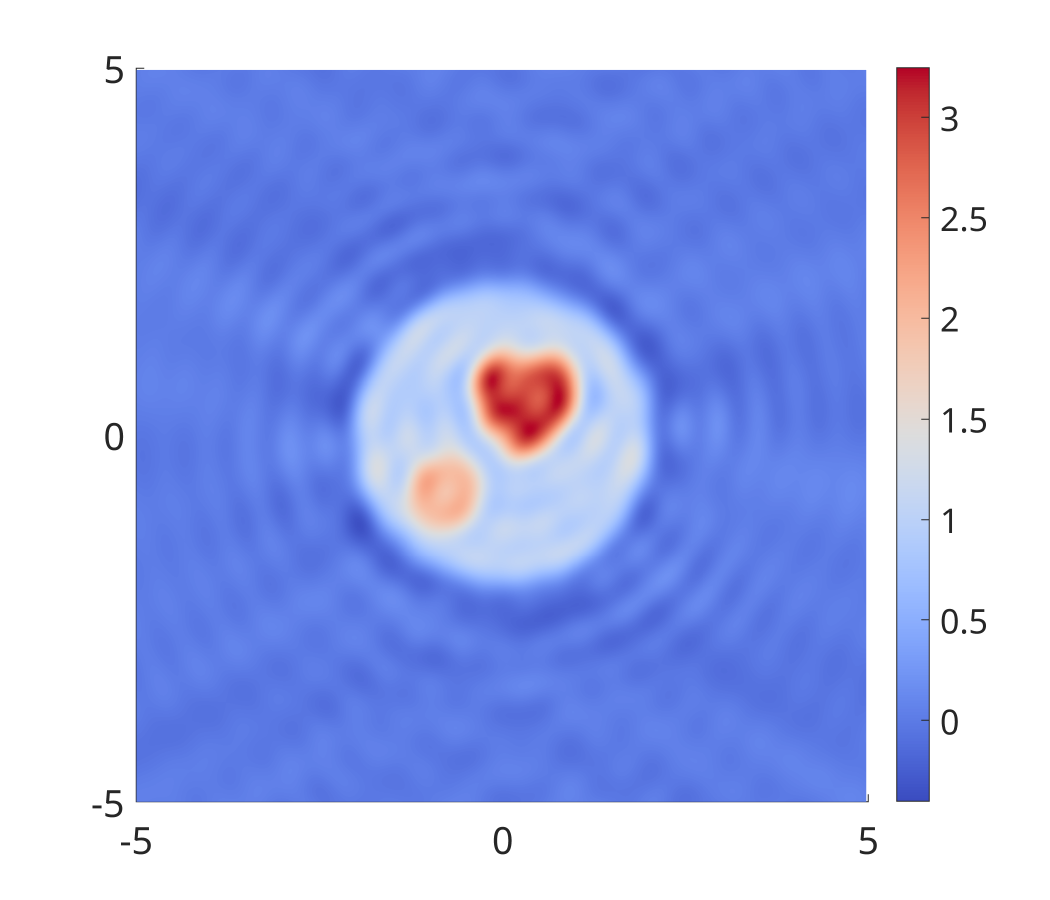}
		\caption{ $\mathbf{f}^{\text{bp}}_{12}$ from data with $1\%$ Gaussian noise,\\ PSNR 28.24, RMSE 0.22, SSIM 0.21}
	\end{subfigure}\hfill
	\begin{subfigure}[t]{0.33\textwidth}
		\centering
		\includegraphics[scale=0.32]{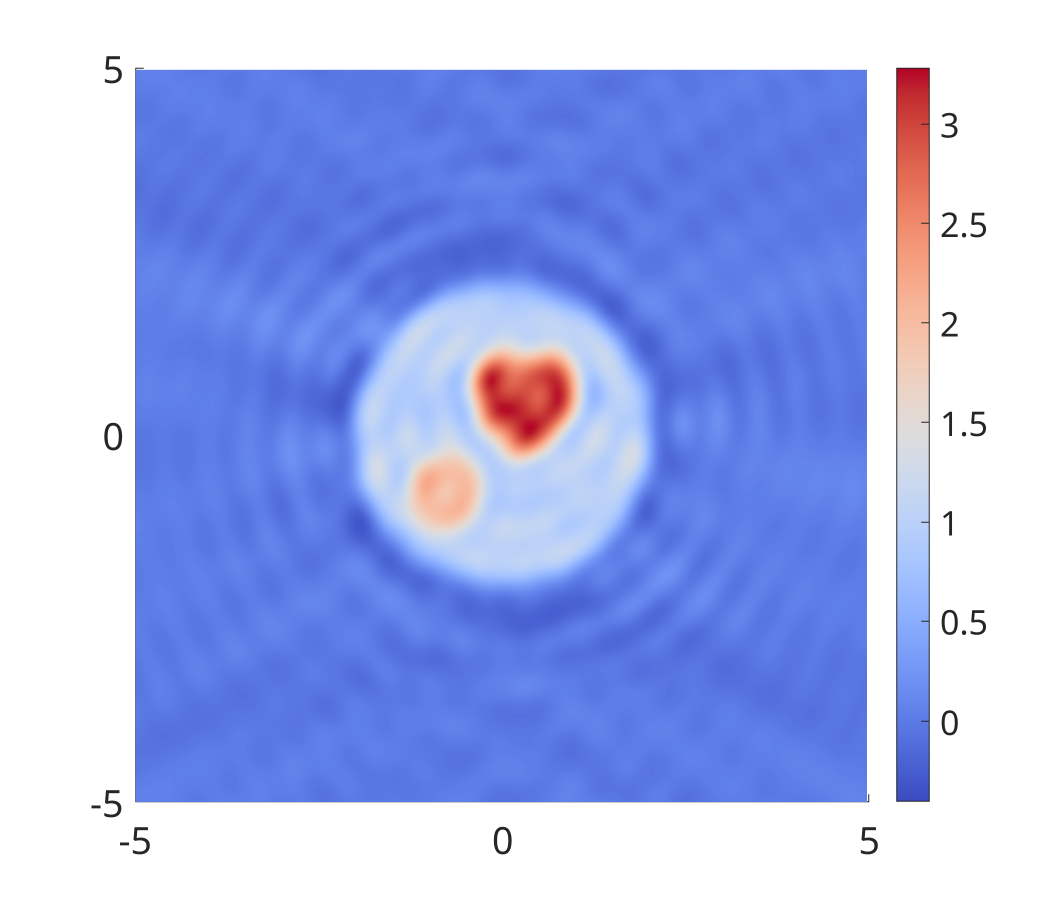}
		\caption{ $\mathbf{f}^{\text{bp}}_{12}$ from data with $5\%$ Gaussian noise,\\ PSNR 28.23, RMSE 0.23, SSIM 0.21}
	\end{subfigure}
	\caption{Reconstruction of the test sample $f$ shown in \autoref{fig:phantom}. For illumination the Gaussian profile from \autoref{eq:gaussian_a} with $A=80$ was used. The reconstruction is carried out using perturbed data with varying noise levels.}
	\label{fig:fbp_80}
\end{figure}
\begin{figure}[h]  
	\centering 
	\captionsetup[subfigure]{width=0.95\linewidth}
	\begin{subfigure}[t]{0.24\textwidth}
		\centering
		\includegraphics[scale=0.24]{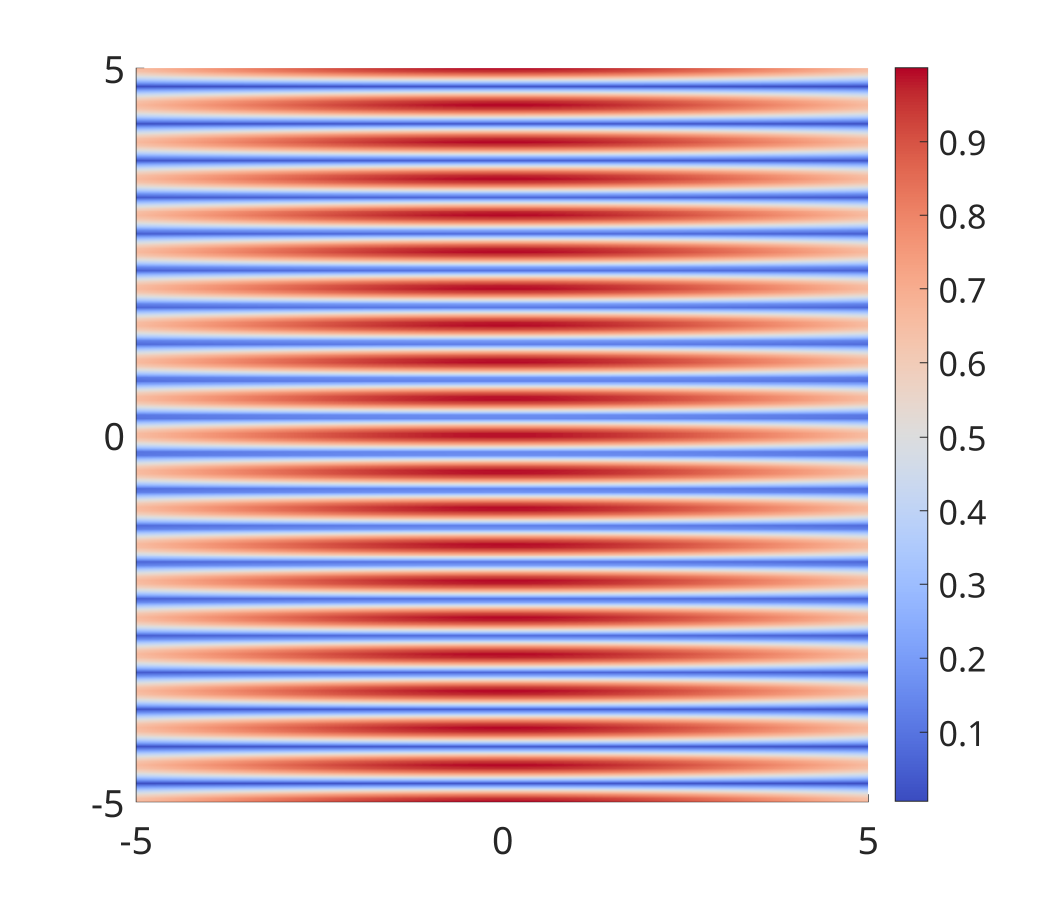}
		\caption{$A=600$}
	\end{subfigure}\hfill
	\begin{subfigure}[t]{0.24\textwidth}
		\centering
		\includegraphics[scale=0.24]{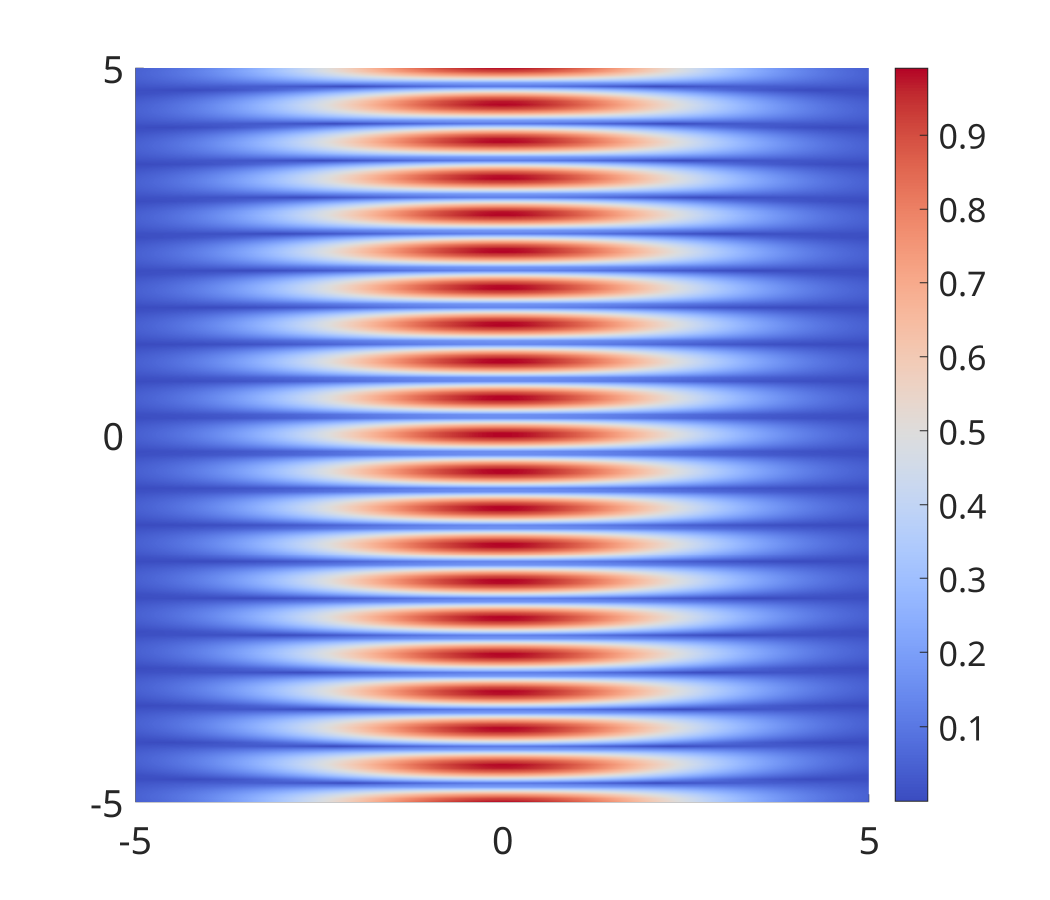}
		\caption{$A=80$}
	\end{subfigure}\hfill
	\begin{subfigure}[t]{0.24\textwidth}
		\centering
		\includegraphics[scale=0.24]{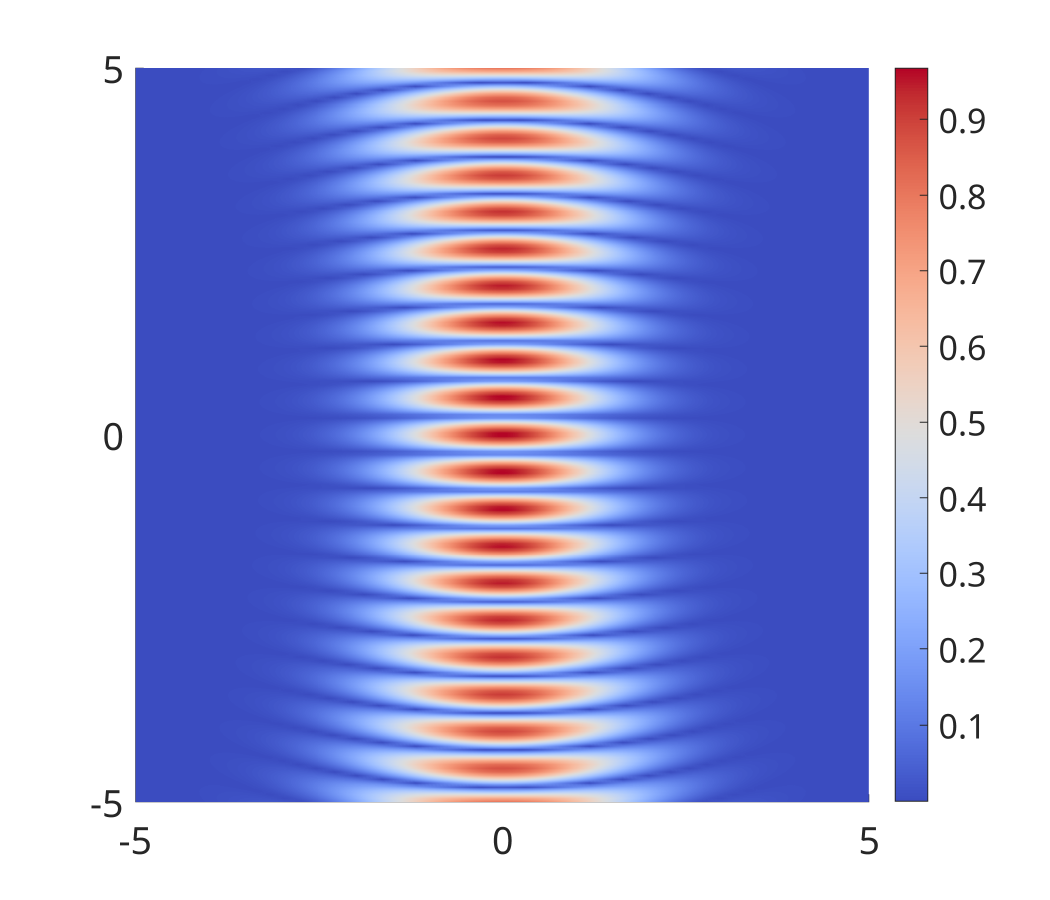}
		\caption{$A=20$}
	\end{subfigure}\hfill
	\begin{subfigure}[t]{0.24\textwidth}
		\centering
		\includegraphics[scale=0.24]{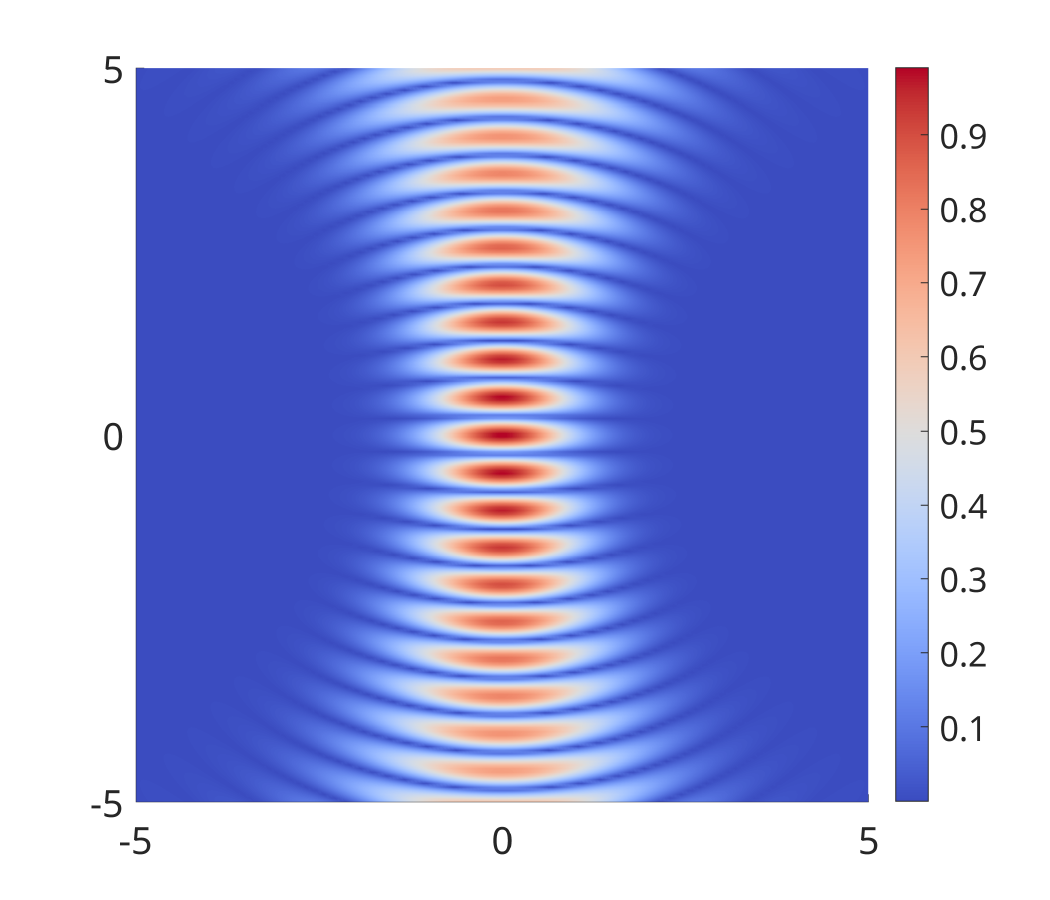}
		\caption{$A=10$}
	\end{subfigure}\\
	\begin{subfigure}[t]{0.24\textwidth}
		\centering
		\includegraphics[scale=0.24]{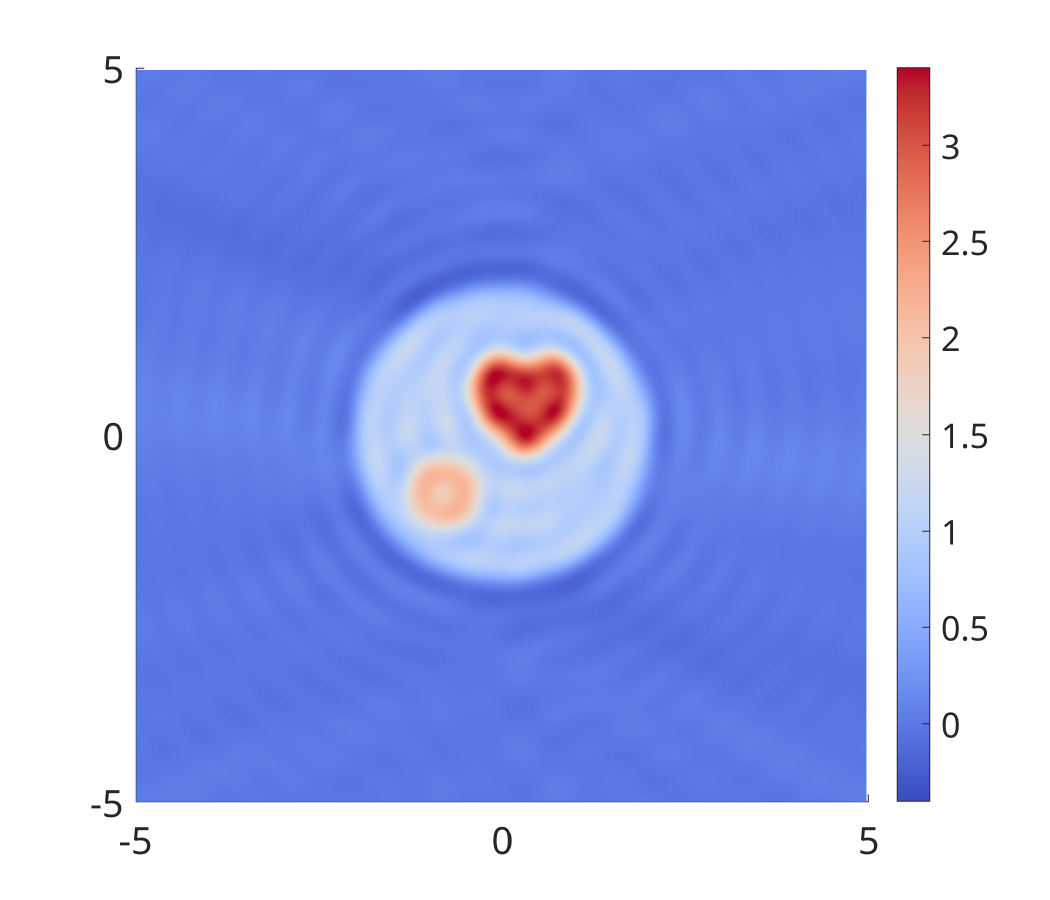}
		\caption{PSNR 29.50, SSIM 0.29, \\RMSE 0.19}
	\end{subfigure}\hfill
	\begin{subfigure}[t]{0.24\textwidth}
		\centering
		\includegraphics[scale=0.24]{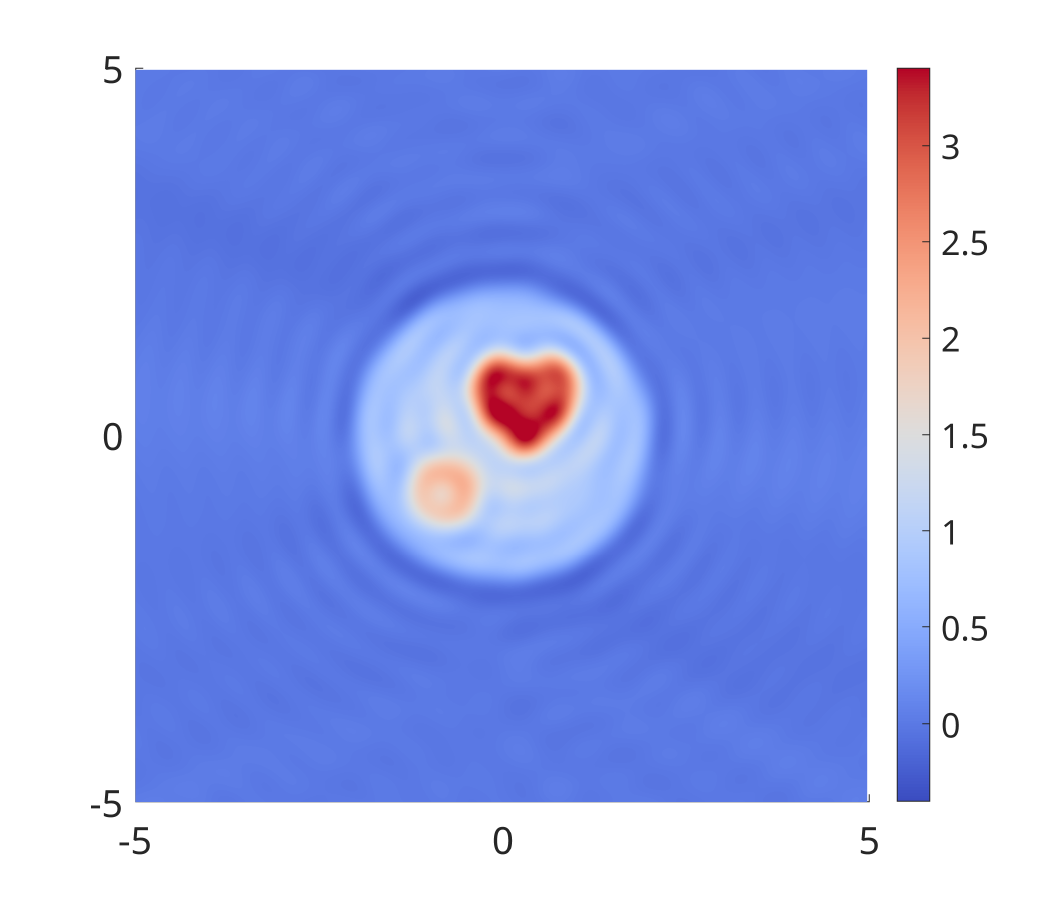}
		\caption{PSNR 28.15, SSIM 0.34, RMSE 0.23}
	\end{subfigure}\hfill
	\begin{subfigure}[t]{0.24\textwidth}
		\centering
		\includegraphics[scale=0.24]{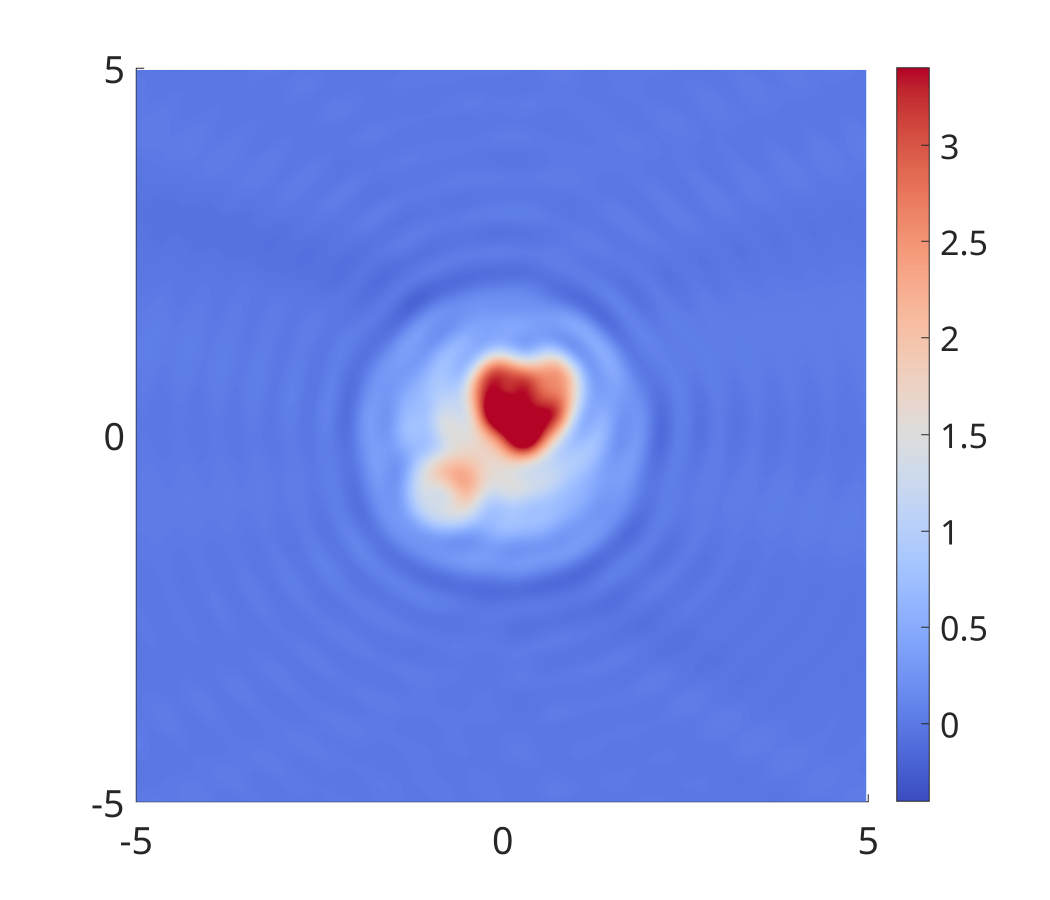}
		\caption{PSNR 22.45, SSIM 0.38, RMSE 0.44}
	\end{subfigure}\hfill
	\begin{subfigure}[t]{0.24\textwidth}
		\centering
		\includegraphics[scale=0.24]{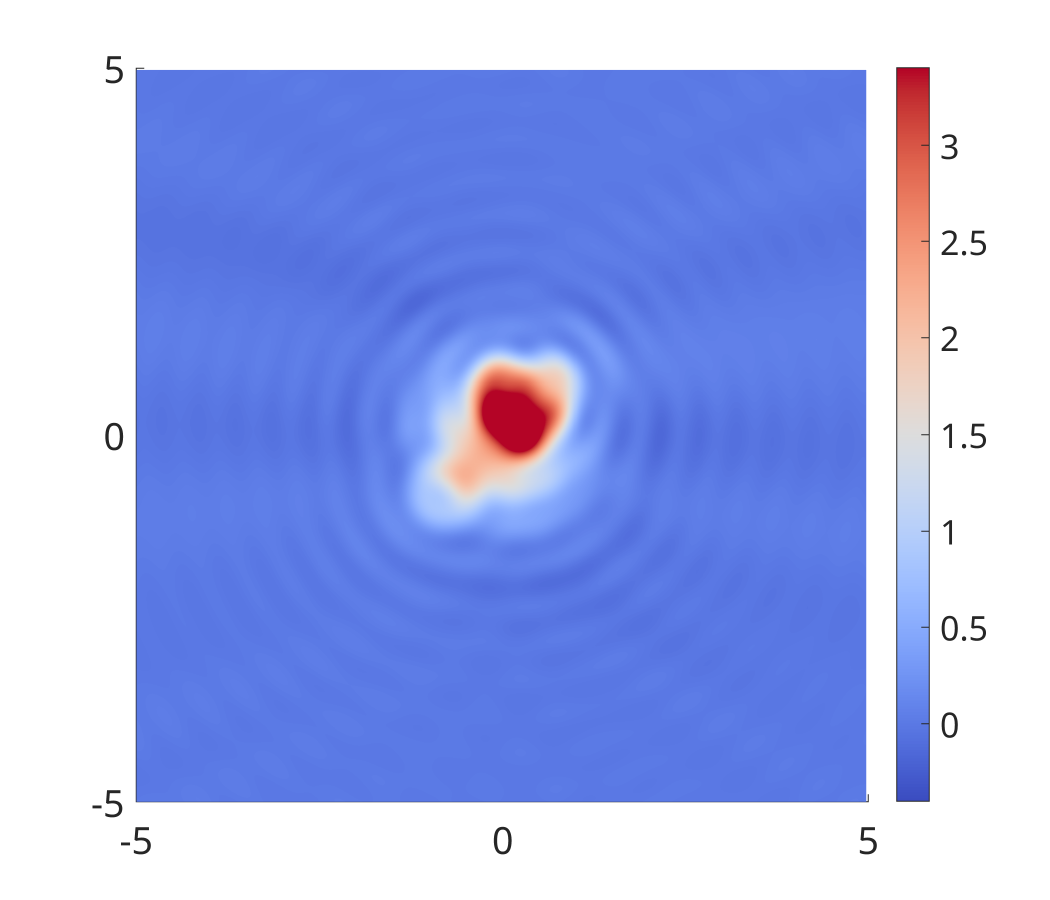}
		\caption{ PSNR 19.60, SSIM 0.34, RMSE 0.61}
	\end{subfigure}
	\caption{\textbf{Reconstruction with conventional DT.} Top: Illumination scenarios using the Gaussian beam profile from \autoref{eq:gaussian_a} for different values of $A$. Bottom: 
		Resulting reconstructions via conventional DT, i.e. without implementing TSVD as an intermediate inversion step.}
	\label{fig:comp}
\end{figure}
\subsection{Reconstruction with conventional DT}
In \autoref{sec:comparison} a comparison of the forward model of conventional DT and focused beam DT was provided. We observed that the measurements differ depending on the chosen imaging scenario and therefore developed a new reconstruction technique for DT that takes customized incident fields as opposite to only plane wave imaging into account.  At last, it is our intention to numerically quantify the impact of this new DT approach adapted to the illumination scenario. In doing so, we illuminate the test sample using a Gaussian beam and reconstruct with conventional DT, i.e. backpropagation without performing TSVD. 

The incident beam waves for different amplitudes $A$ together with the conventional DT solutions are presented in \autoref{fig:comp}. We observe that the reconstruction quality deteriorates as $A$ decreases, i.e. as the plane wave assumption is violated by focusing. In summary, when dealing with an illumination scenario that deviates from the plane wave assumption, such as (focused) beam illumination, it is essential to account for this in the forward model and, consequently, in the reconstruction.
	
\section{Conclusion}\label{sec:conclusion}
In this article, we studied the imaging problem of DT, considering arbitrary illumination scenarios modeled by a weighted superposition of monochromatic plane waves. Motivated by application in ultrasound tomography, we specifically introduced a focused beam as a representative imaging scenario that deviates from plane wave illumination. Based on the Born approximation we derived an adapted forward model to the new illumination concept. Then, we deduced from conventional DT an adapted Fourier diffraction theorem, serving as the basis for quantifying the scattering potential from the detected scattered waves. In contrast to the original theorem \cite{Wol69}, this adaption did not provide a direct relation between the measurements and the Fourier-transformed scattering potential. Therefore, the backpropagation of the scattering potential was divided into two parts: We first inverted the arising integral operator in this Fourier diffraction relation by calculating the TSVD. Subsequently, we extracted the scattering potential using a low-pass filtered Fourier inversion. 

In the practical implementation of this approach, we assessed its performance. It became evident that the choice of the beam profile plays a crucial role in determining the accuracy and stability of the first inversion step and hence of the backpropagation. In particular, our findings revealed that focusing is disadvantageous for the stability of the solution. Nonetheless, we demonstrated that when dealing with an imaging scenario that deviates too much from a plane wave, such as focusing, conventional DT reconstruction is not sufficient and that our approach is to be preferred.

\subsection*{Acknowledgements}
%
%
This research was funded in whole, or in part, by the Austrian Science Fund
(FWF) 10.55776/P34981 -- New Inverse Problems of Super-Resolved Microscopy (NIPSUM)
and SFB 10.55776/F68 ``Tomography Across the Scales'', project F6807-N36
(Tomography with Uncertainties). For open access purposes, the author has
applied a CC BY public copyright license to any author-accepted manuscript
version arising from this submission.
The financial support by the Austrian Federal Ministry for Digital and Economic
Affairs, the National Foundation for Research, Technology and Development and the Christian Doppler
Research Association is gratefully acknowledged.
The authors are grateful to Leopold Veselka for feedback and suggesting several improvements on the manuscript.

	\FloatBarrier
	\printbibliography

\end{document}